\documentclass[onefignum,onetabnum]{siamart190516}
\usepackage[linesnumberedhidden,algo2e,vlined,ruled]{algorithm2e}
\usepackage{algorithmic}

\usepackage[caption=false]{subfig}

\usepackage{lipsum}
\usepackage{amsfonts}
\usepackage{graphicx}
\usepackage{epstopdf}
\usepackage{tcolorbox}
\usepackage{booktabs,longtable,lscape,verbatim,enumitem} 

\usepackage{bm} %

\newsavebox\CBox
\def\textBF#1{\sbox\CBox{#1}\resizebox{\wd\CBox}{\ht\CBox}{\textbf{#1}}}

\newsiamremark{hypothesis}{Hypothesis}
\crefname{hypothesis}{Hypothesis}{Hypotheses}
\newsiamthm{claim}{Claim}
\newsiamthm{remark}{Remark}
\newsiamremark{dataset}{Dataset}
\crefname{dataset}{Dataset}{Datasets}
\crefname{subsection}{section}{sections}
\newsiamthm{assumption}{Assumption}
\newsiamthm{condition}{Oracle}
\newsiamthm{fact}{Fact}
\headers{ }{}

\title{A Riemannian exponential augmented Lagrangian method for computing the projection robust Wasserstein distance}

\author{Bo Jiang \thanks{Key Laboratory for NSLSCS of Jiangsu Province, School of Mathematical Sciences, Nanjing
Normal University, Nanjing, China
  (\email{jiangbo@njnu.edu.cn}).} 
  \and 
   Ya-Feng Liu \thanks{Corresponding author. State  Key  Laboratory of  Scientific and  Engineering  Computing,  Institute of  Computational Mathematics and Scientific/Engineering Computing, Academy of Mathematics and Systems Science, Chinese Academy of Sciences, Beijing, China (\email{yafliu@lsec.cc.ac.cn}).}
}

\usepackage{amsopn}

\DeclareMathOperator{\Diag}{Diag}
\DeclareMathOperator*{\argmin}{argmin}
\DeclareMathOperator{\argmax}{argmax}
\DeclareMathOperator{\grad}{grad}

\newcommand{\be}{\begin{equation}}
\newcommand{\ee}{\end{equation}}
\newcommand{\1}{\mathbf{1}}

\newcommand{\x}{\mathbf{x}}
\newcommand{\Rbb}{\mathbb{R}}
\newcommand{\Kcal}{\mathcal{K}}
\newcommand{\Gcal}{\mathcal{S}}

\newcommand{\Mcal}{\mathcal{M}}
\newcommand{\Tcal}{\mathrm{T}}
\newcommand{\Acal}{\mathcal{A}}
\newcommand{\Lcal}{\mathcal{L}}
\newcommand{\Rcal}{\mathcal{R}}

\newcommand{\Pcal}{\mathcal{P}}

\newcommand{\Fsf}{\mathsf{F}}
\newcommand{\KL}{\mathsf{KL}}
\newcommand{\Tsf}{\mathsf{T}}
\newcommand{\err}{\mathsf{e}}

\newcommand{\st}{\mathrm{s.t.}}
\newcommand{\nn}{\nonumber}

\newcommand{\rev}[1]{{\color{black}{#1}}}

\newcommand{\proj}{\mathsf{Proj}}

\newcommand{\sign}{\mathrm{sign}}
\newcommand{\var}{\mathrm{var}}
\newcommand{\Retr}{\mathsf{Retr}}

\newcommand{\iprod}[2]{ \left\langle #1, #2 \right\rangle}
\newcommand{\iprods}[2]{ \langle #1, #2 \rangle}

\newcommand{\Round}{\texttt{Round}}

\newcommand{\Ocal}{\mathcal{O}}

\newcommand{\iRBBS}{\texttt{iRBBS}}
\newcommand{\RABCD}{\texttt{RABCD}}
\newcommand{\RAGAS}{\texttt{RAGAS}}
\newcommand{\RAGASv}{\texttt{R(A)GAS}}
\newcommand{\RABCDv}{\texttt{R(A)BCD}}
\newcommand{\RBCD}{\texttt{RBCD}}
\newcommand{\RGAS}{\texttt{RGAS}}
\newcommand{\ReALM}{\texttt{ReALM}}
\newcommand{\iRGD}{\texttt{iRGD}}

\begin{document}

\maketitle

\begin{abstract}
Projecting the  distance measures onto a low-dimensional space is an efficient way of mitigating the curse of dimensionality in the classical Wasserstein distance using optimal transport. The obtained maximized distance is referred to as projection robust Wasserstein (PRW) distance. In this paper, we equivalently reformulate the computation of the PRW distance as an optimization problem over the Cartesian product of the Stiefel manifold and the Euclidean space with additional nonlinear inequality  constraints.  We propose a Riemannian exponential augmented Lagrangian method (\ReALM) with a global convergence guarantee to solve this problem.  Compared with the existing approaches, \ReALM~can potentially avoid too small penalty parameters.
Moreover, we propose a framework of inexact Riemannian gradient descent methods to solve the subproblems in \ReALM~efficiently.  In particular, by using the special structure of the subproblem, we give a practical algorithm named as the inexact Riemannian Barzilai-Borwein method with Sinkhorn iteration (\iRBBS). The remarkable features of \iRBBS~lie in that it performs a flexible number of Sinkhorn iterations to compute an inexact gradient with respect to the projection matrix of the problem and adopts the Barzilai-Borwein stepsize based on the inexact gradient information to improve the performance.  We show that \iRBBS~can return an $\epsilon$-stationary point of the original PRW distance problem within $\Ocal(\epsilon^{-3})$ iterations. Extensive numerical results on synthetic and real datasets demonstrate that our proposed \ReALM~as well as \iRBBS~outperform the state-of-the-art solvers for computing the PRW distance.
\end{abstract}

\begin{keywords}
Barzilai–Borwein method, exponential augmented Lagrangian,   inexact gradient, Stiefel manifold,  Sinkhorn  iteration, Wasserstein distance
 \end{keywords}

\begin{AMS}
	65K10, 90C26, 90C47
\end{AMS}

\section{Introduction}
The optimal transport (OT) problem has wide applications in machine learning, data sciences, and image sciences; see  \cite{peyre2019computational} and the references therein for more details. However, its direct application in machine learning may encounter the curse of dimensionality issue since the sample complexity of approximating the Wasserstein distance can grow exponentially in dimension \cite{fournier2015rate,weed2019sharp}. To resolve this issue, by making an important extension to the sliced Wasserstein distance,  Paty and  Cuture \cite{paty2019subspace} and Niles-Wee and Rigollet \cite{niles2019estimation}  proposed to project the distributions to a $k$-dimensional subspace that maximizes the Wasserstein distance, which can reduce the sample complexity and overcome the issue of the curse of dimensionality \cite{niles2019estimation,lin2021projection}. 

In this paper, we focus on the discrete probability measures case.  For $\{x_1, \ldots, x_n\}$\\$\subset \Rbb^d$ and $\{y_1, \ldots, y_n\} \subset \Rbb^d$,  define $\rev{M_{ij}} = (x_i - y_j)(x_i - y_j)^{\Tsf}$ for each $(i,j)$. Let  $\1 \in \Rbb^n$ be the all-one vector and $\delta_x$ be the Dirac delta function at $x$. Consider two discrete probability measures $\mu_n = \sum_{i = 1}^n r_i \delta_{x_i}$ and $\nu_n = \sum_{i=1}^n c_i \delta_{y_i}$,  where  $r,c\in \Delta^n:= \{z\in\Rbb^n \mid \1^\Tsf z= 1, z > 0\}$.   For the integer $1 \leq k \leq d$, the  $k$-dimensional projection robust Wasserstein  (PRW) distance  between $\mu_n$ and $\nu_n$ is defined as   \cite{paty2019subspace} 
\be \label{equ:PRW}
\Pcal_k^2(\mu_n,\nu_n)= \max_{U \in \Gcal} \min_{\pi \in \Pi(r, c)} f(\pi, U),
\ee
where $f(\pi, U)= \iprod{\pi}{C(U)}$  with $C(U)\in \Rbb^{n \times n}$ and $[C(U)]_{ij} = \iprod{M_{ij}}{UU^{\Tsf}}$. Throughout this paper, $\Gcal = \{U \in \Rbb^{d\times k}\mid U^{\Tsf} U = I_k\}$ is known as the Stiefel manifold with $I_k$ being the $k\times k$ identity matrix and  $\Pi(r, c) = \{\pi \in \Rbb^{n \times n} \mid \pi \1 = r, \pi^{\Tsf} \1 = c, \pi \geq 0\}$. Problem \eqref{equ:PRW} is a nonconcave-convex max-min problem over the Stiefel manifold\footnote{This problem can also be seen as a max-min problem over the Grassmann manifold, which can be viewed as a  quotient manifold of $\Gcal$; see \cite{edelman1998geometry} for more details.}, which makes it very challenging to solve.
\subsection{Related works}
To compute the PRW distance, \cite{paty2019subspace} considered  a convex relaxation (but without the theoretical guarantee on the duality gap), wherein an OT
or entropy-regularized OT subproblem with dimension $n$ needs to be solved exactly in each step.  Very recently, 
 Lin et al. \cite{lin2020projection}  viewed problem \eqref{equ:PRW}  as a single-level optimization problem over the Stiefel manifold: 
\be \label{equ:PRW:2}
\max_{U \in \Gcal}\,  p(U),
\ee
where $p(U) = \min_{\pi \in \Pi(r, c)} f(\pi, U)$.  
Observe that $p(U)$ is a weakly convex \cite[Lemma 2.2]{lin2020projection} but nonsmooth function. Then it is possible to use the Riemannian subgradient method developed in \cite{li2021weakly} to solve \eqref{equ:PRW:2}. However, this may still be difficult since computing a  subgradient of  $p(\cdot)$ at $U$ needs to solve an OT problem exactly.  Instead of  solving  \eqref{equ:PRW:2} directly,  Lin et al. \cite{lin2020projection}  considered solving the following regularization problem with a  small regularization parameter $\eta$: 
\be \label{equ:PRW:reg}
\max_{U \in \Gcal}\, p_{\eta}(U),
\ee
where $p_{\eta}(U) = \min_{\pi \in \Pi(r, c)} f(\pi, U) - \eta H(\pi)$,
in which $H(\pi) =  - \sum_{ij} \pi_{ij} \log \pi_{ij}$ is the entropy function. 
To solve problem \eqref{equ:PRW:reg}, Lin et al. \cite{lin2020projection}  proposed a Riemannian (adaptive) gradient ascent with Sinkhorn (\RAGASv) algorithm,  which can be understood as a Riemannian gradient ascent method with an inexact gradient at a fixed inexactness level. 
They showed that  \RAGASv~can return an $\epsilon$-stationary point (see Definition 2.7 in \cite{lin2020projection}) of  PRW \eqref{equ:PRW} in $\Ocal(\epsilon^{-4})$ iterations if $\eta = \Ocal(\epsilon)$ in \eqref{equ:PRW:reg}. However, in each iteration, \RAGASv~needs to calculate $\nabla p_{\eta}(U)$ via solving a regularized OT problem in relatively high precision, which results in a high computational cost.  

 To further reduce the complexity of \RAGASv, rather than solving \eqref{equ:PRW:reg} directly, Huang et al. \cite{huang2021riemannian} focused on an equivalent ``min'' formulation of \eqref{equ:PRW:reg}  
 \be \label{prob:PRW:max}
\min_{U\in\Gcal, \alpha, \beta \in \Rbb^n}  h_{\eta}(\alpha,\beta,U)
\ee 
 via replacing   $p_{\eta}(U)$ by its dual, that is  
 $p_{\eta}(U) = \max_{\alpha,\beta\in \Rbb^n} - h_{\eta}(\alpha,\beta,U),$
where 
\[
h_{\eta}(\alpha,\beta,U) =  r^{\Tsf} \alpha +  c^{\Tsf} \beta + \eta \log \sum\nolimits_{ij} \exp\left( -\frac{\alpha_i + \beta_j +  \langle M_{ij}, UU^{\Tsf}\rangle}{\eta}\right).
\]
Note that the last term in $h_{\eta}(\alpha,\beta, U)$ is also known as the log-exponential aggregation function \cite{li1991aggregate}.
Huang et al.  \cite{huang2021riemannian} proposed two efficient algorithms, named Riemannian (adaptive) block coordinate descent (\RABCDv) algorithms to solve \eqref{prob:PRW:max}.  In their algorithms, $\alpha$ and $\beta$ are updated by one Sinkhorn iteration, while $U$ is updated by a Riemannian gradient descent step with a fixed stepsize. By choosing $\eta = \Ocal(\epsilon)$ in \eqref{prob:PRW:max},  
they showed that the whole  iteration complexity of \RABCDv~to attain an $\epsilon$-stationary point (see \cite[Definition 4.1]{huang2021riemannian} wherein $\epsilon_1=\epsilon_2 = \epsilon$) of PRW \eqref{equ:PRW}  is reduced to $\Ocal(\epsilon^{-3})$,
which significantly improves the complexity of \RAGASv~\cite{huang2021riemannian}. 

However, as pointed out in \cite[Remark 6.1]{huang2021riemannian}, there are two main issues of \RABCDv\\and \RAGASv.   
First, the algorithms are sensitive to the choice of parameter $\eta$. More specifically,  to compute a solution with relatively high \rev{quality}, $\eta$ has to be chosen to be small, which may cause numerical instability. Second, the performance of the algorithms is sensitive to the stepsizes in updating $U$. Hence,  to achieve a better performance, one has to spend some efforts to tune the stepsizes carefully. Resolving the two main issues of \RABCDv~demands some novel approach from both theoretical and computational points of view, and this is the focus of our paper. 

To handle the first issue mentioned above,  we first observe that the approach proposed by Huang et al. \cite{huang2021riemannian} can be understood as a Riemannian exponential penalty method applied to an equivalent formulation \eqref{prob:manifold:nlp:2}  of \eqref{equ:PRW:2}; see \cref{remark:relation:Huang} later.  
On the other hand, the exponential augmented Lagrangian methods (ALM), which are usually more stable than the exponential penalty type of methods,   have been well-studied \cite{bertsekas2014constrained,tseng1993convergence}  and widely used to solve various problems, such as optimal transport \cite{yang2022bregman}, semidefinite programming \cite{doljansky1999interior}, equilibrium problems \cite{torrealba2020exponential}, etc.  Hence, {\it our idea for resolving the first issue is to extend the exponential ALM from the Euclidean case to the Riemannian case.}  It should be remarked that for the fixed subspace $U$ case, such an idea has been adopted in computing the standard OT problem, 
wherein a proximal point-based entropy subproblem with a positive constant $\eta$ needs to be solved inexactly in each iteration; see  \cite{xie2020fast,yang2022bregman} for more details.   

 Problem \eqref{prob:PRW:max}  can also be viewed as a single-level optimization over the Stiefel manifold by letting $h_{\eta}(U)= \min_{\alpha,\beta\in \Rbb^n} h_{\eta}(\alpha,\beta,U)$.  Our idea for resolving the second issue lies in using an inexact Riemannian gradient descent method to solve $\min_{U \in \Gcal} h_{\eta}(U)$, wherein the Riemannian gradient $\grad h_{\eta}(U)$ is computed 
 inexactly at an adaptive inexactness level.  Although the inexact gradient descent method has been well explored in the Euclidean case \cite{carter1991global,devolder2014first,berahas2021global}, to our best knowledge,  there are no results on how to choose the stepsize adaptively for general nonlinear objective functions. One exception is  \cite{hu2007inexact}, where for the strongly convex quadratic objective function, Hu and Dai proposed an inexact Barzilai-Borwein (BB) method and established the R-linear convergence of the inexact BB method if the inexactitude of the approximate gradient is well controlled.   However, the method in \cite{hu2007inexact}  can not be directly extended to solve problem  \eqref{prob:PRW:max} due to the lack of the appropriate line search condition.  Motivated by Hu and Dai's method and successful extensions of the BB method \cite{barzilai1988two} to the Riemannian case \cite{wen2013feasible,jiang2015framework,gao2018new,iannazzo2018riemannian},  our solution to the second issue is to {\it develop an inexact Riemannian gradient descent method with the \rev{BB} stepsize and an appropriate nonmonotone line search condition}.
 
\subsection{Our contributions} In this paper,  by reformulating \eqref{equ:PRW} as an optimization problem defined over the Cartesian product of the Stiefel manifold and the Euclidean space with additional inequality constraints (see \eqref{prob:manifold:nlp:2} further ahead), we design a Riemannian exponential ALM (\ReALM) with customized methods for solving the subproblems to efficiently and faithfully compute the PRW distance. Our main contributions are summarized as follows. 
\begin{itemize}[leftmargin=.8cm]
\item We propose a \ReALM~method for computing the PRW distance,  in which a series of subproblems with dynamically decreasing penalty parameters and adaptively updated multiplier matrices is solved approximately. In theory, we establish the global convergence of \ReALM~in the sense that any limit point of the sequence generated by the algorithm is a stationary point of the original problem. Numerically, \ReALM~always outperforms the Riemannian exponential penalty approach, which solves a series of subproblems with dynamically decreasing penalty parameters or with a sufficiently small penalty parameter. 
In particular, our method can avoid too small penalty parameters in some cases compared with the Riemannian exponential penalty approach.
\item  To efficiently solve the subproblem  \eqref{prob:PRW:max} or \eqref{prob:sub:RALMexp:0:1}, we propose a  framework of inexact Riemannian gradient descent  (\iRGD) methods with convergence and complexity guarantees. Particularly, we give a practical algorithm, namely, the inexact Riemannian \rev{BB} method with Sinkhorn iteration (\iRBBS), wherein a flexible number of  Sinkhorn iterations is performed to compute an inexact gradient with respect to  the projection matrix.   Compared with \RABCDv, our proposed \iRBBS~can not only return a stronger $(\epsilon_1, \epsilon_2)$-stationary point of PRW \eqref{equ:PRW}, compared with the definitions in \cite{lin2020projection,huang2021riemannian}, in $\Ocal(\max\{\epsilon_1^{-2}, \epsilon_2^{-2}\}\epsilon_2^{-1})$ iterations (see \cref{prop:complexity:prw} and \cref{remark:kkt:strong}),  but also has a better numerical performance, which mainly benefits from the inexact BB stepsize.  As a by-product, we also establish the complexity results of the original Sinkhorn iteration in a simple way. 
\end{itemize}
\subsection{Notations and preliminaries}
 For a scalar $a$, let $\lceil a \rceil$ be the smallest nonnegative integer larger than $a$. 
Denote $[n]=\{1, 2, \ldots, n\}$. For a matrix $A$,  denote $A_{\min} = \min_{ij} A_{ij}$.  Denote by $\|A\|_{\Fsf}$ the Frobenius norm of $A$. The $\ell_\infty$-norm of   $A$ is  $\|A\|_{\infty} = \max_{ij} |A|_{ij}$, while the $\ell_1$-norm is $\|A\|_1 = \sum_{ij}|A_{ij}|$.    The variation seminorm of $A$ is   $\|A\|_{\var}= \max_{ij} A_{ij} - \min_{ij} A_{ij}$.  
The notation $\log A$ denotes  a matrix of the same size as $A$, and $(\log A)_{ij} = \log A_{ij}$. 
For a vector $x \in \Rbb^n$, $\Diag(x)$ is an $n \times n$ diagonal matrix with $x$ being its main diagonal. 
We use $\Rbb_+^{n \times n}$ and $\Rbb_{++}^{n \times n}$ to denote the nonnegative and positive orthants of $\Rbb^{n \times n}$, respectively. 
The entropy function is given as $H(\pi) = -\sum_{ij} \pi_{ij} \log \pi_{ij}$ for $\pi \in \Rbb^{n \times n}_+$ and $\|\pi\|_1 = 1$ with the convention that $H(\pi) = -\infty$  if one of the entries $\pi_{ij} = 0$.   For $p,q \in \Delta^n$, the Kullbbback-Leibler divergence between $p$ and $q$ is $\KL(p\| q) = \sum_{i} p_i \log (p_i/q_i)$. Throughout this paper,   we define the matrix $C\in \Rbb^{n \times n}$ with $C_{ij} = \|x_i - y_j\|^2$ and  define $V_{\pi} = \sum_{ij} \pi_{ij}M_{ij}$.

The Riemannian metric endowed on the Stiefel manifold $\Gcal$  is taken as the usual metric on $\Rbb^{d \times k}$.  The tangent space at $U \in \Gcal$ is $\Tcal_U\Gcal = \{\xi\in\Rbb^{d \times k} \mid U^{\Tsf}\xi +\xi^{\Tsf} U = 0\}$.  Let $\Tcal \Gcal = \{(U, \xi) \mid U \in \Gcal \mbox{ and } \xi \in \Tcal_U \Gcal\}$ be the tangent bundle of $\Gcal$.  A smooth map $\Retr: \Tcal \Gcal \to \Gcal : (U, \xi) \mapsto \Retr_U(\xi)$ is called a retraction if each curve $\Rcal(t)= \Retr_U(t\xi)$ satisfies $\Rcal(0) = U$ and $\Rcal'(0) = \xi$; see \cite[Definition 3.47]{boumal2020introduction} or \cite[Definition 4.1.1]{absil2009optimization}. 
 The retraction on the Stiefel manifold has the following nice properties; see \cite{liu2019quadratic,boumal2019global} for instance.  
\begin{fact} \label{lemma:retraction:property}
There exist  positive constants $L_1$ and $L_2$ such that 
$\|\Retr_U(\xi) - U\|_{\Fsf} \leq L_1 \|\xi\|_{\Fsf}$ and $\|\Retr_U(\xi) - (U + \xi)\|_{\Fsf} \leq L_2 \|\xi\|_{\Fsf}^2$  
hold for all $\xi \in \Tcal_U\Gcal$ and $U \in \Gcal$.  
\end{fact}

The Riemannian gradient of a smooth function $f: \Rbb^{d \times k} \to \Rbb$ at $U$ is defined as  $\grad f(U)$, which satisfies  $\iprod{\grad f(U)}{\xi} = \iprod{\nabla f(U)}{\xi}$ for all $\xi\in \Tcal_U\Gcal$.  If $U^{\Tsf}\nabla f(U)$ is symmetric, we have 
$\grad f(U) =  \proj_{\Tcal_U\Gcal}(\nabla f(U)) = (I_d - UU^{\Tsf})\nabla f(U).$

\subsection{Organization}
 The rest of this paper is organized as follows.    The proposed \ReALM~method is introduced in \cref{sec:ReALM}.  A framework of inexact RGD methods and a practical \iRBBS~for solving the subproblem in \ReALM~are proposed in  \cref{sec:subproblem} and \cref{section:iRBBSs}, respectively.   Numerical results are presented in \cref{sec:experiments}.  Finally, we draw some concluding remarks in \cref{sec:conclusions}.
 
\section{A \ReALM~for computing the PRW distance  \eqref{equ:PRW}} \label{sec:ReALM}
In this section, we first give a reformulation of PRW distance problem \eqref{equ:PRW} in \cref{subsection:ReALM:reformulation} and then propose a Riemannian exponential ALM, summarized as \cref{alg:ReALM}, for solving the reformulation problem in \cref{subsection:ReALM}.  The global convergence of \cref{alg:ReALM} is discussed in \cref{subsection:ReALM:convergence}. 
\subsection{Reformulation of PRW distance problem \eqref{equ:PRW}}   \label{subsection:ReALM:reformulation}
Motivated by the dual approach in \cite{huang2021riemannian}, we first rewrite \eqref{equ:PRW} as an optimization problem over the Cartesian product of the Stiefel manifold and the Euclidean space with additional nonlinear inequality constraints. 

Given a fixed $U \in \Gcal$,   consider the OT  problem 
\be \label{prob:OT} 
\min_{\pi\in\Rbb^{n \times n}}   \iprod{\pi}{C(U)} \quad \st \quad  \pi \1 = r, ~~\pi^{\Tsf} \1 = c,~~ \pi \geq 0. 
\ee
Let $\hat \Lcal(\pi, \alpha, \beta) = \iprod{\pi}{C(U)}  + \iprods{\alpha}{\pi \1 - r} + \iprods{\beta}{\pi^{\Tsf}\1 - c} $ be the Lagrangian function with $\alpha \in \Rbb^n$ and $\beta  \in \Rbb^n$ being the  Lagrange multipliers corresponding to the constraints $\pi \1 = r$ and $\pi^{\Tsf} \1 = c$, respectively.  As done in \cite{huang2021riemannian,lin2022efficiency} for deriving the dual of the
 entropy-regularized OT problem,  we add a redundant constraint $\|\pi\|_1= 1$ and  
  derive the dual problem   of \eqref{prob:OT}  as 
 $\max_{\alpha\in\Rbb^n, \beta\in \Rbb^n} \min_{\pi\geq 0, \|\pi\|_1=1}  \hat \Lcal(\pi, \alpha, \beta),$
that is, 
\[ 
\max\limits_{\alpha\in \Rbb^n, \beta \in \Rbb^n}  - h(\x),
\]
where  $\x = (\alpha, \beta, U)$
 and 
 \be\label{equ:hx:def} 
 h(\x) = r^{\Tsf} \alpha   +  c ^{\Tsf} \beta - \varphi(\x)_{\min} \quad\mbox{with}\quad  \varphi_{ij}(\x) =  \alpha_i + \beta_j +  \langle \rev{M_{ij}}, UU^{\Tsf}\rangle.
 \ee
Therefore, the PRW distance  between $\mu_n$ and $\nu_n$  defined in \eqref{equ:PRW} is equal to 
\be \label{prob:manifold:nlp}
\Pcal_k^2(\mu_n,\nu_n) = - \min\limits_{\x \in \Mcal} h(\x),
\ee
where $\Mcal = \Rbb^n \times \Rbb^n \times \Gcal.$
The corresponding minimization problem in \eqref{prob:manifold:nlp}  can be  reformulated as the minimization problem as
\be  \label{prob:manifold:nlp:2} 
\min\limits_{\x \in \Mcal, y \in \Rbb} \ r^{\Tsf} \alpha   +  c ^{\Tsf} \beta  + y \quad 
\st \quad \varphi_{ij}(\x)  + y \geq 0, \quad  \forall~i, j \in [n].
\ee

In this paper, we shall focus on the formulation \eqref{prob:manifold:nlp:2}, whose first-order necessary conditions are established as follows. 
\begin{lemma}[First-order necessary conditions]\label{lem:1stopt}
Given $\bar \x \in \Mcal$ and $\bar y\in\Rbb$, suppose that $(\bar \x, \bar y)$  is a local minimizer of problem \eqref{prob:manifold:nlp:2}, then $(\bar \x, -\varphi(\bar \x)_{\min})$ is a stationary point of problem \eqref{prob:manifold:nlp:2}, namely,  there exists $\bar \pi \in \Pi(r,c)$ such that 
\be \label{equ:kkt}
 \proj_{\Tcal_{\bar U}\Gcal} \left(  -2V_{\bar \pi}   \bar U \right)  =  0, \quad \iprod{\bar \pi}{Z(\bar  \x)} = 0,
 \ee
where $Z(\bar \x) \in \Rbb^{n \times n}$ with 
\be \label{equ:Z}
[Z(\bar \x) ]_{ij}= \varphi_{ij}(\bar \x) - \varphi(\bar \x)_{\min}.
\ee
\end{lemma}
\begin{proof}
Since $(\bar \x, \bar y)$ is a local minimizer of problem \eqref{prob:manifold:nlp:2},  there must hold $\bar y = - \varphi(\bar \x)_{\min} = \max_{ij} \{- \varphi_{ij}(\bar \x)\}$. Moreover, such $\bar \x$ is also a local minimizer of problem  
$\min_{\x \in \Mcal} h(\x)$, where $h(\x)$ is given in \eqref{equ:hx:def}. 
For fixed $\alpha$ and $\beta$,  it is easy to see that $-\varphi_{ij}(\x) + \|C\|_{\infty} \|U\|_{\Fsf}^2$ is convex with respect to $U$. We thus know that  the function 
$- \varphi(\bar \x)_{\min}  + \|C\|_{\infty} \|U\|_{\Fsf}^2 =   \max_{ij} \left\{-\varphi_{ij}(\x) + \|C\|_{\infty} \|U\|_{\Fsf}^2 \right\}$
is convex with respect to $U$, which means that the function $- \varphi(\bar \x)_{\min}$ and thus $h(\x)$ is $\|C\|_{\infty}$-weakly convex with respect to $U$ \cite[Proposition 4.3]{vial1983strong}. 
  Let $\Acal(\bar\x) = \{(i,j)\in [n] \times [n] \mid \varphi_{ij}(\bar \x) = \varphi(\bar \x)_{\min}\}$. 
By \cite[Proposition 4.6]{vial1983strong}, we have $\partial h(\bar \x) = (r, c, 0) + \mathrm{conv} \left\{(-e_i, -e_j,  -2 M_{ij} \bar U) \mid  (i,j) \in \Acal(\bar \x) \right\}$, \rev{where $e_i$ is the $i$-th standard unit vector in $\Rbb^n$.}
Moreover, by \cite[Theorem 4.1]{yang2014optimality}, \cite[Theorem 5.1]{yang2014optimality}, and $\Tcal_{\bar \x}\Mcal =\Rbb^n \times \Rbb^n \times \Tcal_{\bar \x}\Gcal$, there must hold that 
 $0 \in  \proj_{\Tcal_{\bar \x}\Mcal} \partial h(\bar \x)$.
Putting all the above things together  shows that there exists $\bar \pi\in \Rbb_+^{n \times n}$ with 
$\bar \pi_{ij} = 0$  for all  $(i,j) \not \in \Acal(\bar \x)$ and 
$\sum_{(i,j) \in \Acal(\bar \x)} \bar \pi_{i,j} = 1$ 
  such that $\bar \pi \in \Pi(r,c)$ and \eqref{equ:kkt} hold.  The proof is completed. 
\end{proof}

Based on \cref{lem:1stopt}, we define the following approximate stationary point of problem \eqref{prob:manifold:nlp:2},  \rev{wherein} we require that the multiplier matrix $\tilde \pi$ lies in $\Pi(r,c)$.
\begin{definition}\label{def:epsilon:kkt:ours}
Given $\tilde \x \in \Mcal$, we say $(\tilde \x, -\varphi(\tilde \x)_{\min})$ an $(\epsilon_1, \epsilon_2)$-stationary point of problem \eqref{prob:manifold:nlp:2} (equivalent to PRW \eqref{equ:PRW}), if there exists  $\tilde \pi \in \Pi(r,c)$ such that  
$\big\|\proj_{\Tcal_{\tilde U}\Gcal} \big(-2 V_{\tilde \pi}  \tilde U \big)\big\|_{\Fsf}  \leq   \epsilon_1$ and $\iprod{\tilde \pi}{Z(\tilde  \x)} \leq \epsilon_2.$
\end{definition}
\begin{remark}\label{remark:kkt:strong}
Our \cref{def:epsilon:kkt:ours}  of the $(\epsilon_1, \epsilon_2)$-stationary point is stronger than the corresponding point in \cite[Definition 4.1]{huang2021riemannian} and  \cite[Definition 2.7]{lin2020projection} in the sense that the point satisfying the conditions in our definition also satisfies all conditions in \cite{huang2021riemannian} and in \cite{lin2020projection}. (By \cite[Section A]{huang2021riemannian},  the point  \rev{in \cite[Definition 4.1]{huang2021riemannian}} is stronger than that in  \cite[Definition 2.7]{lin2020projection}). To show that, we only need to verify that $\iprod{\tilde \pi}{Z(\tilde  \x)} \leq \epsilon_2$ can imply $f(\tilde \pi, \tilde U) - p(\tilde U) \leq \epsilon_2$, which is clear  since $f(\tilde \pi, \tilde U) - p(\tilde U) \leq \iprod{\tilde \pi}{Z(\tilde  \x)}$.  This inequality comes from the fact that problem \eqref{prob:manifold:nlp:2} is a dual formulation of  $p(U): = \min_{\pi \in \Pi(r, c)} f(\pi, U)$ \rev{with $f(\pi, U)= \iprod{\pi}{C(U)}$} for fixed $U$. 
\end{remark}

\subsection{\ReALM~for solving problem \eqref{prob:manifold:nlp:2}} \label{subsection:ReALM}
In this subsection,  we extend the exponential ALM to the manifold case, wherein the manifold constraints are kept in the subproblem,  to solve problem   \eqref{prob:manifold:nlp:2}.
We define the  augmented Lagrangian function  based on the exponential penalty function as 
\[
\widetilde \Lcal_{\eta}(\x, y, \pi)= r^{\Tsf} \alpha   +  c ^{\Tsf} \beta  + y +  \eta \sum\nolimits_{ij} \pi_{ij} \rev{\exp\left(- \frac{\varphi_{ij}(\x)  + y}{\eta} \right)},
\]
where  $\pi \in \Rbb_{++}^{n \times n}$ is the Lagrange multiplier  corresponding to the inequality constraints in \eqref{prob:manifold:nlp:2}  and $\eta > 0$ is the penalty parameter. Fix the current estimate of $\pi$ and $\eta$ as $\pi^k$ and $\eta_k$, respectively.  Then the subproblem at the $k$-th iteration is given as
\be \label{prob:sub:RALMexp:0} 
 \min_{\x \in \Mcal, y\in\Rbb} \widetilde \Lcal_{\eta_k}(\x, y, \pi^k). 
\ee
Observe that  the minimizer $(\x^{k\ast}, y_k^*)$ of problem \eqref{prob:sub:RALMexp:0} must satisfy the relationship $y_k^* = \eta_k \log(\|\zeta_{\eta}(\x^{k \ast}, \pi^k)\|_1)$,  where the  matrix  $\zeta_{\eta}(\x, \pi)\in \Rbb^{n \times n}$ is given as  
\be \label{equ:zeta} 
  [\zeta_{\eta}(\x, \pi)]_{ij}= \rev{\pi_{ij}\exp\left( - \frac{\varphi_{ij}(\x)}{\eta}\right) = \exp\left( - \frac{\varphi_{ij}(\x) - \eta \log \pi_{ij}}{\eta}\right)}. 
\ee

By eliminating the variable $y$,  we obtain an equivalent formulation of \eqref{prob:sub:RALMexp:0}  as  
  \be \label{prob:sub:RALMexp:0:1} 
   \min_{\x \in \Mcal}\, \left\{\Lcal_{\eta_k}(\x, \pi^k): =   r^{\Tsf} \alpha   +  c ^{\Tsf} \beta  + \eta_k \log(\|\zeta_{\eta_k}(\x, \pi^k)\|_1)\right\}.
\ee
Define  the function $\pi_{\eta}(\x, \pi^k)\in \Rbb^{n \times n}$ with 
\be \label{equ:pi:xi:kappa}
[\pi_{\eta}(\x, \pi^k)]_{ij} = \rev{\frac{[\zeta_{\eta}(\x, \pi^k)]_{ij}}{\|\zeta_{\eta}(\x, \pi^k)\|_1}}.
 \ee
Let $\err_1(\x, \pi^k) = \|\grad_{U}\Lcal_{\eta_k}(\x, \pi^k)\|_{\Fsf} = \| \proj_{\Tcal_U\Gcal} ( -2 V_{\pi_{\eta}(\x,\pi^k)} U)\|_{\Fsf}$ and 
$\err_2(\x, \pi^k) = \|\nabla_{\alpha}\Lcal_{\eta_k}(\x, \pi^k)\|_1 + \|\nabla_{\beta} \Lcal_{\eta_k}(\x, \pi^k)\|_1 = \|r - \pi_{\eta}(\x,\pi^k)\1\|_1 + \| c - \pi_{\eta}(\x,\pi^k)^{\Tsf} \1\|_1.$  We  give the first-order optimality condition of problem  \eqref{prob:sub:RALMexp:0:1}  without a proof. 

\begin{lemma} \label{lem:1stopt:subproblem}
Let $\bar \x = (\bar \alpha, \bar \beta, \bar U) \in \Mcal$ be a local minimizer of problem \eqref{prob:sub:RALMexp:0:1} \rev{(with fixed $\pi^k$).} Then, $\bar \x$ is a stationary point of problem \eqref{prob:sub:RALMexp:0:1}, namely, $\err_1(\bar \x, \pi^k) $$= \err_2(\bar \x, \pi^k) $\\$= 0$.  
\end{lemma}

The $(\epsilon_1, \epsilon_2)$-stationary point of problem  \eqref{prob:sub:RALMexp:0:1} is defined as follows. 
\begin{definition}\label{def:epsilon:kkt:subproblem}
We say $\tilde \x \in \Mcal$ an $(\epsilon_1, \epsilon_2)$-stationary point of problem  \eqref{prob:sub:RALMexp:0:1} (with fixed $\pi^k$) if   $\err_1(\tilde \x, \pi^k) \leq \epsilon_1$ and $\err_2(\tilde \x, \pi^k) \leq \epsilon_2$.
\end{definition}

Denote by  $\x^k$ an $(\epsilon_{k,1}, \epsilon_{k,2})$-stationary point of the subproblem \eqref{prob:sub:RALMexp:0:1}.
We require that our $\x^k$ also satisfies the following conditions:
\be \label{equ:condition:xk}
r^\Tsf \alpha^k = c^\Tsf \beta^k, \quad \|\zeta_{\eta_k}(\x^k, \pi^k)\|_1 = 1, 
\quad  \Lcal_{\eta_k}(\x^k, \pi^k) \leq \Lcal_{\eta_k}(\x^{k-1}, \pi^k).
\ee
These conditions are important to establish the convergence of \ReALM~(as shown later).  The following key observation to the subproblem \eqref{prob:sub:RALMexp:0:1} shows that  $\x^k$ satisfying the first two conditions in \eqref{equ:condition:xk} can be easily obtained.  We omit the detailed proof for brevity since they can be verified easily. 
\begin{proposition} \label{proposition:shift}
For any $\x \in \Mcal$,  consider  $\x^s = (\alpha + \upsilon_1\1, \beta + \upsilon_2 \1, U) \in \Mcal$ with 
$\upsilon_1 = (c^\Tsf \beta - r^\Tsf {\alpha} +  \eta\log(\|\zeta_{\eta}(\x, \pi)\|_1))/2$ and $\upsilon_2 = (r^\Tsf {\alpha} - c^\Tsf \beta + \eta\log(\|\zeta_{\eta}(\x, \pi)\|_1))/2.$
Then we have that $r^\Tsf (\alpha + \upsilon_1\1) = c^\Tsf (\beta + \upsilon_2\1)$ and $\|\zeta_{\eta_k}(\x^s, \pi)\|_1 = 1$ and also that  $\Lcal_{\eta}(\x, \pi) = \Lcal_{\eta}(\x^s, \pi)$  and $\nabla_{\x}\Lcal_{\eta}(\x, \pi) = \nabla_{\x} \Lcal_{\eta}(\x^s, \pi)$. 
\end{proposition}

With $\x^k$  in hand, we compute the candidate of the next estimate $\pi^{k+1}$ as 
\be\label{equ:pi:update:ReALM}
\tilde \pi^{k+1} = \zeta_{\eta_k}(\x^k, \pi^k).
\ee
Denote the matrix $W^k \in \Rbb^{n \times n}$ with $W^k_{ij} = \min\{\eta_k \tilde \pi_{ij}^{k+1}, \varphi_{ij}(\x^k)\}$, 
which measures the complementarity violation in \eqref{equ:kkt} \cite{andreani2008augmented}.
 The penalty parameter $\eta_{k+1}$ is updated  according to the progress of  the complementarity violation \cite{andreani2008augmented}, denoted by $\|W^k\|_{\Fsf}$.   If $\|W^k\|_{\Fsf} \leq \gamma_W \|W^{k-1}\|_{\Fsf}$ with $\gamma_W \in (0,1)$, we keep $\eta_{k+1} = \eta_k$ and update $\pi^{k+1} = \tilde \pi^{k+1}$; otherwise we keep $\pi^{k+1} = \pi^k$ and reduce $\eta_{k+1}$ via 
\be\label{equ:eta_k:update} 
\eta_{k+1} = \min\left\{\gamma_{\eta} \eta_{k},  {\varrho_k}/{\|\!\log  \pi^{k+1}\|_{\infty}}\right\} \quad \mbox{with}\quad  \gamma_{\eta} \in (0,1), \quad \varrho_k \to 0.
\ee

We summarize the above discussion as the complete algorithm in  \cref{alg:ReALM}.

\begin{algorithm}[!htbp]
\caption{\ReALM~for solving \eqref{prob:manifold:nlp:2}.}\label{alg:ReALM}
\SetKwInOut{Input}{} 
Choose  $\epsilon_c, \epsilon_1, \epsilon_2, \eta_0 > 0$, $\gamma_W,\gamma_{\eta}, \gamma_{\epsilon}\in (0,1)$. Select $\x^0 \in \Mcal$ with $r^{\Tsf}\alpha^0 = c^{\Tsf}\beta^0$, $\|\zeta(\x^0, \pi^0)\|_1 = 1$  and compute $W^0$ with $W^0_{ij} = \min\{\eta_0,\varphi_{ij}(\x^0)\}$.
Set $k = 1$, $\pi^k = \1 \1^{\Tsf}$, and $\eta_k = \eta_0$.   Choose  $\epsilon_{k,1}  \geq \epsilon_1, \epsilon_{k,2} \geq \epsilon_2$.\;\\
\For{$k = 1, \ldots, $}{
  Compute  an $(\epsilon_{k,1}, \epsilon_{k,2})$-stationary point $\x^k$ of  \eqref{prob:sub:RALMexp:0:1} satisfying \eqref{equ:condition:xk}.  \\
   Compute  $\tilde \pi^{k+1}$ according to \eqref{equ:pi:update:ReALM} and compute $W^k$.\\ 
  \lIf{$\|W^k\|_{\Fsf}  \leq \epsilon_{c}$, $\err_1(\x^k, \pi^k) \leq \epsilon_1$, $\err_2(\x^k, \pi^k) \leq \epsilon_2$}{
 return $\x^k$ and $\pi^{k+1} = \tilde \pi^{k+1}$}
  \leIf{$\|W^k\|_{\Fsf} \leq \gamma_W \|W^{k-1}\|_{\Fsf}$}{set $\pi^{k+1} =  \tilde \pi^{k+1}$, $\eta_{k+1} = \eta_k$;
}
 {set $\pi^{k+1} = \pi^k$ and update $\eta_{k+1}$ via \eqref{equ:eta_k:update}
} Set $\epsilon_{k+1,1} = \gamma_{\epsilon} \epsilon_{k,1}$ and  $\epsilon_{k+1,2} = \gamma_{\epsilon} \epsilon_{k,2}$.
}
\end{algorithm}

\begin{remark}\label{remark:relation:Huang}
First, if we adopt  \RBCD~or \RABCD~to solve a single subproblem \eqref{prob:sub:RALMexp:0:1}  with  $\pi^k = \1\1^{\Tsf}$, $\epsilon_1^k = \epsilon_1$ and $\epsilon_2^k = \epsilon_2$, then \cref{alg:ReALM} reduces to the approach in \cite{huang2021riemannian}, which can be also viewed as a Riemannian exponential penalty approach.    Second, our proposed \ReALM~is a nontrivial extension of the exponential ALM from the Euclidean case \cite{echebest2016convergence} to the Riemannian case.  
The key differences lie in the last requirement in \eqref{equ:condition:xk}  and the scheme for updating the penalty parameter \eqref{equ:eta_k:update}, which are crucial to establish the boundness of $\{\x^k\}$.  The latter is motivated by the corresponding scheme in \cite{lu2012augmented}, wherein a quadratic augmented Lagrangian function is used, and the magnitude of the penalty parameters outgrows that of Lagrange multipliers. 
\end{remark}

\subsection{Convergence analysis of \ReALM}\label{subsection:ReALM:convergence}
\begin{theorem}
Let $\{\x^k\}$ be the sequence generated by \cref{alg:ReALM} with $\epsilon_1 = \epsilon_2 = \epsilon_c =  0$ and $\x^\infty$ be a limit point of $\{\x^k\}$. Then, $(\x^\infty,0)$ is a stationary  point of problem \eqref{prob:manifold:nlp:2}. 
\end{theorem}
\begin{proof}
With the first two requirements in \eqref{equ:condition:xk}, we have $\Lcal_{\eta_k}(\x^k, \pi^k)  = 2r^\Tsf \alpha^k$. By the third requirement in  \eqref{equ:condition:xk},  we further have 
\be \label{equ:r:alpha:upperbound} 
r^\Tsf \alpha^k \leq r^\Tsf \alpha^0, \quad \forall~k\geq 0. 
\ee
Besides,  since $\x^k$ is an $(\epsilon_{k,1}, \epsilon_{k,2})$-stationary point  of  \eqref{prob:sub:RALMexp:0:1}, we have
\be  \label{equ:thm:ReALM:convergence:kkt} 
\big\|\proj_{\Tcal_{U^{k}}\Gcal} \left(  -2V_{\tilde \pi^{k+1}}    U^{k} \right) \big\|_{\Fsf} \leq \epsilon_{k,1},
\ 
\|r - \tilde \pi^{k+1}\1\|_1 + \|c - (\tilde \pi^{k+1})^{\Tsf}\1\|_1 \leq \epsilon_{k,2}.
\ee

We next consider two cases. 

Case i). $\eta_k$ is bounded below by  some threshold value $\underline\eta > 0$.  Due to the update rule of $\eta_{k+1}$, we can see that  \eqref{equ:eta_k:update} is invoked only finite times.  Besides, due to $0 < \pi^k_{ij} < 1$, without loss of generality, we assume $\eta_k \equiv \underline \eta$ for all $k\geq 2$ and $\lim_{k \to \infty} \pi_{ij}^k = \pi_{ij}^{\infty}$.  Meanwhile, we have  $\pi_{ij}^k = \tilde \pi_{ij}^k$ for $k \geq 2$ and 
\be  \label{equ:thm:ReALM:convergence:00} 
\big|\!\min\{\underline \eta \pi_{ij}^k, \varphi_{ij}(\x^k)\}\big| \to 0.
\ee
We now show that $\{\x^k\}$ has at least a limit point. By \eqref{equ:thm:ReALM:convergence:00},  we must have $\varphi_{ij}(\x^k) \to 0$ for each  $(i,j)\not \in \Acal(\pi^{\infty}):= \{(i,j) \in [n] \times [n] \mid \pi_{ij}^{\infty} = 0\}$ and there must exist $K_1 > 0$ such that for all $k \geq K_1$, there holds $\varphi_{ij}(\x^k) \geq -1$ for each $(i,j) \in \Acal(\pi^{\infty})$. \rev{Recalling \eqref{equ:hx:def}}, then we can conclude from \eqref{equ:thm:ReALM:convergence:00} that there exists $K_2 > 0$ such that for all $k \geq K_2$, there holds that 
$\alpha_i^k + \beta_j^k + \langle \rev{M_{ij}}, U^k(U^k)^{\Tsf}\rangle \geq - 1$ for each  $(i,j) \in [n] \times [n].$
Multiplying both sides of the above assertions by $c_j$ and then summing the obtained inequality from $j = 1$ to $n$, we have
\be \label{equ:thm:ReALM:convergence:a0} 
\alpha_i^k \geq -1 -  \sum\nolimits_{j} c_j \langle \rev{M_{ij}}, U^k(U^k)^{\Tsf}\rangle - c^{\Tsf}\beta^k \geq -1 - \|C\|_{\infty} -  r^{\Tsf}\alpha^0,
\ee
where the second inequality is due to  $r^{\Tsf} \alpha^k = c^{\Tsf} \beta^k$, \eqref{equ:r:alpha:upperbound},  and   
\be \label{equ:M:UU}
\langle \rev{M_{ij}}, U(U)^{\Tsf}\rangle = [C(U)]_{ij}  \leq  C_{ij} \leq \|C\|_{\infty} \quad \forall~U \in \Gcal.
\ee Combining  \eqref{equ:r:alpha:upperbound} and \eqref{equ:thm:ReALM:convergence:a0}, we further have 
 $-1 - \|C\|_{\infty} -  r^{\Tsf}\alpha^0 \leq \alpha_i^k \leq r_i^{-1} \big( (1 - r_i)(1 + \|C\|_{\infty} +  r^{\Tsf}\alpha^0) + r^{\Tsf}\alpha^0\big).$
 Similarly, we can establish a similar bound for each $\beta_j^k$. Recalling that $U^k$ is in a compact set, it is thus safe to say that the sequence $\{\x^k\}$ has at least one limit point, denoted by $\x^{\infty} = \lim_{k \in \Kcal, k \to \infty} \x^k$.  Again with $\lim_{k \to \infty} \pi_{ij}^k = \pi^{\infty}$, we have from \eqref{equ:thm:ReALM:convergence:00}  that  $\min\{\underline \eta \pi_{ij}^{\infty}, \varphi_{ij}(\x^\infty)\} = 0$, which, together with the fact \rev{$0 \leq \iprod{\pi^{\infty}}{Z(\x^{\infty})}  = \sum_{ij} \iprods{\pi_{ij}^\infty}{\varphi_{ij}(\x^\infty)} - \varphi(\x^\infty)_{\min}$,}
 further implies  $\varphi(\x^\infty)_{\min} = 0$ and $\iprod{\pi^{\infty}}{Z(\x^{\infty})} = 0$. This shows that $(\x^\infty,0)$ is feasible.   Moreover, letting $k \in \cal K$ go to infinity in \eqref{equ:thm:ReALM:convergence:kkt}, we further have $\pi^{\infty} \in \Pi(r,c)$ and $\|\proj_{\Tcal_{U^{\infty}}\Gcal} \left(-2V_{\pi^{\infty}} U^{\infty} \right) \|_{\Fsf}  = 0$.  From \cref{lem:1stopt}, we know that $(\x^\infty,0)$ is a stationary point of problem \eqref{prob:manifold:nlp:2}.

Case ii).  The sequence $\{\eta_k\}$ is not bounded below by any positive number, namely, $\lim_{k \to \infty} \eta_k = 0$. By the updating rule, we know that $\eta_k$ is updated via \eqref{equ:eta_k:update} infinitely many times. Hence, there must exist $k_1 < k_2 < \cdots $ such that $\eta_{k_\ell} \to 0$ as $\ell \to \infty$ and 
$\eta_s = \eta_{k_\ell} =  \min\big\{\gamma_{\eta} \eta_{k_\ell - 1},  {\varrho_{k_\ell}}/{|\log \pi_{\min}^{k_\ell}|}\big\}$  with 
$k_\ell \leq s < k_{\ell+1}$.
By \eqref{equ:zeta}  and the second assertion in \eqref{equ:condition:xk}, we have $\pi_{ij}^k \exp(-\varphi_{ij}(\x^k)/\eta_k)\leq 1$, which, together  with $0 < \pi_{ij}^k < 1$ \rev{and \eqref{equ:eta_k:update}}, implies 
\be\label{equ:thm:ReALM:convergence:d0} 
\varphi_{ij}(\x^{k_\ell}) \geq - \eta_{k_\ell} |\!\log \pi_{ij}^{k_\ell}|   \geq - \eta_{k_\ell} \rev{\|\!\log  \pi^{k_\ell}\|_{\infty}} \geq  - \varrho_{k_\ell},  \quad   \forall~\ell \geq 1.  
\ee
Using the same arguments as in the proof in Case i), we can show that $\{\x^k\}$ is bounded over  $\{k_1, k_2, \ldots\}$ and thus $\{\x^k\}$ has at least a limit point. Without loss of generality,  assume  
$\lim_{l \to \infty} \x^{k_\ell} = \x^\infty$.  By \eqref{equ:thm:ReALM:convergence:d0} and $\varrho_k \to 0$ in \eqref{equ:eta_k:update}, we  know that  $(\x^{\infty},0)$ is feasible to problem \eqref{prob:manifold:nlp:2}.  
Due to the compactness of $\tilde \pi^k$ in \eqref{equ:pi:update:ReALM}, there must exist a subset $\Kcal \subseteq \{k_1, k_2, \ldots\}$ such that $\lim_{k \in \Kcal, k \to \infty} \tilde \pi_{ij}^{k} = 
\tilde \pi^{\infty}$. Recalling  \eqref{equ:pi:update:ReALM}, we have 
$\tilde \pi_{ij}^{k+1} = \pi_{ij}^{k} \exp(-\varphi_{ij}(\x^{k})/\eta_{k})$ for  each $k \geq 1.$
Let $\Acal(\x^\infty)= \{(i,j) \in \rev{[n]\times [n]} \mid  \varphi_{ij}(\x^{\infty}) = 0\}$. We claim  $\Acal(\x^\infty) \neq \emptyset$. Otherwise, for every $(i,j)$ we have $\varphi_{ij}(\x^{\infty}) > 0$ and thus  
\be  \label{equ:thm:ReALM:convergence:d1} 
0 \leq \lim_{k \in \Kcal, k \to \infty} \tilde \pi_{ij}^{k+1}  \leq  \lim_{k \in \Kcal, k \to \infty} \rev{\exp\left(-\frac{\varphi_{ij}(\x^{k})}{\eta_{k}}\right)} = 0,
\ee
where the equality uses $\lim_{\ell \to \infty} \x^{k_\ell} = \x^\infty$, $\Kcal \subseteq \{k_1, k_2, \ldots\}$\rev{, and} $\lim_{\ell \to \infty} \eta_{k_\ell} \to 0$.  This makes a contradiction with $\|\pi^{k+1}\|_1 = 1$.  Moreover, \eqref{equ:thm:ReALM:convergence:d1}  also further implies that 
$\iprod{\tilde \pi^{\infty}}{Z(\x^{\infty})} = 0$. 
Finally, with \eqref{equ:thm:ReALM:convergence:kkt}, we further have $\|\proj_{\Tcal_{U^{\infty}}\Gcal}(-2V_{\tilde \pi^{\infty}}    U^{\infty})\|_{\Fsf}$\!\!\\ $ = 0$ and $\tilde \pi^{\infty} \in \Pi(r,c)$.  Putting the above things together, we know that that $(\x^\infty,0)$ is a stationary point of  \eqref{prob:manifold:nlp:2}.  The proof is completed. 
\end{proof}

\section{An \iRGD~framework for solving the subproblem \eqref{prob:sub:RALMexp:0:1}}\label{sec:subproblem}
For ease of reference, we represent  subproblem \eqref{prob:sub:RALMexp:0:1} by removing the subscript/superscript $k$ and replacing $\pi^k$ as $\kappa$ to avoid possible misleading:
\be \label{prob:sub:RALMexp}  
\min_{\x  \in \Mcal}\, \left\{\Lcal_{\eta}(\x, \kappa):=  r^{\Tsf} \alpha + c^{\Tsf} \beta + \eta \log\left(\|\zeta_{\eta}(\x,\kappa)\|_1\right)\right\},
\ee
where the parameter $\eta > 0$ and the matrix  $\kappa\in \Rbb^{n \times n}_{++}$. 
Throughout this and next sections, we also denote $\Lcal_{\eta}(\x, \kappa)$, $\zeta_{\eta}(\x,\kappa)$,  $\pi_{\eta}(\x, \kappa)$, $\err_1(\x, \kappa)$, and $\err_2(\x,\kappa)$  as $\Lcal(\x)$, $\zeta(\x)$,  $\pi(\x)$, $\err_1(\x)$, and $\err_2(\x)$ for short, respectively. 

At first glance, problem \eqref{prob:sub:RALMexp} is a  three-block optimization problem and can be efficiently solved by the \RABCDv~method   proposed by \cite{huang2021riemannian}. However, as stated therein,   tuning the stepsize for updating  $U$ is not easy for \RABCDv. In sharp contrast, we understand  \eqref{prob:sub:RALMexp}   as optimization with only one variable $U$ and propose a framework of inexact Riemannian gradient descent  (\iRGD)  methods, which facilitates us to choose the stepsize of updating the variable $U$ adaptively and efficiently.

Problem \eqref{prob:sub:RALMexp}  can be seen as follows:
$$  
\min_{U\in\Gcal}\, \left\{q(U):= \min_{\alpha\in \Rbb^n, \beta \in \Rbb^n} \Lcal(\x)\right\}.
$$
To show the smoothness of $q(\cdot)$, we define 
\be\label{equ:h:pi:U} 
h(\pi, U)=  \iprod{C(U) - \eta \log \kappa}{\pi} - \eta H(\pi).
\ee  
It is easy to see that  
\be\label{equ:h:L:primal:dual} 
\max_{\alpha\in\Rbb^n, \beta \in \Rbb^n} -\Lcal(\x)\quad  \mbox{is the dual formulation of} \quad \min_{\pi \in \Pi(r,c)}  h(\pi, U).
\ee
  Therefore, we can write 
\be \label{equ:qU:2}
q(U)=  - \min_{\pi \in \Pi(r,c)} h(\pi, U).
\ee
Since the entropy function is 1-strongly convex with respect to the $\ell_1$-norm over the probability simplex,   the minimization problem in \eqref{equ:qU:2} has a unique solution. We have the following proposition \cite[Lemma 3.1]{lin2020projection}.
\begin{proposition} \label{proposition:grad:Lin}
The function $q(\cdot)$ in \eqref{equ:qU:2} is differentiable over $\Rbb^{d\times k}$ and $\grad q(U)    = \proj_{\Tcal_{U}\Gcal}\big(-2 V_{\pi^*_U} U\big)$ with $\pi^{*}_U = \argmin_{\pi \in \Pi(r,c)} \  h(\pi, U)$. 
\end{proposition} 

Let  
\be \label{equ:alpha:beta:U*}
(\alpha^{*}_U, \beta^{*}_U) \in \argmin_{\alpha\in \Rbb^n, \beta \in \Rbb^n} \Lcal(\x).
\ee
We can  characterize the connection between $\pi^{*}_U$ and  $(\alpha^{*}_U, \beta^{*}_U)$ as follows, which gives a new formulation of $\grad q(U)$ and provides  more insights into approximating  $\grad q(U)$.\!  
\begin{lemma}  \label{lemma:grad:qU}
Let $(\alpha^{*}_U, \beta^{*}_U)$ be given in \eqref{equ:alpha:beta:U*}. 
Then there holds that  $\pi^*_U = \pi(\x_U^*)$ with $\x_U^* = (\alpha_U^*, \beta_U^*, U)$, \rev{where $\pi(\cdot)$ is defined in \eqref{equ:pi:xi:kappa} and $\pi^{*}_U$ is defined in \cref{proposition:grad:Lin}} and hence
\be \label{equ:qU:gradqU} 
\quad \grad q(U) = \grad_U \Lcal(\x_U^*) = \proj_{\Tcal_{U}\Gcal}\big(-2 V_{\pi(\x^*_U)} U\big).
\ee
\end{lemma}
\begin{proof}
By \eqref{equ:hx:def}\rev{, \eqref{equ:zeta},}  and \eqref{equ:pi:xi:kappa}, \rev{recalling $[C(U)]_{ij} = \langle M_{ij}, UU^\Tsf\rangle$,} we have
\be \label{equ:CU:alpha:beta}
[C(U)]_{ij} + \eta \log [\rev{\pi(\x_U^*)}]_{ij} - \eta \log \kappa_{ij} =  - [\alpha^*_U]_i -  [\beta^*_U]_j - \eta  \log (\|\zeta(\x_U^*)\|_1).
\ee
By the optimality  of $\alpha^{*}_U$  and  $\beta^{*}_U$, from \cref{lem:1stopt:subproblem}, we know   $\pi(\x_U^*)\1 = r$ and $\pi(\x_U^*)^{\Tsf}\1 = c$.  With \eqref{equ:CU:alpha:beta} and \eqref{equ:h:pi:U}, by some  calculations, we have  
$h\left(\pi(\x_U^*), U\right)   = -r^{\Tsf} \alpha^*_U - c^{\Tsf} \beta^*_U  - \eta \log (\|\zeta(\x_U^*)\|_1)  
=  -\Lcal(\x_U^*).$ By the dual formulation \eqref{equ:h:L:primal:dual}, we have  $-\Lcal(\x_U^*)$ $\leq h(\pi_U^*,U)$. Hence, we know $h\left(\pi(\x_U^*), U\right)  \leq  h(\pi_U^*,U)$, which, together with the optimality and uniqueness of $\pi_U^*$, yields $\pi(\x_U^*) = \pi_U^*$.  Finally, using \cref{proposition:grad:Lin}, we arrive at \eqref{equ:qU:gradqU}.  The proof is completed. 
\end{proof}

Hence we could use the Riemannian gradient descent (RGD) method \cite{boumal2019global} to solve \eqref{prob:sub:RALMexp}.  The main iterations are given as 
\be\label{equ:RBB:exact} 
U^{t+1} = \Retr_{U^t}\left(-\tau_t \grad q(U^t)\right),
\ee
where the stepsize $\tau_t>0$ is chosen such that  the sufficient descent evaluated by $q(U^{t+1})$ holds.  The initial guess of $\tau_t$ can be  taken as the so-called BB stepsize  \cite{wen2013feasible}, while it can also be chosen as the constant stepsize, namely, $\tau_t \equiv \tau \in (0, 2/L)$ with 
\be\label{equ:L} 
L =   2 (L_1^2 + L_2) \|C\|_{\infty} + 4L_1^2\|C\|_{\infty}^2/{\eta},
\ee where $L_1$ and $L_2$  appear in \cref{lemma:retraction:property}. 
 One can refer to \cite{boumal2019global,lin2020projection} for more details. 
Although RGD \eqref{equ:RBB:exact}  with the BB stepsize is efficient in practice,  it needs to calculate $\pi^*_{U^t}$ exactly, which can be challenging (or might be unnecessary) to do. 

Motivated by the well-established inexact gradient methods for solving optimization with noise \cite{carter1991global,devolder2014first,berahas2021global,shi2022noise} and the expression \eqref{equ:qU:gradqU}, 
we consider to  use $\grad_U \Lcal(\x^t) = \proj_{\Tcal_{U^t}\Gcal}\left(-2 V_{\pi^t}U^t\right)$, the gradient information of $\Lcal(\x^t)$ with $\x^t = (\alpha^t, \beta^t, U^t)$,  to approximate $\grad q(U^t)$, \rev{where $\pi^t$ is an approximation of $\pi^*_{U^t}$.} Here, $(\alpha^t, \beta^t) \approx
\min_{\alpha\in \Rbb^n, \beta \in \Rbb^n} \Lcal(\alpha, \beta, U^t)$
satisfying
 $\err_2(\x^t) \leq \theta_t$, where $\theta_t \geq 0$ is a preselected tolerance. Given the inexactness parameter $\theta_{t+1}\geq 0$ and the stepsize $\tau_t \geq 0$, our framework of \iRGD~updates 
$\x^{t+1} = (\alpha^{t+1},\beta^{t+1}, U^{t+1})$ via the following inexact Riemannian gradient oracle (\texttt{iRGO}):
 \tcbset{notitle, width=1.0\textwidth,top=-4mm,bottom=0mm,opacityframe=0.8}
\begin{tcolorbox} 
\begin{subequations}\label{equ:iRGD}
\be \label{equ:iRGO}
\x^{t+1} = \texttt{iRGO}(\x^t, \tau_t, \theta_{t+1}),  \quad \forall~t \geq 0
\ee
with $U^{t+1}$ and $(\alpha^{t+1}, \beta^{t+1}) \approx \argmin_{\alpha\in \Rbb^n, \beta \in \Rbb^n} \Lcal(\alpha, \beta, U^{t+1}) $ satisfying
\begin{align}
& U^{t+1} = \Retr_{U^t}\left(-\tau_t  \xi^t\right)~\mbox{with}~\xi^t = \grad_U \Lcal(\x^t), & \mbox{{\small(inexact RGD step)}}\hspace{-4mm} \label{equ:RBB:inexact} \\
& \err_2(\x^{t+1}) \leq \theta_{t+1}.  & \mbox{{\small (inexactness criterion)}}\hspace{-4mm}   \label{equ:inexact:grad:cond}
\end{align}
\end{subequations}
\end{tcolorbox}
\begin{remark}
If $\theta_t \equiv 0$ in \eqref{equ:inexact:grad:cond}, then $(\alpha^t, \beta^t) \in \argmin_{\alpha\in \Rbb^n, \beta \in \Rbb^n} \Lcal(\alpha, \beta, U^t)$, which with \cref{lemma:grad:qU} tells that  $\grad q(U^t) = \grad_U \Lcal(\x^t)$. This means that we perform a Riemannian gradient descent step with the exact gradient information.
\end{remark}

Given  $\rho \geq 0$, define the potential function as 
\be \label{equ:potential:function}
E_{\rho}(\x^t) = \Lcal(\x^{t}) + \rho  (\err_2(\x^{t}))^2,  
\ee 
We need to make further assumptions to establish the convergence of \iRGD~\eqref{equ:iRGD}. 
\begin{assumption} \label{assump:sufficient:decrease}
i) There exist $\rho \geq 0$,  $\delta_1 > 0$,  $\delta_2 \geq 0$, and $\{\omega_t\}$ such that   $\Omega:= \sum_{t = 0}^{+\infty} \omega_t < +\infty$ and 
\be  \label{equ:ls:new:general}
E_{\rho}(\x^{t+1})    \leq  E_{\rho}(\x^t) - \delta_1 \tau_t \|\xi^t\|^2_{\Fsf}  - \delta_2 (\err_2(\x^{t+1}))^2 + \omega_{t}    
\ee
holds for all $t \geq 0$.
ii) There exist $\underline \tau > 0$ such that $\tau_t \geq \underline \tau$ for all $t \geq 0$. 
\end{assumption}

 Let $\x^* = (\alpha^*, \beta^*, U^*)$ be the minimizer of \cref{prob:sub:RALMexp}.
Using a similar argument as in \cite[Lemma 4.7]{huang2021riemannian}, by \eqref{equ:M:UU} and $- \log \kappa_{ij}  \leq  \|\!\log \kappa\|_{\infty}$, 
 we have 
\be \label{equ:Lcal:opt}
\Lcal(\x^*)  \geq- \|C\|_{\infty} - \eta \|\!\log \kappa\|_{\infty}.  
\ee
Let $\widetilde \Upsilon = E_{\rho+ \delta_2}(\x^0)  +  \Omega +  \|C\|_{\infty} + \eta \|\!\log \kappa\|_{\infty}$. 
We define  
\be \label{equ:Upsilon}
\Upsilon = 
\begin{cases}
{\widetilde \Upsilon}/{\min\{\delta_1 \underline \tau, \delta_2\}}, & \mbox{if}\ \delta_2 > 0,\\
\widetilde \Upsilon/(\delta_1 \underline \tau) + \Theta, & \mbox{otherwise}
\end{cases}
\ee
with  $\Theta= \sum_{t = 0}^{+\infty} \theta_t$. 
Now, we are ready to discuss the convergence of \iRGD~\eqref{equ:iRGD}.

\begin{theorem} \label{theorem:convergence}
Let $\{\x^t\}$ be the sequence generated by \iRGD~\eqref{equ:iRGD}. 
 Suppose that \cref{assump:sufficient:decrease} holds with $\delta_2 > 0$ or \cref{assump:sufficient:decrease} holds with $\delta_2  = 0$ and $\Theta < +\infty$. 
Then  
 $\err_1(\x^t) \to 0$ and $\err_2(\x^t) \to 0.$
Moreover,  given  any $\epsilon_1 > 0$ and $\epsilon_2 > 0$, \iRGD~can return a point $\x^t$ with $\err_1(\x^t) \leq \epsilon_1$ and $\err_2(\x^t) \leq \epsilon_2$ in  $\lceil \Upsilon \max\{\epsilon_1^{-2},  \epsilon_2^{-2}\}\rceil$ iterations, where $\Upsilon$ is defined in \eqref{equ:Upsilon}.
\end{theorem}
\begin{proof}
First, since \cref{assump:sufficient:decrease} holds,  by \eqref{equ:potential:function} and  $\|\xi^t\|_{\Fsf} =  \err_1(\x^{t})$, we have 
\[
\delta_1 \underline \tau  (\err_1(\x^t))^2  + \delta_2 (\err_2(\x^{t+1}))^2 \leq \Lcal(\x^t) - \Lcal(\x^{t+1}) +  \rho((\err_2(\x^{t}))^2 - (\err_2(\x^{t+1}))^2) + \omega_t.
\]
Summing the above inequalities over $t = 0, \ldots, T$ and  adding the term $\delta_2(\err_2(\x^0))^2$ on both sides of the obtained equation, and then combining them with  $\Lcal(\x^{T+1})\geq \Lcal(\x^*)$ and \eqref{equ:Lcal:opt}, we  have 
$\sum_{t=0}^T  \left( \delta_1 \underline \tau (\err_1(\x^t))^2 +  \delta_2 (\err_2(\x^{t}))^2 \right)
 \leq \widetilde \Upsilon$ for any $T \geq 1$. This, together with \eqref{equ:Upsilon}, 
 further  implies 
 \be \label{equ:theorem:convergence:b3}  
\sum\nolimits_{t = 0}^T  (\err_1(\x^t))^2 + (\err_2(\x^t))^2  \leq  \Upsilon.
\ee
Moreover, since $\Theta < + \infty$, we know that $\Upsilon$ is bounded. Hence, we have $\err_1(\x^t) \to 0$ and $\err_2(\x^t) \to 0$. 
Suppose $\err_1(\x^t) \leq\epsilon_1$ and $\err_2(\x^t)\leq \epsilon_2$ are  fulfilled for $t = \overline T$ but not fulfilled for all $t < \overline T$. Then there hold that $\err_1(\x^{\overline T}) \leq \epsilon_1$, $\err_2(\x^{\overline T}) \leq \epsilon_2$, and $(\err_1(\x^t))^2 + (\err_2(\x^t))^2 > \min\{\epsilon_1^2, \epsilon_2^2\}$  for all $t < \overline T$. Setting $T$ in \eqref{equ:theorem:convergence:b3}  as $\overline T$ yields $\overline T \min\{\epsilon_1^2, \epsilon_2^2\} \leq  \Upsilon$, which 
completes the proof. 
\end{proof}

To end this section, we give two remarks on \iRGD. 

First, we can also consider \iRGD~with another inexactness condition:
 \be \label{equ:iBB:general} 
\|\grad q(U^t) - \xi^t \|_{\Fsf} \leq \theta \|\xi^t\|_{\Fsf}\quad \mbox{with}\quad   \theta \in (0, 1),
 \ee
which can be seen as an extension of the inexact condition in the Euclidean space considered in \cite{carter1991global,hu2007inexact}. Using a similar proof therein, we can prove 
$\|\grad q(U^t)\|_{\Fsf} \to 0$
if  $\tau_t \equiv \tau \in (0, 2(1 - \theta)/L)$.  However, the inexact condition in \eqref{equ:iBB:general} is generally not easy to verify since  $\grad q(U^t)$ is not available and computing it needs to solve the corresponding subproblem exactly. 
 In sharp contrast,   our inexactness condition in \cref{equ:inexact:grad:cond} is easy to implement. With $U^{t+1}$ and $(\alpha^{t}, \beta^{t})$ in hand, we can use various  methods, such as gradient descent, block coordinate descent (also known as Sinkhorn algorithm for this particular problem), Newton method, to find $\x^{t+1}$ satisfying \eqref{equ:inexact:grad:cond}. 
 
 Second,  the sufficient descent property \eqref{equ:ls:new:general} will hold if the partial function $\nabla_U \Lcal(\alpha, \beta, U)$ has some Lipschitz property (as shown in \cref{equ:lips}). In \cref{section:iRBBSs}, we will present a practical \iRBBS, for which \cref{assump:sufficient:decrease} holds with $\delta_2 = 1/2$. In this case,  the requirement on the inexactness parameter $\theta_t$ is very weak, i.e.,  it only needs to be nonnegative.  This provides us much freedom to choose $\theta_t$. 
 Moreover, we also point out that if we use the gradient descent 
 or Newton-type method instead of the Sinkhorn iteration to compute $(\alpha^t, \beta^t)$, \cref{assump:sufficient:decrease}  holds with $\delta_2 = 0$. 
In addition, very recently,   Berahas et al. \cite{berahas2021global} proposed a generic line search procedure based on the variant of the Euclidean version of \eqref{equ:iBB:general}, however, the line search condition therein involves the error between the exact function value and the approximate function value which is generally hard to evaluate. In contrast, verifying the sufficient decrease condition \eqref{equ:ls:new:general} does not need to estimate this error.  Moreover, \cite{berahas2021global}  did not discuss how to choose the stepsize adaptively as we do in \cref{section:iRBBSs}. 

\section{A practical \iRBBS~for solving  the subproblem \eqref{prob:sub:RALMexp:0:1}} \label{section:iRBBSs}
In this section, we first give a practical version of \iRGD~\eqref{equ:RBB:inexact}, named as inexact Riemannian Barzilai-Borwein method with Sinkhorn iteration (\iRBBS),  to solve the subproblem \eqref{prob:sub:RALMexp:0:1} \rev{or \eqref{prob:sub:RALMexp}} in  \cref{subsection:iRBBSs}.  Then, we verify that \cref{assump:sufficient:decrease} holds with $\rho \in [0, \eta/2)$ and  $\delta_2 = \eta/2 - \rho>0$  in \cref{subsection:iRBBSs:verification}. The convergence and complexity results of \iRBBS~are discussed in \cref{subsection:iRBBSs:convergence}.
\subsection{\iRBBS} \label{subsection:iRBBSs}
To make the framework of \iRGD~\eqref{equ:iRGD} a practical algorithm, the first main ingredient is how to compute $\alpha^{t+1}$ and  $\beta^{t+1}$ such that \eqref{equ:inexact:grad:cond} holds. Thanks to the famous Sinkhorn iteration \cite{cuturi2013sinkhorn}, we can do this efficiently.  Given $U^{t+1} \in \Gcal$ and $\alpha^{(0)} = \alpha^{t}$, $\beta^{(0)} = \beta^{t}$, 
the Sinkhorn  iteration for $\ell = 0, 1, \ldots $ is given as 
\begin{subequations}\label{equ:SK}
\begin{align}
\alpha^{(\ell+1)} ={}& \alpha^{(\ell)} - \eta \log r + \eta \log(\zeta^{(\ell)} \1), \label{equ:SK:a} \\[2pt]
\beta^{(\ell+1)} ={}& \beta^{(\ell)} - \eta \log c + \eta \log ( (\zeta^{(\ell+\frac12)})^{\Tsf} \1),\label{equ:SK:b} 
\end{align}
\end{subequations}
where 
$\zeta^{(\ell)} = \zeta(\alpha^{(\ell)}, \beta^{(\ell)}, U^{t+1})$ and $\zeta^{(\ell+\frac12)} = \zeta(\alpha^{(\ell+1)}, \beta^{(\ell)}, U^{t+1}).$  \rev{Here, we call \eqref{equ:zeta}  that $[\zeta(\alpha, \beta, U)]_{ij} = \kappa_{ij}\exp(-(\alpha_i + \beta_j + \langle \rev{M_{ij}}, UU^{\Tsf}\rangle)/\eta)$.}
Denote  $u^{(\ell)} = \exp(-\alpha^{(\ell)}/\eta)$ and  $v^{(\ell)} = \exp(-\beta^{(\ell)}/\eta)$. 
An equivalent formulation of \eqref{equ:SK} is given as 
\be\label{equ:SK:2}
u^{(\ell+1)} = r/(A v^{(\ell)}),  \quad v^{(\ell+1)} = c/(A^{\Tsf} u^{(\ell+1)}), 
\ee
where $A \in \Rbb^{n \times n}$ with $A_{ij} = \kappa_{ij} \exp(-\langle \rev{M_{ij}}, UU^{\Tsf}\rangle/\eta)$. Here, for vectors $a, b \in \Rbb^n$, the vector $a/b \in \Rbb^n$ is defined as $(a/b)_i = a_i/b_i$.  

Note that  $\alpha^{(\ell+1)}$ and $\beta^{(\ell+1)}$ in \eqref{equ:SK} satisfy 
\be\label{equ:sk:a:2} 
\alpha^{(\ell+1)} \in \argmin_{\alpha \in \Rbb^n}\ \Lcal(\alpha, \beta^{(\ell)}, U^{t+1})\quad \mbox{and}\quad  \beta^{(\ell+1)} \in \argmin_{\beta \in \Rbb^n}\ \Lcal(\alpha^{(\ell+1)}, \beta, U^{t+1}).
\ee
By  \cite[Remark 3.1]{huang2021riemannian}, we have  
\be \label{equ:sk:prop1}
\|\zeta^{(\ell+\frac12)}\|_1 = \|\zeta^{(\ell+1)}\|_1 = 1, \quad \forall~\ell \geq 0
\ee
and hence 
\be \label{equ:pi:zeta} 
\pi^{(\ell+\frac12)}= \zeta^{(\ell+\frac12)}/\|\zeta^{(\ell+\frac12)}\|_1 = \zeta^{(\ell+\frac12)}~\,\, \mbox{and}~\,\, \pi^{(\ell+1)} = \zeta^{(\ell+1)}/\|\zeta^{(\ell+1)}\|_1 = \zeta^{(\ell+1)}.
\ee
From the optimality condition of $\beta^{(\ell+1)}$, we have 
 $\big(\pi^{(\ell+1)}\big)^\Tsf \1 - c = 0$.
Therefore, to make condition  \eqref{equ:inexact:grad:cond} hold,  we stop the Sinkhorn  iteration once 
\be \label{equ:theorem:iBB:convergence:01}
\|\pi^{(\ell+1)}\1 - r\|_1 \leq  \theta_{t+1},
\ee
and  then set $\alpha^{t+1} = \alpha^{(\ell+1)}$, $\beta^{t+1} = \beta^{(\ell+1)}$, and $\pi^{t+1} = \pi^{(\ell+1)}$.

  Next, we choose the stepsize $\tau_t$ in \eqref{equ:RBB:inexact}. As shown later in \cref{subsection:iRBBSs:verification},  \cref{assump:sufficient:decrease} holds with  $\rho \in [0, \eta/2)$,  $\delta_2 = \eta/2 - \rho$, and $\underline \tau \in (0, 2(1-\delta)/L)$.  It means that we can choose $\tau_t \in (\underline \tau, 2(1-\delta)/L)$. However, this always gives small stepsizes, making the whole algorithm perform slowly.   Here, we would like to adaptively choose the stepsize $\tau_t$ satisfying \eqref{equ:ls:new:general} by 
 adopting  the simple backtracking line search technique 
starting from an initial guess of the stepsize $\tau_t^{(0)}$. 
Owing to the excellent performance of the  BB method  in Riemannian optimization \cite{wen2013feasible,jiang2015framework,gao2018new,iannazzo2018riemannian},   we choose the initial guess $\tau_t^{(0)}$  for $t \geq 1$ as a new Riemannian BB stepsize with safeguards: 
 \be \label{equ:tau:ABB:safeguards}
 \tau^{(0)}_t = \min\{\max\{\tau_t^{\mathrm{BB}},\tau_{\min}\}, \tau_{\max}\},
 \ee
 where $\tau_{\max} > \tau_{\min} > 0$ are preselected stepsize safeguards and $\tau_1^{\mathrm{BB}} = \tau_1^{\mathrm{BB2}}$ and 
 \[
 \tau_t^{\mathrm{BB}} = 
 \begin{cases} \min\{\tau_{t-1}^{\mathrm{BB2}}, \tau_t^{\mathrm{BB2}}, \max\{\tau_t^{\mathrm{new}}, 0\}\}, & 
  \mbox{if}\  \tau_t^{\mathrm{BB2}}/\tau_t^{\mathrm{BB1}} < \psi_t,\\
  \tau_t^{\mathrm{BB1}}, &  \mbox{otherwise}, 
  \end{cases}
  \quad \forall~t \geq 2.
  \] 
 Here, $\tau_t^{\mathrm{BB1}} = \|U^t - U^{t-1}\|_{\Fsf}^2/|\iprods{U^t - U^{t-1}}{\xi^t - \xi^{t-1}}|$, 
 $\tau_t^{\mathrm{BB2}} = |\iprods{U^t - U^{t-1}}{\xi^t - \xi^{t-1}}|/\|\xi^t - \xi^{t-1}\|_{\Fsf}^2$, and $\tau_t^{\mathrm{new}}$ is defined following \cite[Eq.\,(2.15)]{huang2021equipping}. In our numerical tests,  we set the initial $\psi_t$  to be 0.05 and   update $\psi_{t+1} = \psi_t/1.02$ if  $\tau_t^{\mathrm{BB2}}/\tau_t^{\mathrm{BB1}} < \psi_t$  and update $\psi_{t+1} = 1.02 \psi_t$
 otherwise. 
 Note that our proposed  Riemannian BB stepsize is a variant of the novel BB stepsize presented in \cite{huang2021equipping}.

   The remaining thing is to identify the choice of $\omega_t$ in \eqref{equ:ls:new:general}.  Setting the initial reference function value $E_0^r = E_{\rho}(\x^{0})$, we follow the Zhang-Hager's approach \cite{zhang2004nonmonotone} to update  $E_{t+1}^r = (\gamma Q_t E^r_t +   E_{\rho}(\x^{t+1}))/Q_{t+1}$ and $Q_{t+1} = \gamma Q_t + 1$ with a constant  $\gamma \in [0,1)$ and $Q_0 = 1$.  Given $\rho \in \left[0,\eta/2\right)$, by choosing $\omega_t = E_t^r - E_{\rho}(\x^t)$ in \eqref{equ:ls:new:general}, we derive our nonmonotone line search condition   as 
 \be\label{equ:iRBBSs:NML:2}
 E_{\rho}(\x^{t+1})    \leq  E_t^r - \delta_1 \tau_t \|\xi^t\|^2_{\Fsf} - \left(\eta/2 - \rho\right)(\err_2(\x^{t+1}))^2.
  \ee
  Note that \rev{by \eqref{equ:Lcal:opt} and using the argument in} \cite[Theorem 5]{sachs2011nonmonotone},  we know that such $\omega_t$ satisfies \rev{$\sum_{t = 0}^{+\infty}\omega_t \leq \frac{\gamma}{1 - \gamma} ( E_{\rho}(\x^{0}) +\|C\|_{\infty} + \eta \|\!\log \kappa\|_{\infty}) < + \infty$.}

 We are ready to summarize the complete \iRBBS~algorithm  in \cref{alg:iRBBSs}.
\begin{algorithm}[!htbp]
\caption{A practical \iRBBS~for solving \eqref{prob:sub:RALMexp}.}\label{alg:iRBBSs}
   Choose $\tau_{\max} > \tau_{\min} > 0$, $\tau_0 > 0$, $\epsilon_1, \epsilon_2\geq 0$, $\sigma, \delta_1 \in (0,1)$, $\rho\in [0,\eta/2)$, $\delta_2 = \eta/2 - \rho$, and $\gamma \in [0,1)$. 
   Choose $(\alpha^{-1}, \beta^{-1}, U^0) \in \Mcal$. Set  $\alpha^{(0)} = \alpha^{-1}, \beta^{(0)}=\beta^{-1}$ and perform  the Sinkhorn iteration \eqref{equ:SK}  or \eqref{equ:SK:2} at $U^0$ until   
\eqref{equ:theorem:iBB:convergence:01} holds for some $\ell$.
Set  $\alpha^0 = \alpha^{(\ell+1)}, \beta^0 = \beta^{(\ell+1)}$ and $\pi^{0} = \pi^{(\ell+1)}$.

\For{$t = 0,1,\ldots$}{
Compute $\xi^t: = \grad_U \Lcal(\x^t)$. 

\lIf{$\|\xi^t\|_{\Fsf} \leq \epsilon_1$ and $\|\pi^t \1 - r\|_1 \leq \epsilon_2$}{
break}
Set $\tau^{(0)}_t$ according to \eqref{equ:tau:ABB:safeguards} if $t \geq 1$ and $\tau^{(0)} = \tau_0$ if $t = 0$.\\
\For{$s =0,1,\ldots$}{
Set $\tau_t = \tau_t^{(0)} \sigma^s$ and update $U^{t+1}  = \Retr_{U^t} \left( -\tau_t \xi^t \right)$.

Set  $\alpha^{(0)} = \alpha^{t}$ and $\beta^{(0)}=\beta^{t}$
and  perform the Sinkhorn iteration \eqref{equ:SK} or \eqref{equ:SK:2} at $U^{t+1}$ until   
\eqref{equ:theorem:iBB:convergence:01} holds for some $\ell$.

Set  $\alpha^{t+1} = \alpha^{(\ell+1)}, \beta^{t+1} = \beta^{(\ell+1)}$  and $\pi^{t+1} = \pi^{(\ell+1)}$.

\lIf{\eqref{equ:iRBBSs:NML:2} holds}{
 break}
 }
}
\end{algorithm}
Some remarks on \cref{alg:iRBBSs} are listed in order.  

First, the basic idea of proposing \cref{alg:iRBBSs} is sharply different from that of   \RABCDv~developed in \cite{huang2021riemannian}.  Ours is based on the \iRGD~framework while the latter is based on the BCD approach. 
Moreover, the overall complexity of finding an $(\epsilon_1, \epsilon_2)$-stationary point of our \cref{alg:iRBBSs} (equipped with a rounding procedure given in \cite[Algorithm 2]{altschuler2017near}) is in the same order as that of \RABCDv; see \cref{prop:complexity:prw}. 

Second, one may also understand our method as an inexact version of BCD if treating $(\alpha,\beta)$ as one block \cite{bonettini2011inexact,gur2022convergent}.  However, the problem setting or inexactness conditions therein are quite different from ours. It should also be emphasized that 
it is the inexact Riemannian BB viewpoint other than the inexact BCD  viewpoint that enables us to choose the stepsize adaptively via leveraging the efficient BB  stepsize. 
 Actually, tuning the best stepsize for the $U$-update in \RABCDv~is nontrivial. It is remarked in \cite[Remark 6.1]{huang2021riemannian} that {\it ``the adaptive algorithms \RABCD~and \RAGAS~are also sensitive to the step size, though they are usually faster than their non-adaptive versions \RBCD~and \RGAS.''} and {\it ``How to tune these parameters more systematically is left as a future
work.''}  Our numerical results in \cref{subsec:num:subprob} 
show the efficiency of our \iRBBS~over \RBCD~and its variant \RABCD.  

Third,  it might be better to use possible multiple Sinkhorn iterations rather than only one iteration as done in \RABCDv~in updating $\alpha$ and $\beta$ from the computational point of view.   The computational cost of updating $\alpha^{(l +1)}$ and $\beta^{(\ell+1)}$ via one Sinkhorn iteration   \eqref{equ:SK} or \eqref{equ:SK:2}   
is $\Ocal(n^2)$. In contrast,  the cost of updating $U^{t+1}$ via performing a Riemannian gradient descent step \eqref{equ:RBB:inexact} is $\Ocal(ndk + n^2k + dk^2)$, wherein the main cost is to compute $V_{\pi^{t}} U^t$, which can be done by observing
$ V_{\pi} U 
 =  X \Diag(\pi\1)X^{\Tsf} U + Y \Diag(\pi^{\Tsf}\1) Y^{\Tsf} U   - X\pi Y^{\Tsf} U - Y\pi^{\Tsf} X^{\Tsf} U. 
$
Considering that the cost of updating $\alpha$ and $\beta$ is much less than that of updating  $U$, it is reasonable to update $\alpha$ and  $\beta$ multiple times and update $U$ only once.

\subsection{Verification of \cref{assump:sufficient:decrease} for \iRBBS} \label{subsection:iRBBSs:verification}
We next show that \cref{assump:sufficient:decrease} holds with $\rho \in [0, \eta/2)$, $\delta_2 = \eta/2 - \rho$, and $\underline \tau \in (0, 2(1 - \delta_1)/L)$. 
We first explore the sufficient descent property of the Sinkhorn iteration \eqref{equ:SK} or \eqref{equ:SK:2}.

\begin{lemma}\label{lemma:sufficient:feasi:sk}
Let $\{(\alpha^{\ell}, \beta^\ell)\}$ be the sequence generated by \eqref{equ:SK} and consider   $\{\pi^{(\ell)}\}$ and  $\{\pi^{(\ell +\frac12)}\}$ in \eqref{equ:pi:zeta}.  Then 
 we have 
$\|\pi^{(0)}\1 - r\|_1 +  \|(\pi^{(0)})^\Tsf\1 - c\|_1  \geq \|(\pi^{(\frac12)})^\Tsf \1 - c\|_1  \geq   \|\pi^{(1)}\1 - r\|_1$ and 
\be 
 \|\pi^{(\ell)}\1 - r\|_1 \geq  \|(\pi^{(\ell+\frac12)})^\Tsf \1 - c\|_1 \geq  \|\pi^{(\ell+1)}\1 - r\|_1, \quad \forall~\ell \geq 1. \label{equ:pi:decrease:1} 
\ee
\end{lemma}
\begin{proof}
By the optimality of $\alpha^{(\ell+1)}$ in \eqref{equ:sk:a:2} \rev{and \eqref{equ:pi:zeta}}, we have $\pi^{(\ell+\frac12)}\1 - r = 0$ for $\ell \geq 0$. 
Therefore, we  have 
\[
\begin{aligned}
 &\|\pi^{(\ell+1)}\1 - r\|_1 = \|\pi^{(\ell+\frac12)}\1  - \pi^{(\ell+1)}\1 \|_1 \leq  \|\pi^{(\ell+\frac12)}   - \pi^{(\ell+1)} \|_1 \nn 
 \end{aligned}
 \]
 \[
 \begin{aligned}
={}& \sum_{ij}  \pi^{(\ell+\frac12)}_{ij} \big|1 - \exp(- (\beta^{(\ell+1)}_j - \beta^{(\ell)}_j)/\eta)\big| = \sum_{ij}  \frac{\pi^{(\ell+\frac12)}_{ij}}{[(\pi^{(\ell+\frac12)})^{\Tsf} \1]_j}  \left|[(\pi^{(\ell+\frac12)})^{\Tsf} \1]_j  - c_j\right| \nn \\
={}& \|(\pi^{(\ell+\frac12)})^{\Tsf} \1 - c\|_1,    \nn  
\end{aligned}
\]
where the first inequality uses the Cauchy-Schwarz inequality, the second equality comes from   \eqref{equ:pi:zeta} \rev{and the definitions of $\zeta^{(\ell)}$ and $\zeta^{(\ell+\frac12)}$ after \eqref{equ:SK}}, and the third equality is due to \eqref{equ:SK:b} \rev{and \eqref{equ:pi:zeta}}.  This proves $\|(\pi^{(\ell+\frac12)})^\Tsf \1 - c\|_1 \geq  \|\pi^{(\ell+1)}\1 - r\|_1$ for  $\ell \geq 0$.  On the other hand, by the optimality of $\beta^{(\ell+1)}$ in \eqref{equ:sk:a:2} \rev{and \eqref{equ:pi:zeta}}, we have $(\pi^{(\ell)})^\Tsf\1 - c = 0$ for $\ell\geq 1$.  Using a similar argument, we can prove $\|\pi^{(\ell)}\1 - r\|_1 \geq  \|(\pi^{(\ell+\frac12)})^\Tsf \1 - c\|_1$ for $\ell \geq 1$ and $\|\pi^{(0)}\1 - r\|_1 +  \|(\pi^{(0)})^\Tsf\1 - c\|_1  \geq \|(\pi^{(\frac12)})^\Tsf \1 - c\|_1$.  The proof is completed. 
\end{proof}
\begin{lemma}[Sufficient decrease of $\Lcal$ in $(\alpha, \beta)$]\label{lemma:sufficien:decrease} Let $\{(\alpha^{\ell}, \beta^\ell)\}$ be the sequence generated by \eqref{equ:SK} and consider  $\{\pi^{(\ell)}\}$ in \eqref{equ:pi:zeta}. Then, for each $\ell \geq 0$,   we have
\be \label{lemma:sufficient:decrease:11}  
\Lcal(\alpha^{(\ell+1)},\beta^{(\ell+1)},U^{t+1}) \leq\Lcal(\alpha^{(\ell)},\beta^{(\ell)},U^{t+1})  -  (\eta/2) \big(\|\pi^{(\ell)}\1 - r\|_1^2 + \|\pi^{(\ell+1)} \1 - r\|_1^2 \big).
\ee
\end{lemma}

\begin{proof}
By  \cite[Lemma 4.1]{huang2021riemannian}, we have $\Lcal(\alpha^{(\ell+1)},\beta^{(\ell+1)},U^{t+1}) \leq\Lcal(\alpha^{(\ell+1)},\beta^{(\ell)},U^{t+1})$\!\!\\ $-  (\eta/2) \|(\pi^{(\ell+\frac12)}) ^\Tsf\1 - c\|_1^2$. By \cref{lemma:sufficient:feasi:sk}, for $\ell \geq 0$, we further have
\be\label{lemma:sufficien:decrease:c00}
\Lcal(\alpha^{(\ell+1)},\beta^{(\ell+1)},U^{t+1}) \leq\Lcal(\alpha^{(\ell+1)},\beta^{(\ell)},U^{t+1}) -  (\eta/2) \|(\pi^{(\ell+1)}) \1 - r\|_1^2.
\ee
Moreover, with \eqref{equ:sk:prop1} and \eqref{equ:pi:zeta}, we  get
\be\label{lemma:sufficien:decrease:c01}
\begin{aligned}
&\Lcal(\alpha^{(\ell+1)},\beta^{(\ell)},U^{t+1}) -  \Lcal(\alpha^{(\ell)},\beta^{(\ell)},U^{t+1}) = \iprods{r}{\alpha^{(\ell+1)} - \alpha^{(\ell)}}  \\
 ={}& -\eta \iprods{r}{\log r - \log (\pi^{(\ell)} \1)} = - \eta \KL(r\| \pi^{(\ell)} \1) \leq - (\eta/2)\|\pi^{(\ell)}\1 - r\|_1^2,
\end{aligned}
\ee
where the first equality uses \eqref{equ:SK:a} and the last inequality is due to Pinsker's inequality \cite{kullback1967lower}. 
Combining \eqref{equ:pi:decrease:1}, \eqref{lemma:sufficien:decrease:c00}, and \eqref{lemma:sufficien:decrease:c01}, we obtain  \eqref{lemma:sufficient:decrease:11}. The proof is completed. 
\end{proof}

Let $(\alpha^{t+1\ast}, \beta^{t+1\ast}) \in \argmin_{\alpha\in\Rbb^n, \beta\in \Rbb^n} \Lcal(\alpha, \beta, U^{t+1})$. 
By refining the proof in \cite{dvurechensky2018computational} and \cite{lin2022efficiency}, we can prove 
\be
\max\{\|\alpha^{(\ell+1)}\|_{\var},  \|\alpha^{t+1\ast}\|_{\var},\|\beta^{(\ell+1)}\|_{\var}, \|\beta^{t+1\ast}\|_{\var}\} \leq R^{t+1},  \quad \rev{\forall~\ell \geq 0}, \label{equ:alpha:bound} 
\ee
where 
\be\label{equ:Rt1}
R^{t+1} = \|C(U^{t+1})\|_{\var} + \eta \Psi
\ee
with $\Psi =   \|\!\log \kappa \|_{\var} +  \max\{\|\!\log r\|_{\var},\|\!\log c\|_{\var}\}$. 
Since $\nabla_{\alpha} \Lcal(\alpha^{(\ell+1)}, \beta^{(\ell+1)}, U^{t+1}) =  r - \pi^{(\ell+1)}\1$ and  $\nabla_{\beta} \Lcal(\alpha^{(\ell+1)}, \beta^{(\ell+1)}, U^{t+1}) = c - (\pi^{(\ell+1)})^\Tsf\1  = 0$ and $\Lcal(\alpha, \beta, U^{t+1})$ is jointly convex with respect to $\alpha$ and $\beta$, we have 
\be \label{equ:L:sufficient:descent:00}
\Lcal(\alpha^{(\ell+1)}, \beta^{(\ell+1)},U^{t+1}) - \Lcal(\alpha^{t+1*}, \beta^{t+1*},U^{t+1}) \leq  \iprods{\pi^{(\ell+1)}\1 - r}{\alpha^{t+1*} - \alpha^{(\ell+1)}}. 
\ee
 Given $x\in \Rbb^n$, let $x_{\mathrm{m}} =\|x\|_{\var}/2 + x_{\min}$,  it holds that 
 $\|x -  x_{\mathrm{m}}  \1\|_{\infty}  = \|x\|_{\var}/2$. 
 For $y\in\Rbb^n$ with $\iprod{y}{\1} = 0$,  we further know that
$\iprod{y}{x} = \iprod{y}{x - x_{\mathrm{m}} \1}  
\leq \|y\|_{1} \|x - x_{\mathrm{m}} \1\|_{\infty} 
= ({\|x\|_{\var}}/{2})  \|y\|_{1}.$
Applying this assertion with $y = \pi^{(\ell+1)}\1 - r$, $x = \alpha^{t+1*}$  or $x = \alpha^{(\ell+1)}$,
we obtain from \eqref{equ:alpha:bound} and \eqref{equ:L:sufficient:descent:00} that
\[
\Lcal(\alpha^{(\ell+1)}, \beta^{(\ell+1)},U^{t+1}) - \Lcal(\alpha^{t+1*}, \beta^{t+1*},U^{t+1}) \leq R^{t+1}  \|\pi^{(\ell+1)}\1 - r\|_1 \quad \rev{\forall~\ell \geq 0}.
\]

We need to use the following elementary  results, whose proof is omitted due to the space reason.  
\begin{lemma} \label{lemma:sequence}
Let $\vartheta_1$ and  $\vartheta_2$ be two given positive constants.  Consider two  sequences $\{a_\ell\}, \{b_\ell\} \subseteq \Rbb_{+}$ with $\ell \geq 0$. If they obey: $a_\ell - a_{\ell+1} \geq  \vartheta_1 (b_\ell^2 + b_{\ell+1}^2)$ and  $a_\ell \leq \vartheta_2 b_\ell$ for all $\ell \geq 1$,  
then  we have $\min_{1 \leq i \leq \ell} b_i \leq \rev{\frac{\vartheta_2/\vartheta_1}{\ell + \sqrt{2}-1}}$ for all $\ell \geq 1$. 
\end{lemma}

 Applying \cref{lemma:sequence}  with $a_\ell = \Lcal(\alpha^{(\ell)},\beta^{(\ell)},U^{t+1}) - \Lcal(\alpha^{t+1*}, \beta^{t+1*},U^{t+1})$ and $b_\ell = \|\pi^{(\ell)}\1 - r\|_1$ and using \eqref{equ:pi:decrease:1} and \eqref{lemma:sufficient:decrease:11}, we have 
 $\|\pi^{(\ell+1)}\1 - r\|_1 \leq \rev{\frac{2R^{t+1}}{\eta(\ell+\sqrt{2}-1)}}$ for all $\ell \geq 0.$
 We thus have the following results. 

\begin{lemma}\label{lemma:sk:finite}
The total number of Sinkhorn iterations to find a point $\pi^{(\ell+1)}$ satisfying \eqref{equ:theorem:iBB:convergence:01} is at most \rev{$\lceil \frac{2R^{t+1}}{\eta \theta_{t+1}} + 2 - \sqrt{2}\rceil$}, where $R^{t+1}$ is defined in \eqref{equ:Rt1}.
\end{lemma}
\begin{remark}
We provide a new analysis of the iteration complexity of the original Sinkhorn iteration based on the decreasing properties developed in \cref{lemma:sufficient:feasi:sk,lemma:sufficien:decrease}. 
 It differs from the approach developed in  \cite{dvurechensky2018computational}, wherein a  switching strategy is adopted to establish the complexity. 
\end{remark}

Given $\alpha \in \Rbb^n$ and $\beta \in \Rbb^n$,   by \cite[Lemma 4.8]{huang2021riemannian}, for any $\iota \in [0,1]$ and $U_1, U_2 \in \Gcal$, 
\be \label{equ:lips}
\|\nabla_U\Lcal(\alpha, \beta, \rev{U_2}) - \nabla_U\Lcal(\alpha, \beta, \iota U_1 + (1 - \iota) U_2)\|_{\Fsf}  \leq 2 \left(\|C\|_{\infty} + 2\|C\|_{\infty}^2/\eta\right) \rev{\iota}\|U_1 - U_2\|_{\Fsf}.  
\ee
With \cref{equ:lips} and \cref{lemma:retraction:property}, following the proof of Lemma 3 in \cite{boumal2019global}, we have 
\be\label{equ:pullback}
\Lcal\left(\alpha, \beta, \Retr_U(\xi)\right) \leq\Lcal(\alpha, \beta, U) + \iprod{\grad_U\Lcal(\x)}{\xi} + \frac{L}{2}\|\xi\|_{\Fsf}^2, ~ \forall~U \in \Gcal,~\xi \in \Tcal_U \Gcal, 
\ee
where $L$ is defined in \eqref{equ:L}.

Applying  \eqref{equ:pullback} with $\alpha = \alpha^t$, $\beta = \beta^t$, $U = U^t$, and $\xi = -\tau \grad_U\Lcal(\x^{t})$,  for $\tau \in (0, 2(1 - \delta_1)/L)$,  we have 
$\Lcal(\alpha^{t},\beta^{t},U^{t+1}) \leq  \Lcal(\x^{t})  - \delta_1 \tau \|\grad_U\Lcal(\x^{t})\|_{\Fsf}^2$. 
In addition, applying \eqref{lemma:sufficient:decrease:11} with  $\alpha^{t+1} = \alpha^{(\ell+1)}$, $\beta^{t+1} = \beta^{(\ell+1)}$, and $\pi^{t+1} = \pi^{(\ell+1)}$ for some $\ell$,  we have 
$\Lcal(\x^{t+1})  \leq \Lcal(\alpha^{t},\beta^{t},U^{t+1})  - \frac\eta2 \rev{(\err_2(\x^{t+1})^2)}.$
Combining the above two assertions  together 
 shows that \cref{assump:sufficient:decrease}  
 holds with $\rho \in [0, \eta/2)$, $\delta_2 = \eta/2 - \rho$, and $\underline \tau \in (0, 2(1 - \delta_1)/L)$.  This also shows that the backtracking line search in \cref{alg:iRBBSs} terminates in at most $\lceil \log \frac{2(1-\delta_1)}{\tau_{\max}L}/\log \sigma\rceil$  trials. 
We now proved  the claims at the beginning of this subsection.
 \subsection{Convergence and complexity of \iRBBS} \label{subsection:iRBBSs:convergence}
Based on the preparation we have made, by using \cref{theorem:convergence}, we can establish the convergence and complexity results of \iRBBS, namely, \cref{alg:iRBBSs}. 

\begin{theorem} \label{theorem:complexity:practical}
 Let $\{\x^t\}$ be the sequence generated by  \cref{alg:iRBBSs}. If  $\epsilon_1 = \epsilon_2 = 0$, we have $\err_1(\x^t) \to 0$ and $\err_2(\x^t) \to 0$. If $\epsilon_1 > 0$ and $\epsilon_2 > 0$, then \cref{alg:iRBBSs} stops in at most $\lceil \Upsilon \max\{\epsilon_1^{-2},  \epsilon_2^{-2}\}\rceil$ iterations for any $\theta_t\geq 0$, where $\Upsilon$ is defined in \eqref{equ:Upsilon}. 
 \end{theorem}

\begin{remark}\label{remark:iRBBSs}
For \iRBBS, namely, \cref{alg:iRBBSs}, there is  much freedom to choose \rev{$\theta_{t+1}$} in the inexactness condition \eqref{equ:inexact:grad:cond}. More specifically, 
for any $\pi\in \Rbb_{+}^{n\times n}$ with $\|\pi\|_1 = 1$, we have $\|\pi \1 - r\|_1 \leq \|\pi\1\|_1 + \|r\|_1  \leq \|\pi\|_1 + \|r\|_1 = 2$. Therefore, if we choose $\rev{\theta_{t+1}} \equiv 2$, then the Sinkhorn iteration always stops in one iteration, and the total number of Sinkhorn iterations is  $\lceil \Upsilon \max\{\epsilon_1^{-2},  \epsilon_2^{-2}\}\rceil$. If $\rev{\theta_{t+1}} \geq  \rev{\frac{2R^{t+1}}{\eta (\ell_{\max} - 2 + \sqrt{2})}}$ with \rev{$\ell_{\max}\geq 1$} being a preselected positive integer, we know that the Sinkhorn iteration will stop in at most $\ell_{\max}$ iterations and the total number of Sinkhorn iterations is $\lceil \ell_{\max} \Upsilon \max\{\epsilon_1^{-2},  \epsilon_2^{-2}\}\rceil$. 
\end{remark}

Next, we investigate the connection of the approximate stationary points of problems \eqref{prob:sub:RALMexp} and  \eqref{prob:manifold:nlp:2}. 
\begin{theorem}\label{thm:kkt:connection}
Suppose $\tilde \x \in \Mcal$ with $\|\zeta(\tilde \x)\|_1 = 1$  is an $(\epsilon_1, \epsilon_2)$-stationary point of problem  \eqref{prob:sub:RALMexp}, namely, 
$\err_1(\tilde \x) \leq \epsilon_1$ and $\err_2(\tilde \x) \leq \epsilon_2$.  Then,  we have
\begin{subequations}
\begin{align}
&\|\proj_{\Tcal_U\Gcal} ( -2 V_{\hat \pi}\tilde U)\|_{\Fsf} \leq \epsilon_1 + 2 \|C\|_{\infty}\epsilon_2,\label{thm:connection:kkt:00} \\
&\iprod{\hat \pi}{Z(\tilde \x)} \leq  (2  \log n +  \|\!\log \kappa\|_{\var})\eta +  (\|\tilde \alpha\|_{\var} + \|\tilde \beta\|_{\var} + \rev{\|C\|_{\infty}})\epsilon_2,\label{thm:connection:kkt:01} 
\end{align}
\end{subequations}
where $\hat \pi: =\Round(\pi(\tilde \x), \Pi(r,c))$ is a feasible matrix returned by the rounding procedure given in \cite[Algorithm 2]{altschuler2017near} on $\pi(\tilde \x)$. 
\end{theorem}
\begin{proof}
We first show that \eqref{thm:connection:kkt:00} is true. 
First, by \cite[Lemma 3.2]{chambolle2022accelerated}, we have $\hat \pi \in \Pi(r,c)$ and 
 \be\label{equ:rounding:estimate} 
 \|\hat \pi - \pi(\tilde \x)\|_1 \leq \|\pi(\tilde \x)\1 - r\|_1 + \|\pi(\tilde \x)^{\Tsf}\1 - c\|_1 = \err_2(\tilde \x)\leq \epsilon_2. 
 \ee 
Second,  by the triangular inequality and the nonexpansive property of the projection operator, we have  $\|\proj_{\Tcal_U\Gcal} (A_1)\|_{\Fsf} \leq 
\|\proj_{\Tcal_U\Gcal} (A_1) - \proj_{\Tcal_U\Gcal} (A_2)\|_{\Fsf}  + \|\proj_{\Tcal_U\Gcal} (A_1)\|_{\Fsf}$\!\\ 
$ \leq \|A_1 - A_2\|_{\Fsf} + \|\proj_{\Tcal_U\Gcal} (A_1)\|_{\Fsf}$, where $A_1, A_2 \in \Rbb^{d\times k}$. Hence, we have
\begin{align} \label{equ:V:pit:hat:pit}
\|\proj_{\Tcal_U\Gcal} ( -2 V_{\hat \pi}\tilde U)\|_{\Fsf} \leq   
  2\|(V_{\hat \pi} - V_{\pi(\tilde \x)})\tilde U\|_{\Fsf}  +  \epsilon_1  \leq   2\|V_{\hat \pi} - V_{\pi(\tilde \x)}\|_{\Fsf}  + \epsilon_1,
\end{align} 
where first  inequality uses $\|\proj_{\Tcal_{\tilde U}\Gcal} ( -2 V_{\pi(\tilde \x)}\tilde U)\|_{\Fsf} \leq \epsilon_1$ and the second inequality is due to the fact that $\|A\tilde U\|_{\Fsf} \leq \|A\|_{\Fsf}$ for any matrix $A$.
Moreover,  observe that  $\|V_{\hat \pi} - V_{\pi(\tilde \x)}\|_{\Fsf}  
  =  \| \sum_{ij} (\hat \pi_{ij} - \pi(\tilde \x)_{ij}) M_{ij} \|_{\Fsf}   \leq  \|C\|_{\infty}\|\hat \pi - \pi(\tilde \x)\|_1$. Combining the above assertions with  
  \eqref{equ:rounding:estimate} and \eqref{equ:V:pit:hat:pit} yields \eqref{thm:connection:kkt:00}. 

Next, we show that  \eqref{thm:connection:kkt:01}  is also true.  By the Cauchy-Schwarz inequality and \eqref{equ:rounding:estimate}, we have
\begin{align}  \label{thm:connection:kkt:d0}
\iprod{\hat \pi}{Z(\tilde \x)} ={}&  \iprod{\hat \pi - \pi(\tilde \x)}{Z(\tilde \x)}  + \iprod{\pi(\tilde \x)}{Z(\tilde \x)}  \leq \|Z(\tilde \x)\|_{\infty} \epsilon_2 + \iprod{\pi(\tilde \x)}{Z(\tilde \x)} \nn \\
\leq{}& (\|\tilde \alpha\|_{\var} + \|\tilde \beta\|_{\var} + \rev{\|C\|_{\infty}}) \epsilon_2 + \iprod{\pi(\tilde \x)}{Z(\tilde \x)}, 
\end{align}
where  the second inequality uses  \eqref{equ:hx:def}, \eqref{equ:Z}, and \eqref{equ:M:UU}.
The remaining is to bound $\iprod{\pi(\tilde \x)}{Z(\tilde \x)}$.  By $\|\zeta(\tilde \x)\|_1 = 1$, \eqref{equ:zeta}, and \eqref{equ:pi:xi:kappa},  we have 
$\varphi(\tilde \x)_{ij}  =  \eta (\log \kappa_{ij} -  \log \pi(\tilde \x)_{ij})$  and $\varphi(\tilde \x)_{\min} \geq \eta (\log \kappa)_{\min}$.  Again with \eqref{equ:Z}, we have 
$\iprod{\pi(\tilde \x)}{Z(\tilde \x)}  = \iprod{\pi(\tilde \x)}{\varphi(\tilde \x)}  - \varphi(\tilde \x)_{\min} \leq   \eta \left(\iprod{\log \kappa}{\pi(\tilde \x)} + H(\pi(\tilde \x)) -  (\log \kappa)_{\min}\right)$,
which, together with $H(\pi(\tilde \x)) \leq 2\log n$ and $\iprod{\log \kappa}{\pi(\tilde \x)} \leq (\log \kappa)_{\max}$, further implies 
$\iprod{\pi(\tilde \x)}{Z(\tilde \x)}  \leq   (2  \log n +  \|\!\log \kappa\|_{\var})\eta.$
This, together with  \eqref{thm:connection:kkt:d0},  yields \eqref{thm:connection:kkt:01}. The proof is completed. 
\end{proof}

 Let $\{\x^t\}$ be the sequence generated by  \cref{alg:iRBBSs}.  By \eqref{equ:sk:prop1}, we know that $\|\zeta(\x^t)\|_1 = 1$. 
By $\|C(U)\|_{\infty} \leq \|C\|_{\infty}$, \eqref{equ:L}, \eqref{equ:alpha:bound}, \eqref{equ:Rt1}, 
 and  with the help of \cref{thm:kkt:connection}, \cref{remark:kkt:strong,remark:iRBBSs}, we can immediately establish the following complexity result of \cref{alg:iRBBSs}. Note that the following complexity to attain an $(\epsilon_1',\epsilon_2')$-stationary point is in the same order as $\epsilon_1'$ and $\epsilon_2'$ to that of \RABCDv~in \cite{huang2021riemannian}.  However, our algorithm can return a  stronger $(\epsilon_1',\epsilon_2')$-stationary point.  

 \begin{corollary}\label{prop:complexity:prw}
By choosing $\eta = \epsilon_2'/(4\log n + 2 \|\!\log \kappa\|_{\var})$, $\epsilon_1 = \epsilon_1'/2$, and  $\epsilon_2 = \rev{\min\{\epsilon_1'/(4\|C\|_{\infty}),\epsilon_2'/(4 \eta \Psi+ 6\|C\|_{\infty})\}}$,  where $\Psi$ is defined after \eqref{equ:Rt1},   \cref{alg:iRBBSs} can return an $(\epsilon_1',\epsilon_2')$-stationary point of  problem \eqref{prob:manifold:nlp:2}  
in  
$\Ocal(T_{\epsilon_1', \epsilon_2'})$ iterations  with 
\[
T_{\epsilon_1', \epsilon_2'} = \max\{(\epsilon_1')^{-2},(\epsilon_2')^{-2}\}(\epsilon_2')^{-1}.
\]  
If $\theta_t \geq \rev{\frac{2R^{t}}{\eta (\ell_{\max} - 2 + \sqrt{2})}}$,  the total number of Sinkhorn iterations  is  $\Ocal(\ell_{\max} T_{\epsilon_1', \epsilon_2'})$ and  
the total arithmetic operation complexity is 
$\Ocal\big( (n^2 (k +  \rev{\ell_{\max}}) + ndk +dk^2) T_{\epsilon_1', \epsilon_2'}\big).$
 Here, the $\Ocal(\cdot)$ hides the constants related to $E_{\rho}(\x^{0})$, $\Psi$, $L_1$, $L_2$,  $\log n$, and $\|C\|_\infty$.
 \end{corollary}

 Finally, suppose $U$ is fixed other than a variable and choose $\eta = \Ocal(\epsilon)$. In this case, we know from \cref{thm:kkt:connection} that by computing an $\epsilon$-stationary point of the regularized OT problem \eqref{prob:sub:RALMexp:0:1}, one can return a feasible $\epsilon$-stationary point of the OT problem \eqref{prob:OT} evaluated by the primal-dual gap other than the primal gap used in the literature, such as \cite[Theorem 1]{altschuler2017near}.  This might be of independent interest to the OT community.  

\section{Experimental results}\label{sec:experiments} 
In this section, we report plenty of numerical results to evaluate the performance of our proposed approaches to computing the PRW distance \eqref{equ:PRW}. 
We implemented all methods in MATLAB  2021b and performed all the experiments in macOS Catalina 10.15 on a Macbook Pro with a 2.3GHz 8-core Intel Core i9.
We follow the ways in \cite{lin2020projection,huang2021riemannian,paty2019subspace} to generate the synthetic and real datasets.  
\begin{dataset}[Synthetic dataset: Fragmented hypercube \cite{paty2019subspace}]\label{data:hypercube}
Define a map $P(x) = x + 2\sign(x) \odot \sum_{k = 1}^{k^*} e_k$, where $\sign(\cdot)$ is taken elementwise,  $k^*\in \{1, \ldots, d\}$ and $\{e_1, \ldots, e_{k^*}\}$ is the canonical basis of $\Rbb^{d}$. 
Let $\mu = \mathcal{U}([-1,1]^d)$ be the uniform distribution over an hypercube and  $\nu = P_{\#}(\mu)$  be  the pushforward of $\mu$ under the map $P$.   We set $k = k^* = 2$ and take both the weight vectors $r$ and $c$  as $\1/n$. 
\end{dataset}

\begin{dataset}[Real datasets: Shakespeare operas \cite{lin2020projection}]\label{data:shakespeare}
  We consider six Shakespeare operas, including  Henry V (H5),  Hamlet (H), Julius Caesar (JC),  The Merchant of Venice (MV), Othello (O), and Romeo and Juliet (RJ).  We follow the way in \cite{lin2020projection} to transform the preprocessed script as a measure over $\Rbb^{300}$. 
    Each postprocessing script corresponds to a matrix $X\in \Rbb^{300\times n_X}$. The values $n_X$ of H5, H, JC, MV, O, and RJ are 1303, 1481, 910, 1008, 1148, and 1141, respectively.   The weight vector $r$ or $c$ is  taken as $\1_{n_X}/n_X$. We choose $k = 2$.
\end{dataset}

\begin{dataset}[Real datasets: digits from MNIST datasets \cite{lin2020projection}]\label{data:mnist}
 For each digit $0, 1, \ldots, 9$, we extract the 128-dimensional features from a pre-trained convolutional neural network.   Each digit corresponds to a matrix $X \in \Rbb^{128 \times n_X}$.  The values $n_X$ of $0, 1, \ldots, 9$ are $980$, 1135, 1032, 1010, 982, 892, 958, 1028, 974, and 1009,  respectively.
 The weight vector $r$ or $c$ is  taken as $\1_{n_X}/n_X$.  We choose $k = 2$. 
\end{dataset}
\subsection{Comparison on solving the subproblem \eqref{prob:sub:RALMexp}} \label{subsec:num:subprob}
Since \cite{huang2021riemannian}  has shown the superiority of \RABCDv~over \RAGASv~proposed in \cite{lin2020projection},  in this subsection, we mainly compare our proposed  \cref{alg:iRBBSs}  with  \RABCDv.  Subproblem \eqref{prob:sub:RALMexp} with $\kappa$ being an all-one matrix and relatively small  $\eta$ is used to compute the PRW distance in  \cite{huang2021riemannian}.  We follow the same settings of $\eta$ in \cite{huang2021riemannian}.  For \cref{data:hypercube}, we choose $\eta = 0.2$ when $d < 250$ and $\eta = 0.5$ otherwise.  For  \cref{data:shakespeare}, we set $\eta  = 0.1$. For \cref{data:mnist}, we set $\eta = 8$.

{\it Parameters of 
\cref{alg:iRBBSs}.} We choose  $\tau^{\min} = 10^{-10}, \tau^{\max} = 10^{10}$, $\gamma = 0.85$, $\tau_0 = 10^{-3}$, $\sigma = 1/2$, $\delta_1 = 10^{-4}$, and $\rho = 0.49\eta$.  As for tolerances,  we  always set  $\epsilon_1  = 2 \|C\|_{\infty} \epsilon_2$  (motivated by \eqref{thm:connection:kkt:00}) and  $\epsilon_2 = 10^{-6} \max\{\|r\|_{\infty}, \|c\|_{\infty}\}$.
To make the residual error more comparable, we  choose 
\be\label{equ:theta:choice:numerical} 
\theta_0 = 1, \quad \theta_{t+1} = \max\left\{{\theta}/{(2\|C\|_{\infty})}\cdot\err_1(\x^{t}), \epsilon_2\right\}, \quad \forall~t \geq 0.
\ee
For numerical tests in this subsection,  we can always use the Sinhorn iteration \eqref{equ:SK:2} 
and we find that with such a choice, the  Sinkhorn iteration \eqref{equ:SK:2}  always stops in a relatively small number of iterations. 
 We consider several different values of  $\theta$ in \eqref{equ:theta:choice:numerical} and the corresponding version of \cref{alg:iRBBSs} is denoted as \iRBBS-$\theta$ for brevity.  Note that when $\theta = 0$, we have $\theta_{t+1} \equiv \epsilon_2$ for all $t \geq 0$, which means that we calculate $\grad q(U^t)$ almost exactly; when $\theta = +\infty$, we have $\theta_{t+1} \geq 2$ for all $t \geq 0$ and the Sinkhorn iteration always stop in 1 iteration, as discussed in \cref{remark:iRBBSs}. 
 
{\it Parameters of \RABCDv~and \RAGASv.} As stated in \cite[Remark 6.1]{huang2021riemannian}, to achieve the best performance of \RABCDv, one has to spend some efforts to tune the stepsizes. Here, we adopt the stepsizes (with a slight modification, marked in  \textit{italic} type,  to have better performance for some cases)  used in \cite{huang2021riemannian}.    
For \cref{data:shakespeare},  $\tau_{\texttt{RBCD}} = \textit{0.09}$  if the  instance is X/RJ and $\tau_{\texttt{RBCD}} = 0.1$ otherwise; $\tau_{\texttt{RABCD}} = \textit{0.0015}$
if the  instance is X/RJ with X $\neq$ H, $\tau_{\texttt{RABCD}} = \textit{0.001}$ if the  instance is H/RJ  and $\tau_{\texttt{RABCD}} = 0.0025$ otherwise.   For \cref{data:hypercube},  $\tau_{\texttt{RBCD}} = \tau_{\texttt{RABCD}} = 0.001$.  For \cref{data:mnist}, $\tau_{\texttt{RBCD}} = 0.004$, and we do not test \RABCD~for this dataset since the well-chosen stepsize is not provided in \cite{huang2021riemannian}.
 We stop  \RABCDv~when  $\err_1(\alpha^{t}, \beta^{t-1}, U^t) \leq \epsilon_1$ and $\err_2(\alpha^{t}, \beta^{t-1}, U^t) \leq \epsilon_2$ or the maximum iteration number reaches 5000.

 {\it Initial points and retraction.} We choose $\alpha^{-1} = \mathbf{0}$, $\beta^{-1} = \mathbf{0}$ for \iRBBS~and $\alpha^0 = \mathbf{0}, \beta^0 = \mathbf{0}$ for \RABCDv.   As for $U^0$, we choose $U^0 \in \argmax\nolimits_{U \in \Gcal}   \langle V_{\pi^0}, UU^{\Tsf}\rangle$ with  $\pi^0 \in \Pi(r, c)$ for all methods. Here,  $\pi^0$ is formulated by firstly generating a matrix $\tilde \pi^0$ with each entry randomly drawn from the standard uniform distribution on $(0,1)$ and then rounding $\tilde \pi^0/\|\tilde \pi^0\|_1$ to $\pi^0$ via Algorithm 2 proposed in \cite{altschuler2017near}.    The retraction operator in all the above methods is taken as the QR-retraction.
 
 For \cref{data:hypercube}, we randomly generate 10 instances for each $(n,d)$ pair,  each equipped with 5 randomly generated initial points. For each instance of  \cref{data:shakespeare,data:mnist}, we randomly generate  20 initial points. In the tables and figures in this subsection,  we report the average performance. 
The term ``nGrad'' means the total number of calculating  \rev{$\grad_U\Lcal(\x)$}, ``nSk'' means the total number of Sinkhorn iterations, and ``time'' represents the running time in seconds evaluated by ``tic-toc'' commands.  Note that for \RABCDv~and \iRBBS-inf, nGrad is always equal to nSk. Moreover, considering that we aim to compute the PRW distance,  with $U^t$ returned by some method in hand, we invoke  Gurobi 9.5 (\url{https://www.gurobi.com}) to compute a relatively accurate  PRW distance, denoted as $\mathrm{PRW_p}= \min_{\pi \in \Pi(r, c)} f(\pi, U^t)$. 
Note that a larger $\mathrm{PRW_p}$ means a higher solution quality of the corresponding $U^t$.


\begin{figure}[!tbp]
\centering
\subfloat[$n = 100$]{
\begin{minipage}[b]{0.242\textwidth}
\includegraphics[width=1\linewidth]{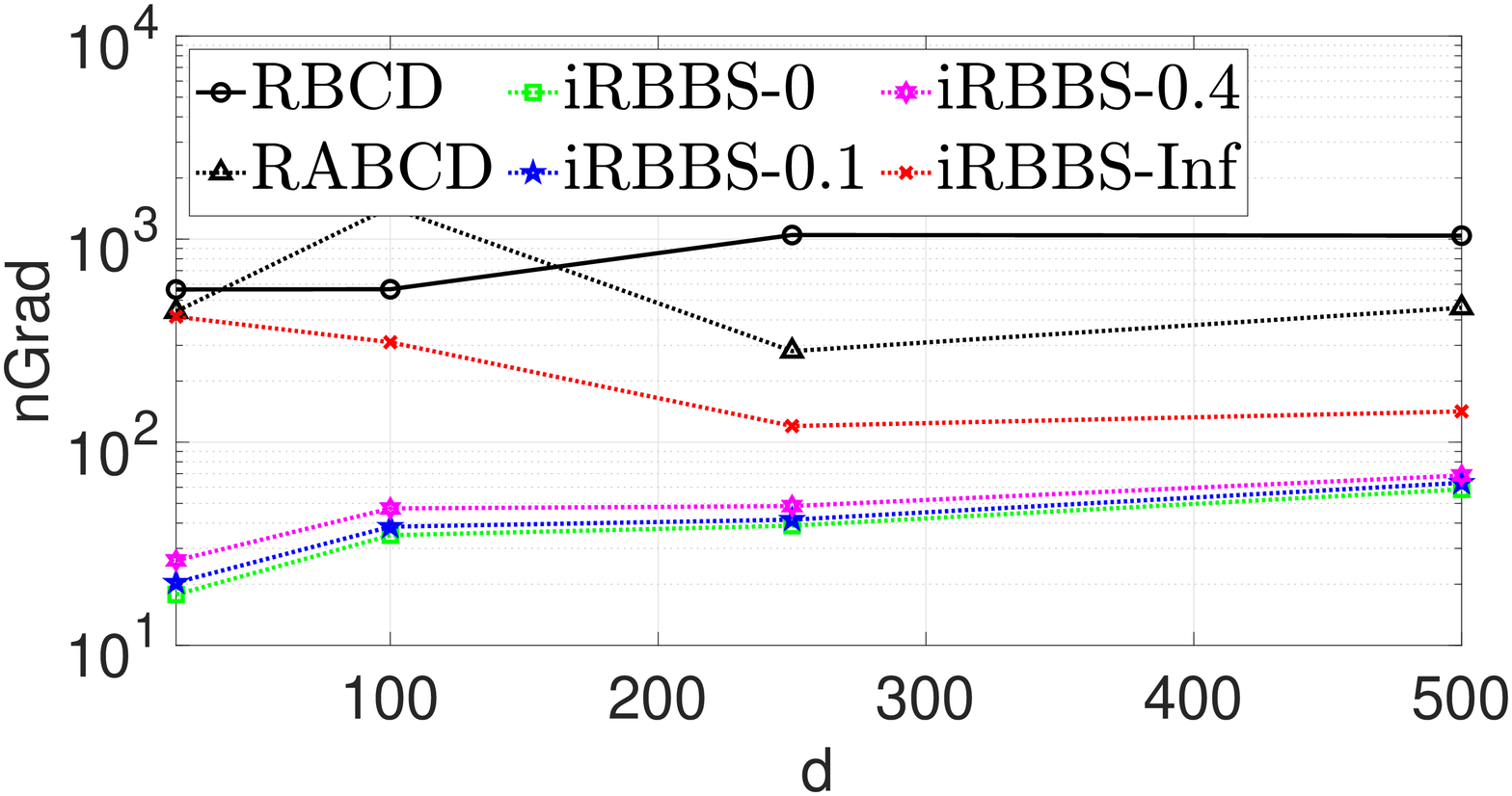}\\
\includegraphics[width=1\linewidth]{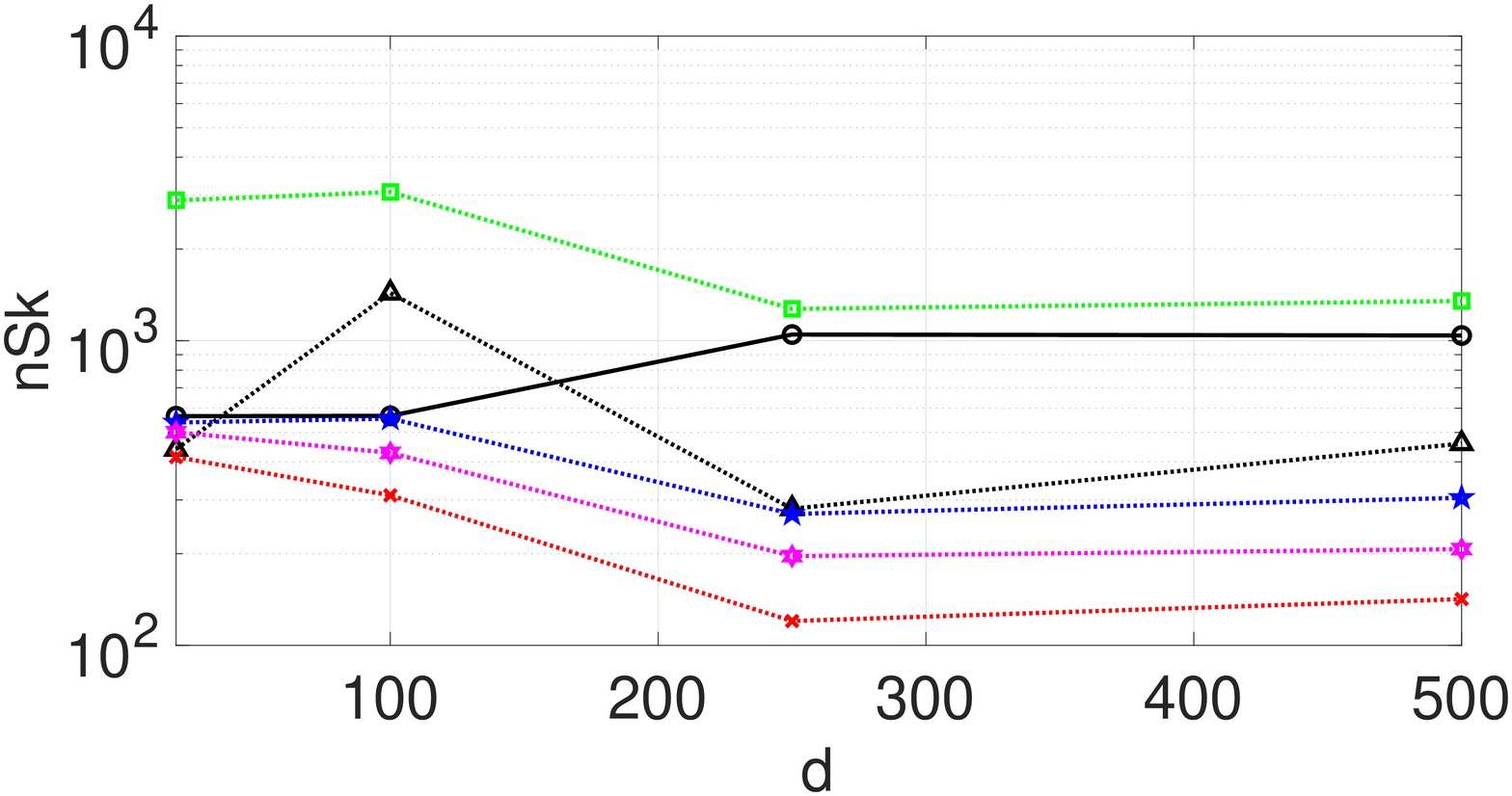}\\
\includegraphics[width=1\linewidth]{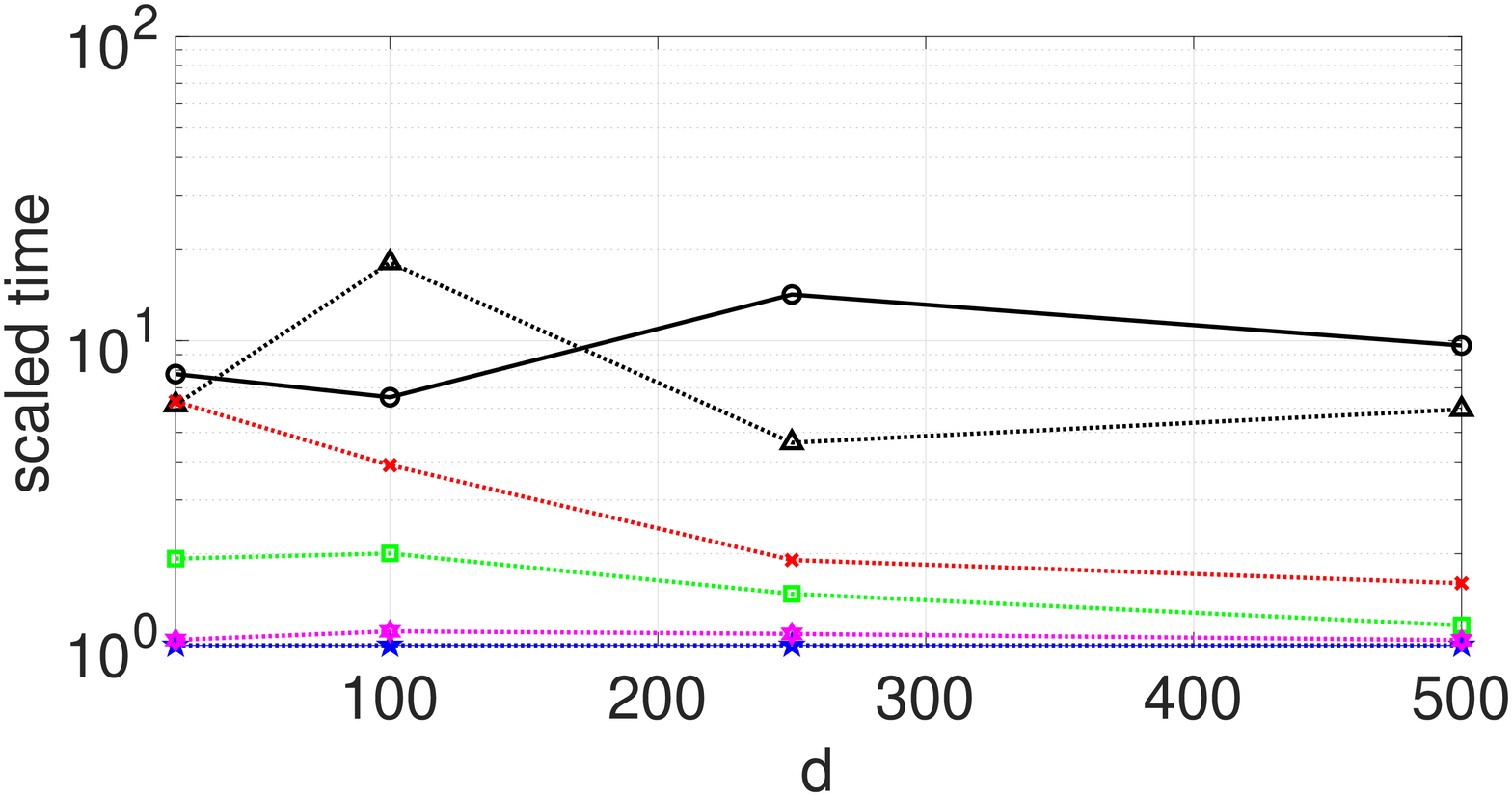}
\end{minipage}
}
\hspace{-3mm}
\subfloat[$d = 50$]{
\begin{minipage}[b]{0.245\textwidth}
\includegraphics[width=1\linewidth]{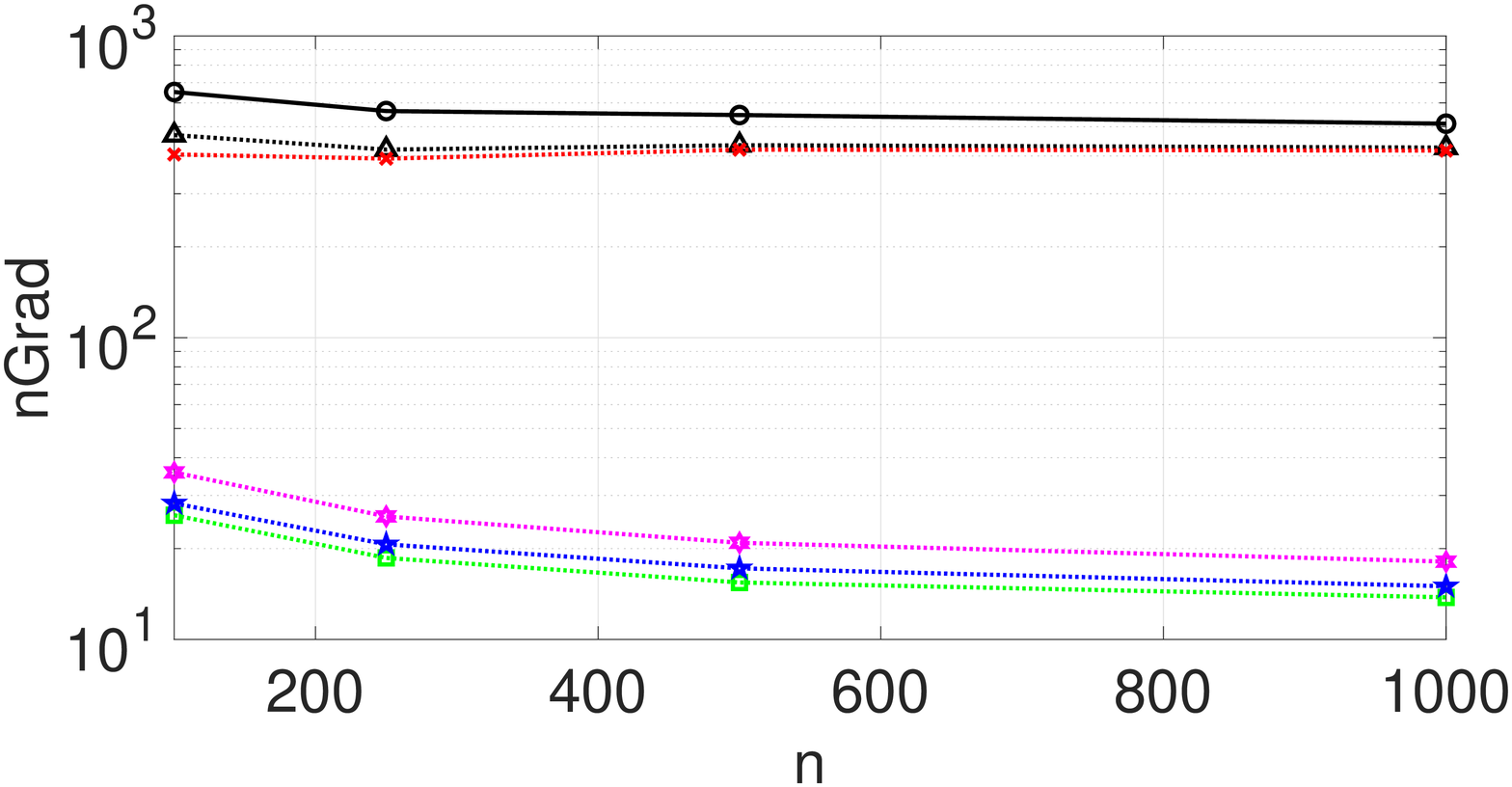}\\
\includegraphics[width=1\linewidth]{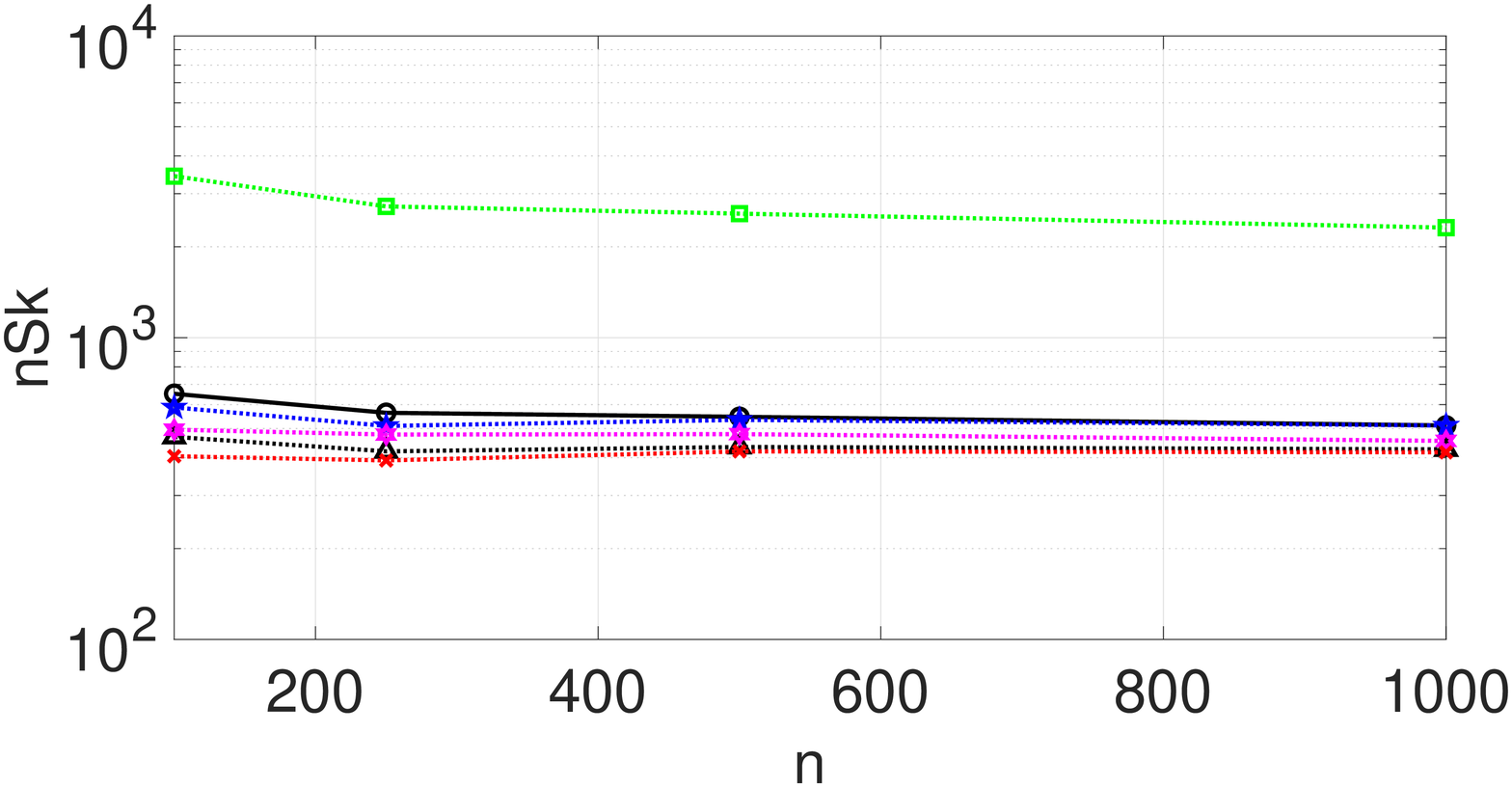}\\
\includegraphics[width=1\linewidth]{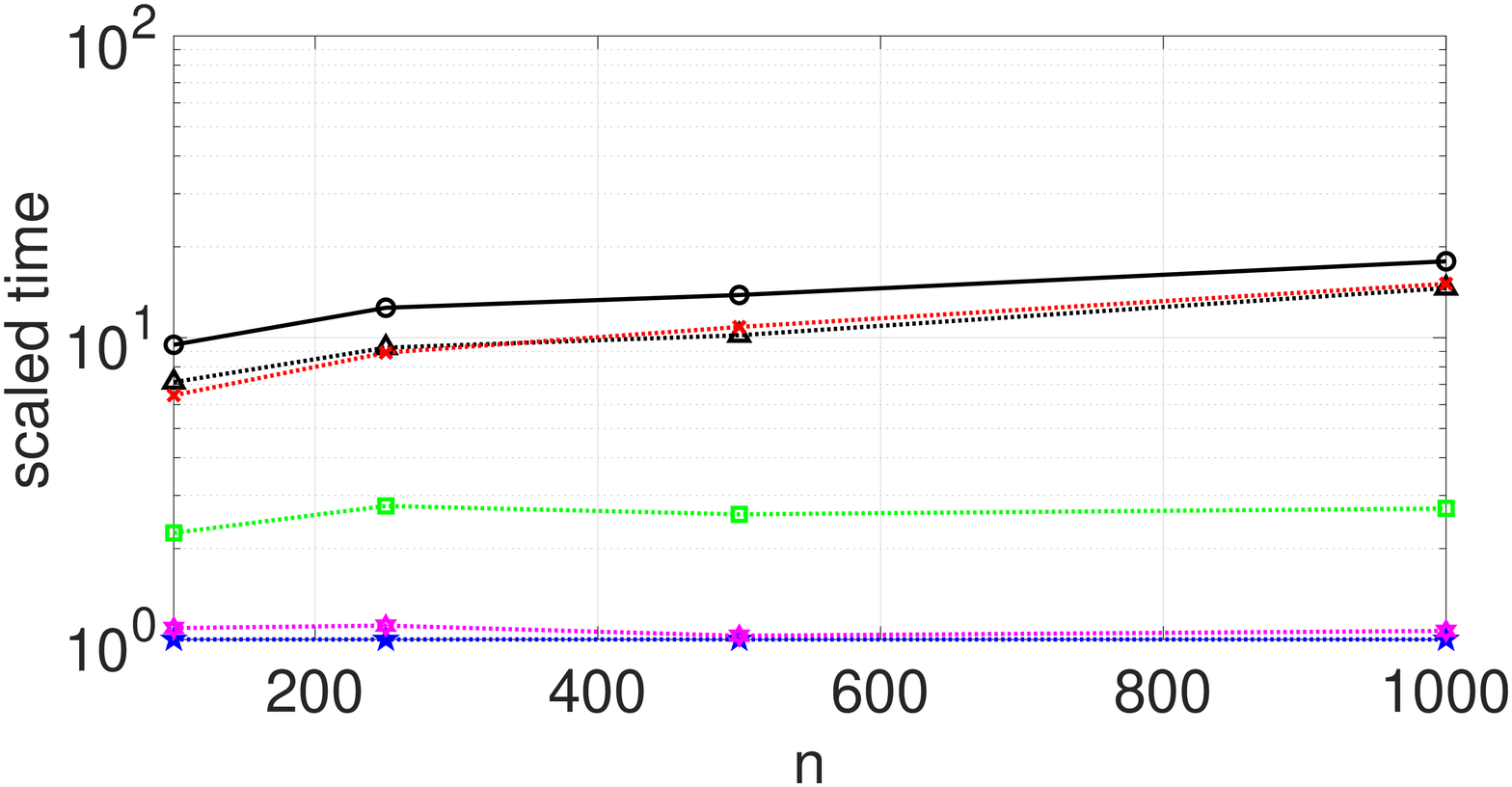}
\end{minipage}
}
\hspace{-3mm}
\subfloat[$n = d$]{
\begin{minipage}[b]{0.242\textwidth}
\includegraphics[width=1\linewidth]{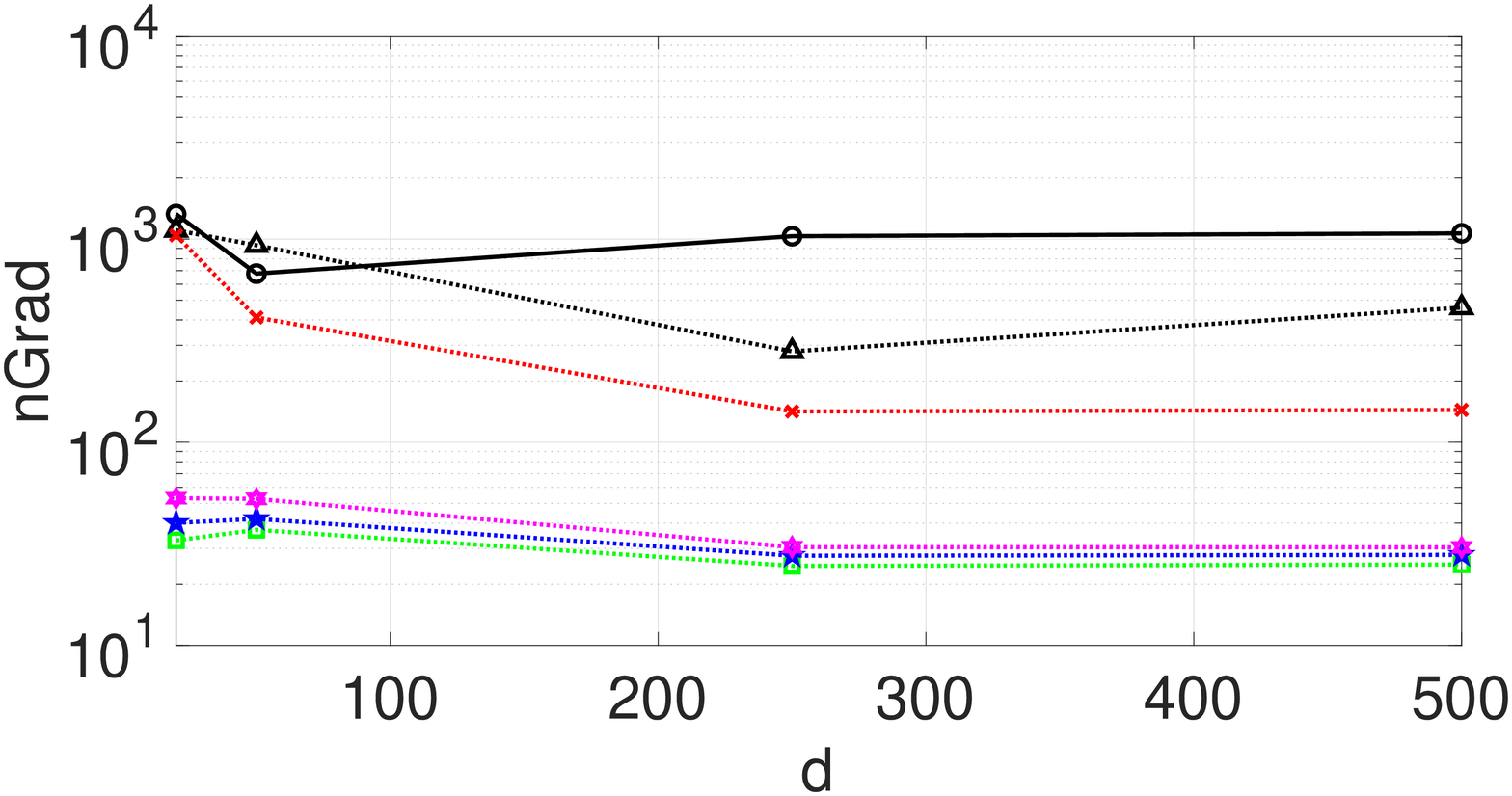}\\
\includegraphics[width=1\linewidth]{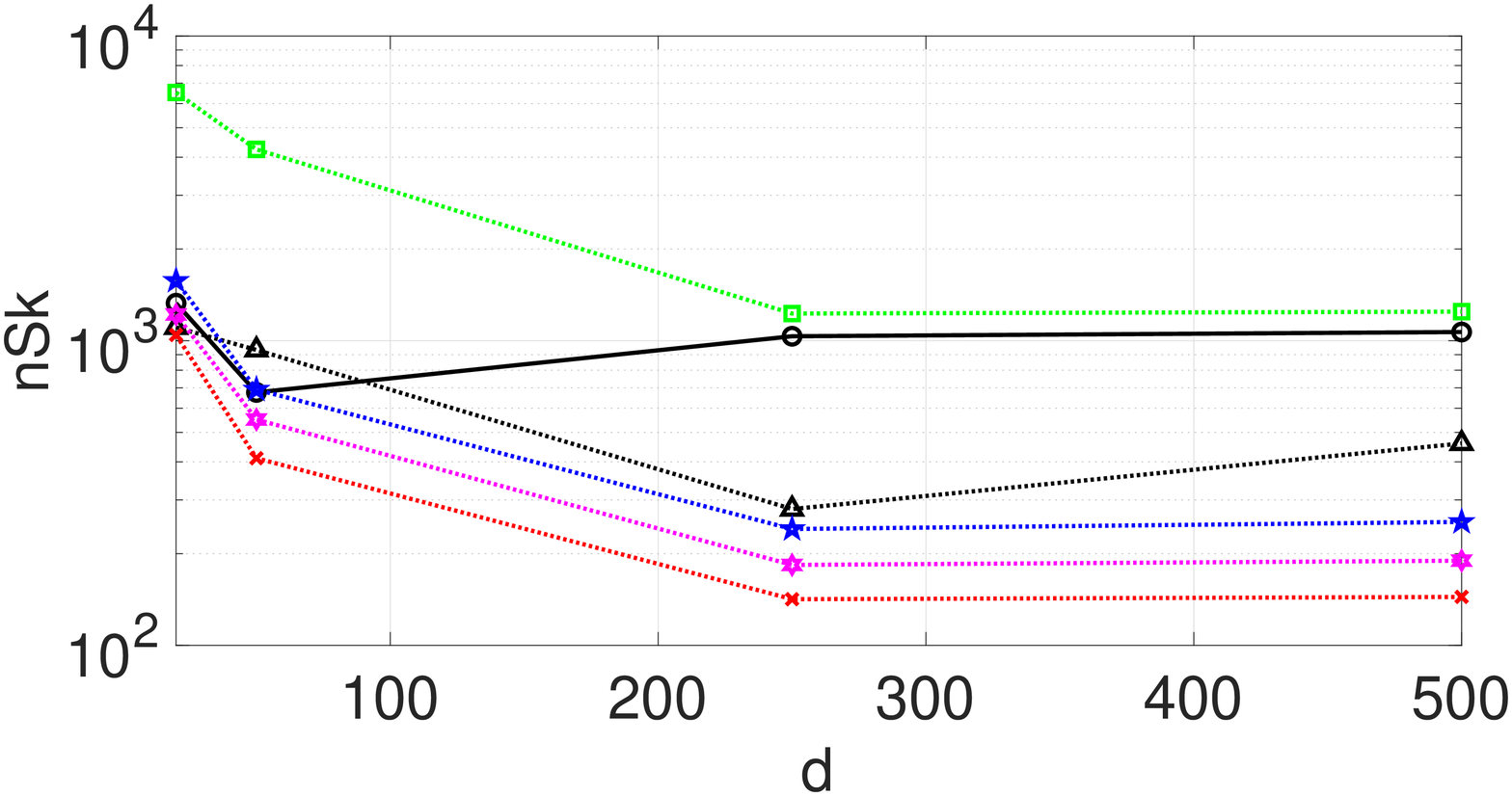}\\
\includegraphics[width=1\linewidth]{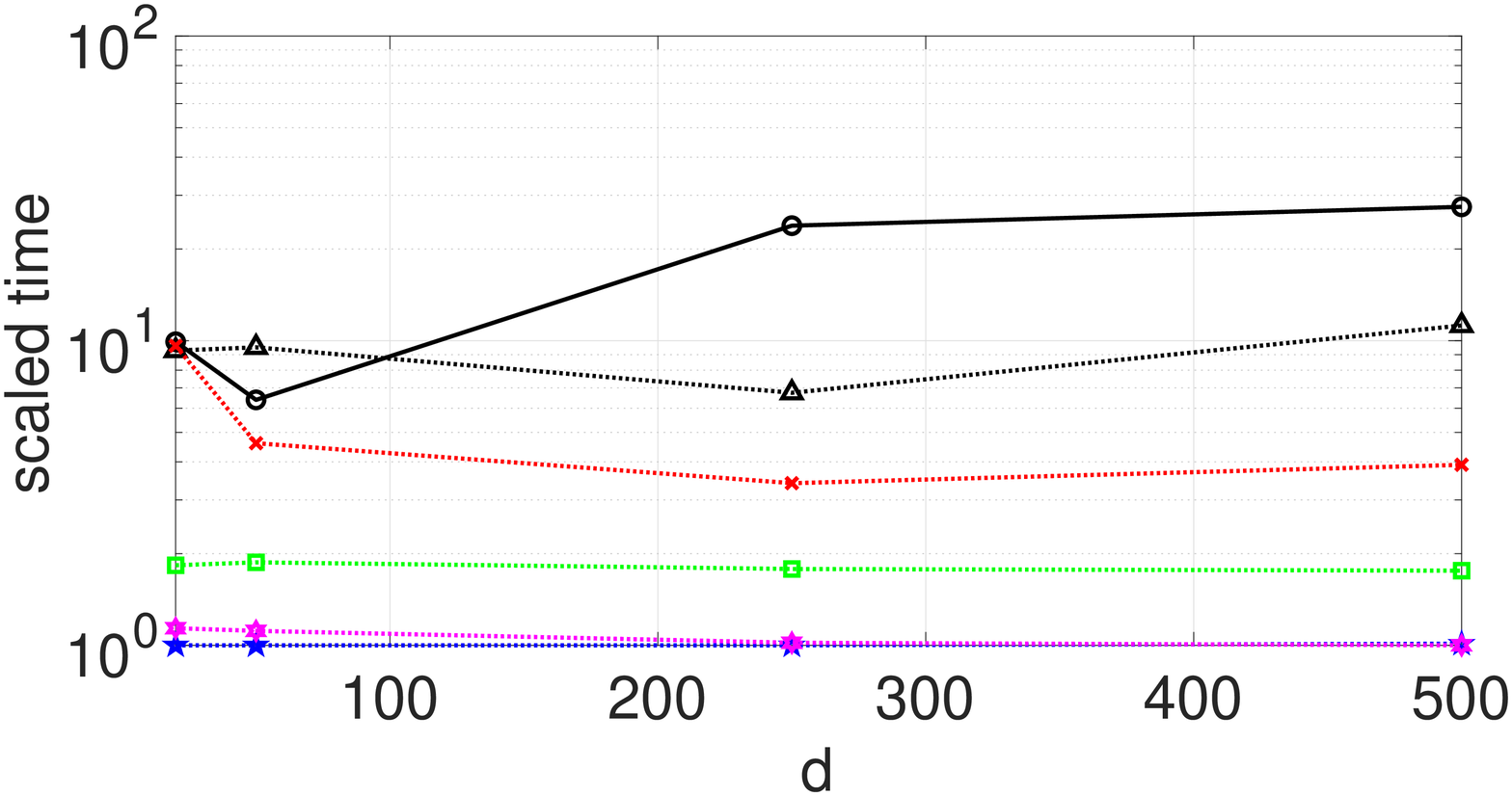}
\end{minipage}
}
\hspace{-3mm}
\subfloat[$n = 10d$]{
\begin{minipage}[b]{0.242\textwidth}
\includegraphics[width=1\linewidth]{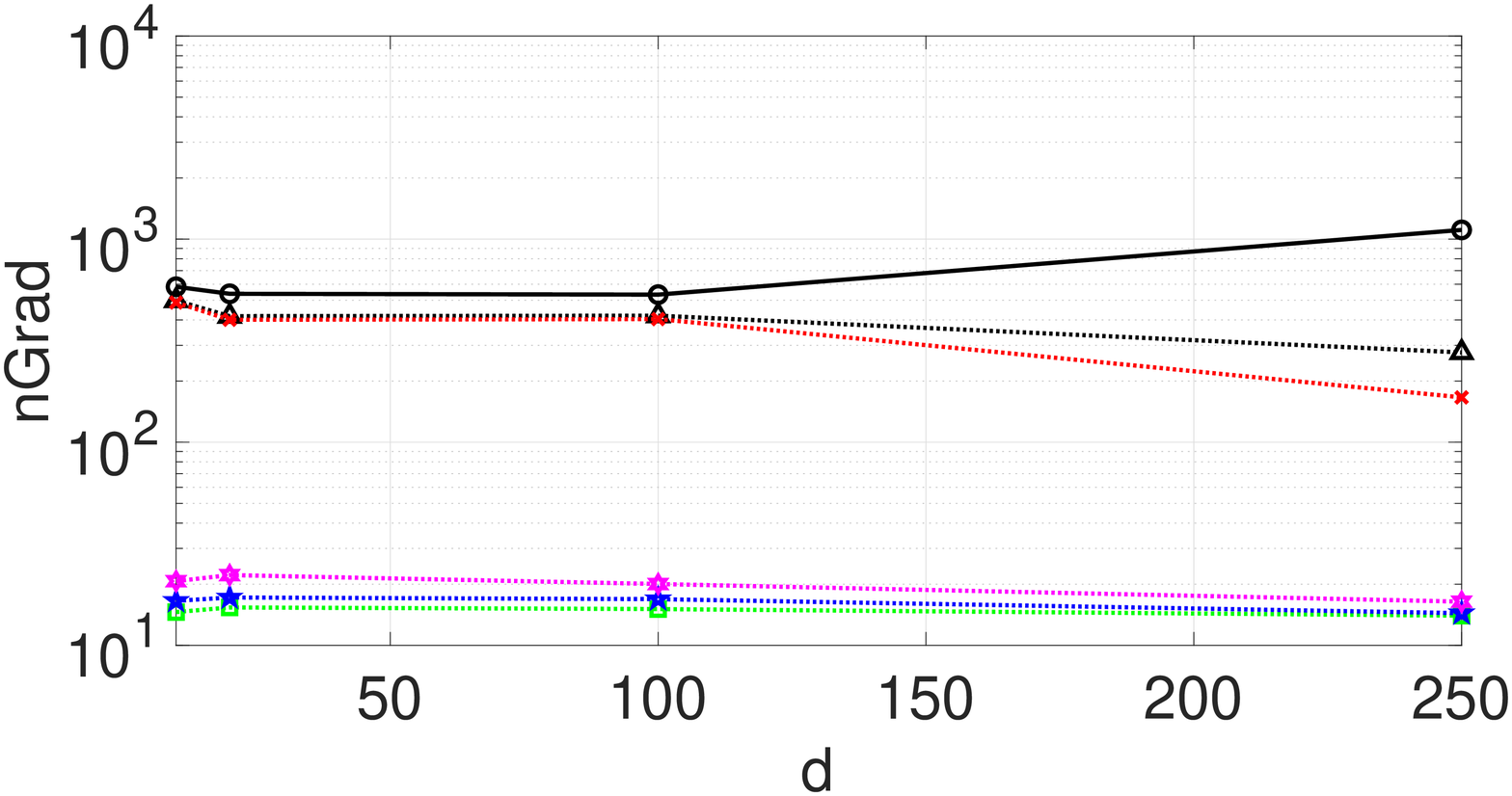}\\
\includegraphics[width=1\linewidth]{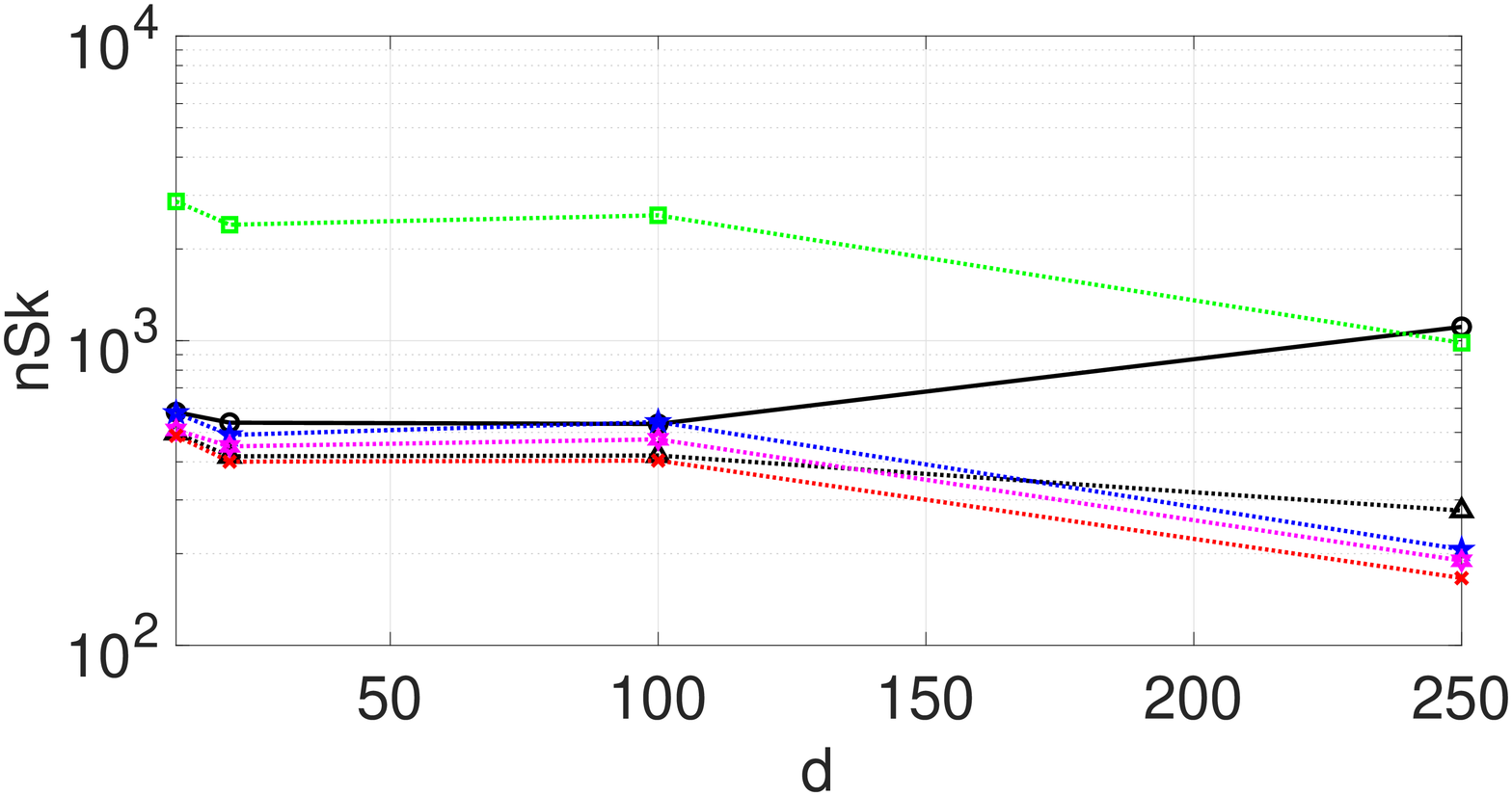}\\
\includegraphics[width=1\linewidth]{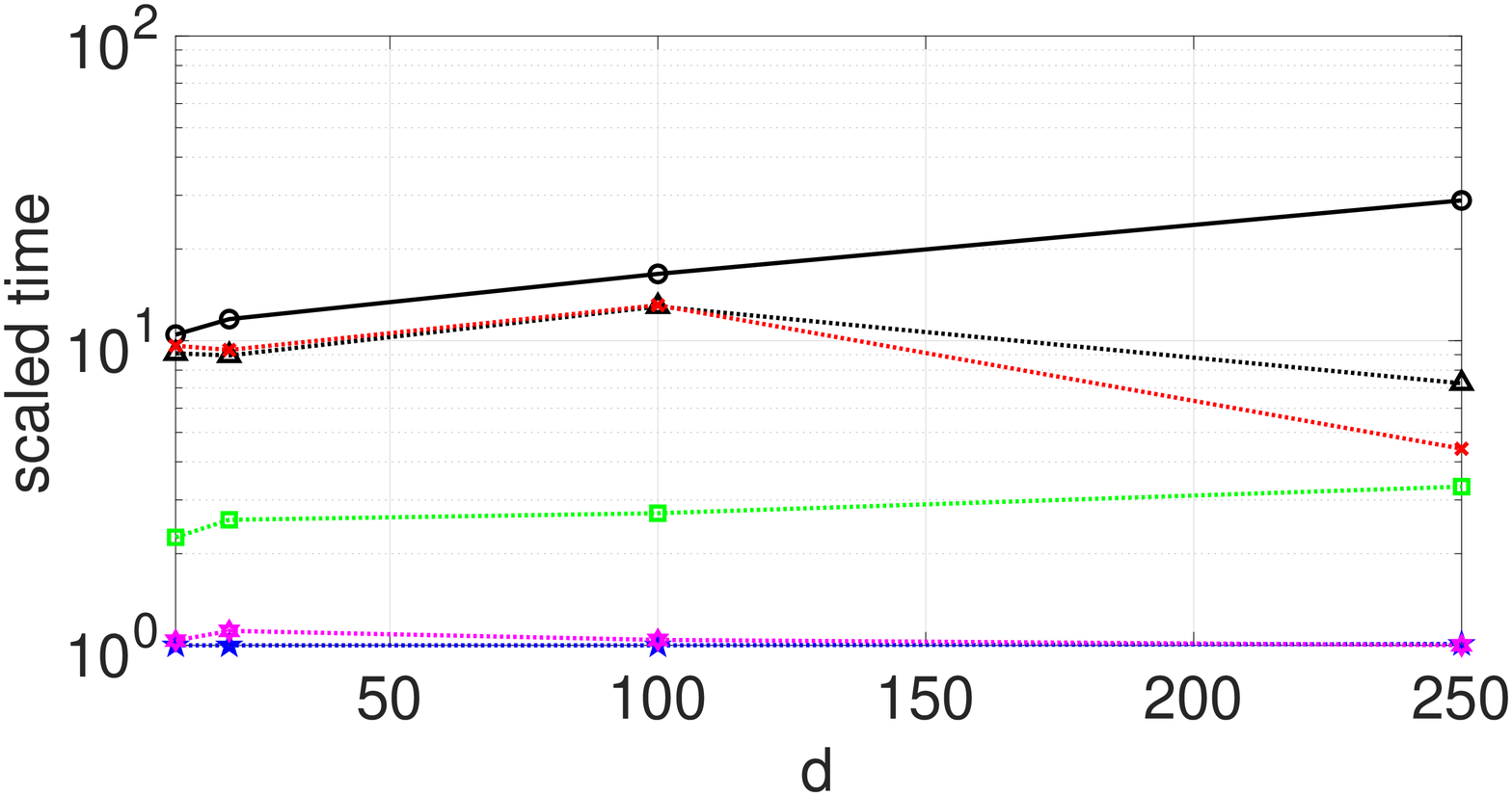}
\end{minipage}
}
\caption{Results for \cref{data:hypercube}. The scaled time means the time of each method over the minimum time of all methods.  For \emph{(a)}, $d \in \{20,100,250, 500\}$; for \emph{(b)}, $n \in \{100, 250, 500, 1000\}$;  for \emph{(c)}, $d \in \{20,50,250, 500\}$; for \emph{(d)}, $d \in \{10, 20, 100, 250\}$.  
} \label{fig:hypercube}
\end{figure}

The comparison results for  \cref{data:hypercube}   are plotted as  \cref{fig:hypercube}. Note that the values $\mathrm{PRW_p}$ returned by different methods are almost the same for this dataset. Therefore, we do not report the values in the figure. From the figure,  we can draw the following observations.  
i). \iRBBS~with  smaller $\theta$  always has less nGrad but  more nSk. This means that computing $\grad q(U^t)$ in a relatively high precision can help to reduce the whole iteration number of updating $U$.   In particular,  \iRBBS-0, which computes (almost) the exact  $\grad q(U^t)$,  takes the least nGrad.    Owing to that the complexity of one Sinkhorn iteration is much less than that of updating $U$ (see remarks at the end of \cref{subsection:iRBBSs}),   \iRBBS~with a moderate $\theta$ generally achieves the best overall performance.   The last rows of  \cref{fig:hypercube} show that \iRBBS-0.1 is almost the fastest among all \iRBBS~methods. 
ii).  Our \iRBBS-$\theta$ is better than \RABCDv~in terms that the former ones always have smaller nGrad and nSk and are always faster. Particularly,  for \cref{data:hypercube}, \iRBBS-0.1 is always more than 5x faster than \RABCD~and is about more than 10x  faster than \RBCD. For the largest instance $d = 250$, $n = 2500$, \iRBBS-0.1 (1.6s) is more than 7.2x faster than \RABCD~(11.1s) and about  28.6x faster than \RBCD~(44.4s).   It should be also mentioned that for all 800 problem instances, there are 5/21 instances in total for which \RBCD/\RABCD~meet the maximum iteration numbers and return solutions not satisfying the stopping criteria.

\begin{table}[!tbp]
\setlength{\tabcolsep}{1.8pt}
\centering
\scriptsize
 \caption{The average  results of 20 runs for  \cref{data:shakespeare}. In this table, ``a'', ``b'' stand for \RBCD~and \RABCD, respectively;   ``c'', ``d'', and ``e'' stand for \iRBBS-inf,  \iRBBS-0.1,  and \iRBBS-0, respectively.  For \RBCD, \RABCD, and \iRBBS-inf, \emph{nGrad} is always equal to \emph{nSk}. }\label{table:shakespeare:movie:iRBBSs}
\vspace{-2mm}
\begin{tabular}{@{}ccccccrrrrrrrrrr@{}}
\toprule
&\multicolumn{5}{c}{$\mathrm{PRW_p}$}  &  \multicolumn{5}{c}{nGrad/nSk} & \multicolumn{5}{c}{time} \\ 
     \cmidrule(l){2-6}   \cmidrule(l){7-11}   \cmidrule(l){12-16}  
data   & \multicolumn{1}{c}{a}& \multicolumn{1}{c}{b}&\multicolumn{1}{c}{c}&\multicolumn{1}{c}{d}&\multicolumn{1}{c}{e} & \multicolumn{1}{c}{a}  & \multicolumn{1}{c}{b}&\multicolumn{1}{c}{c}&\multicolumn{1}{c}{d}&\multicolumn{1}{c}{e} & \multicolumn{1}{c}{a}  & \multicolumn{1}{c}{b}&\multicolumn{1}{c}{c}&\multicolumn{1}{c}{d}&\multicolumn{1}{c}{e}\\ 
\midrule
H5/H& \textBF{0.0491}& 0.0491& 0.0491& 0.0491& 0.0491&   871 &  604 &  320& 64/1296& 61/3218& 9.6& 6.7& 3.4& \textBF{1.2}& 1.8 \\
H5/JC& 0.0596& 0.0596& \textBF{0.0596}& 0.0596& 0.0596&   904 &  497 &  193& 63/1571& 60/3047& 5.3& 2.9& 1.2& \textBF{0.6}& 0.7 \\
H5/MV& 0.0625& 0.0625& 0.0625& \textBF{0.0625}& 0.0625&   800 &  444 &  118& 63/1373& 61/2812& 5.4& 2.9& 0.8& \textBF{0.6}& 0.8 \\
H5/O& 0.0500& 0.0500& 0.0500& 0.0500& \textBF{0.0500}&   612 &  397 &  129& 43/1084& 45/2410& 4.9& 3.2& 1.0& \textBF{0.5}& 0.7 \\
H5/RJ& 0.1802& 0.1802& 0.1802& \textBF{0.1802}& 0.1802&   474 &  440 &  319& 73/1670& 69/3041& 3.8& 3.5& 2.4& \textBF{0.8}& 1.0 \\
H/JC& 0.0501& 0.0501& 0.0501& \textBF{0.0572}& 0.0572&   637 &  523 &  427& 41/1049& 42/2020& 4.4& 3.6& 2.9& \textBF{0.4}& 0.6 \\
H/MV& \textBF{0.0384}& 0.0384& 0.0384& 0.0384& 0.0384&  2794 & 1661 &  440& 117/2045& 125/5129& 22.9& 13.6& 3.4& \textBF{1.2}& 1.6 \\
H/O& 0.0140& \textBF{0.0140}& 0.0140& 0.0140& 0.0140&   796 &  503 &  212& 49/1471& 53/3093& 7.7& 4.7& 2.0& \textBF{0.8}& 1.2 \\
H/RJ& 0.1895& 0.1895& 0.1895& \textBF{0.1895}& 0.1895&   360 &  441 &  265& 46/933& 47/1763& 3.4& 4.2& 2.3& \textBF{0.6}& 0.8 \\
JC/MV& 0.0133& 0.0133& 0.0133& \textBF{0.0133}& 0.0133&  1816 &  975 &  216& 66/1781& 65/3742& 7.2& 3.8& 0.9& \textBF{0.4}& 0.6 \\
JC/O& 0.0098& 0.0098& \textBF{0.0098}& 0.0098& 0.0098&  1160 &  763 &  101& 52/1402& 45/2454& 6.5& 3.5& 0.5& \textBF{0.4}& 0.5 \\
JC/RJ& 0.1103& 0.1103& 0.1103& 0.1103& \textBF{0.1103}&   340 &  332 &  235& 44/1016& 53/2536& 1.5& 1.5& 1.1& \textBF{0.3}& 0.5 \\
MV/O& 0.0101& 0.0101& \textBF{0.0101}& 0.0101& 0.0101&  2222 & 1396 &  154& 90/2448& 88/5186& 13.2& 8.2& 0.9& \textBF{0.8}& 1.1 \\
MV/RJ& 0.1683& \textBF{0.1683}& 0.1683& 0.1683& 0.1683&   842 &  763 &  632& 85/1688& 95/4805& 5.0& 4.5& 3.8& \textBF{0.7}& 1.1 \\
O/RJ& 0.1256& 0.1256& \textBF{0.1256}& 0.1241& 0.1241&   376 &  372 &  210& 55/1070& 62/2749& 2.5& 2.4& 1.4& \textBF{0.5}& 0.8 \\
 \cmidrule(l){2-16} 
AVG& 0.0754& 0.0754& 0.0754& \textBF{0.0758}& 0.0758&  1000 &  674 &  265& 63/1460& 65/3200& 6.89& 4.62& 1.86& \textBF{0.66}& 0.90\\   
\bottomrule
\end{tabular}
\end{table}
\begin{table}[!htbp]
\setlength{\tabcolsep}{4pt}
\centering
\scriptsize
 \caption{The average  results of 20 runs for \cref{data:mnist}. In this table, ``a'', ``c'', ``d'',  and ``e'' stand for \RBCD, \iRBBS-inf,  \iRBBS-0.1,  and \iRBBS-0, respectively.  For \RBCD~and \iRBBS-inf, \emph{nGrad} is always equal to \emph{nSk}. }\label{table:mnist:iRBBSs}
\vspace{-2mm}
\begin{tabular}{@{}cccccrrrrrrrr@{}}
\toprule
&\multicolumn{4}{c}{$10^{-3}\times \mathrm{PRW_p}$}  &  \multicolumn{4}{c}{nGrad/nSk} & \multicolumn{4}{c}{time} \\ 
     \cmidrule(l){2-5}   \cmidrule(l){6-9}   \cmidrule(l){10-13}  
data   & \multicolumn{1}{c}{a}&\multicolumn{1}{c}{c}&\multicolumn{1}{c}{d}&\multicolumn{1}{c}{e} & \multicolumn{1}{r}{a}   &\multicolumn{1}{r}{c}&\multicolumn{1}{r}{d}&\multicolumn{1}{r}{e} & \multicolumn{1}{r}{a}  &\multicolumn{1}{r}{c}&\multicolumn{1}{r}{d}&\multicolumn{1}{r}{e}\\ 
\midrule
D0/D1& 0.9746& 0.9746& 0.9746& \textBF{0.9746}&   519 &  132& 59/506& 46/1062& 2.7& 0.7& \textBF{0.4}& 0.4 \\
D0/D2& 0.7942& \textBF{0.7942}& 0.7942& 0.7942&   988 &  187& 72/943& 74/2655& 4.8& 1.0& \textBF{0.5}& 0.6 \\
D0/D3& 1.2021& 1.2021& 1.2021& \textBF{1.2021}&   554 &  177& 64/934& 44/1386& 2.6& 0.8& 0.4& \textBF{0.4} \\
D0/D4& 1.2205& 1.2205& \textBF{1.2308}& 1.2308&   912 &  243& 98/998& 81/2947& 4.1& 1.1& \textBF{0.5}& 0.6 \\
D0/D5& 1.0266& \textBF{1.0276}& 1.0266& 1.0266&  1178 &  252& 89/1402& 96/4355& 4.8& 1.1& \textBF{0.5}& 0.8 \\
D0/D6& 0.8027& 0.8027& 0.8027& \textBF{0.8027}&   919 &  186& 66/1052& 54/2387& 4.0& 0.8& \textBF{0.4}& 0.5 \\
D0/D7& 0.8558& 0.8558& \textBF{0.8558}& 0.8558&   601 &  163& 63/747& 56/1727& 2.7& 0.8& \textBF{0.4}& 0.4 \\
D0/D8& 1.0529& \textBF{1.0529}& 1.0529& 1.0529&   709 &  185& 84/1330& 72/3302& 3.1& 0.8& \textBF{0.5}& 0.6 \\
D0/D9& 1.0846& \textBF{1.0846}& 1.0846& 1.0846&   698 &  179& 68/1108& 55/2436& 3.3& 0.8& \textBF{0.4}& 0.5 \\
D1/D2& 0.6647& \textBF{0.6647}& 0.6647& 0.6647&   757 &  114& 55/638& 47/1511& 4.1& 0.6& \textBF{0.4}& 0.5 \\
D1/D3& 0.8610& \textBF{0.8610}& 0.8610& 0.8610&  1649 &  230& 105/982& 71/2065& 8.8& 1.3& \textBF{0.7}& 0.8 \\
D1/D4& 0.6666& 0.6666& 0.6666& \textBF{0.6666}&   982 &  153& 65/505& 54/1282& 5.1& 0.8& \textBF{0.4}& 0.5 \\
D1/D5& \textBF{0.8400}& 0.8400& 0.8400& 0.8400&  3436 &  540& 184/1503& 157/3844& 15.9& 2.4& \textBF{1.0}& 1.1 \\
D1/D6& 0.7957& 0.7957& \textBF{0.7957}& 0.7954&  2919 &  331& 128/1095& 119/3354& 15.0& 1.7& \textBF{0.8}& 1.1 \\
D1/D7& \textBF{0.5723}& 0.5723& 0.5723& 0.5723&   961 &  157& 65/525& 56/1253& 5.2& 0.9& \textBF{0.4}& 0.5 \\
D1/D8& 0.8784& \textBF{0.8784}& 0.8784& 0.8783&  5000 &  402& 200/2001& 181/5505& 25.8& 2.1& \textBF{1.3}& 1.6 \\
D1/D9& 0.8534& 0.8534& \textBF{0.8534}& 0.8534&   951 &  168& 69/581& 58/1321& 5.6& 1.0& \textBF{0.6}& 0.7 \\
D2/D3& \textBF{0.7188}& 0.7188& 0.7188& 0.7188&  2070 &  336& 120/1733& 100/4897& 10.0& 1.6& \textBF{0.8}& 1.0 \\
D2/D4& 1.0684& 1.0854& \textBF{1.0854}& 1.0854&  5000 & 4602& 550/8093& 382/16291& 24.3& 21.4& \textBF{3.2}& 3.4 \\
D2/D5& 1.0767& 1.0767& \textBF{1.0767}& 1.0767&  1320 &  272& 109/1910& 89/3990& 5.9& 1.2& \textBF{0.6}& 0.7 \\
D2/D6& \textBF{0.8993}& 0.8993& 0.8993& 0.8993&  1202 &  245& 102/967& 87/3205& 5.8& 1.2& \textBF{0.6}& 0.7 \\
D2/D7& 0.6950& \textBF{0.6950}& 0.6950& 0.6950&   951 &  174& 67/1156& 63/3108& 4.7& 0.9& \textBF{0.5}& 0.8 \\
D2/D8& 0.6711& 0.6711& 0.6711& \textBF{0.6711}&  1293 &  213& 72/925& 69/3547& 6.2& 1.0& \textBF{0.4}& 0.7 \\
D2/D9& 1.0567& \textBF{1.0567}& 1.0567& 1.0567&  1985 &  284& 120/1907& 93/3753& 9.4& 1.3& \textBF{0.7}& 0.8 \\
D3/D4& 1.1994& 1.2003& \textBF{1.2003}& 1.2003&  5000 &  312& 76/762& 69/2523& 22.2& 1.4& \textBF{0.4}& 0.5 \\
D3/D5& 0.5817& \textBF{0.5817}& 0.5817& 0.5817&  1187 &  181& 71/793& 58/2430& 4.8& 0.8& \textBF{0.4}& 0.5 \\
D3/D6& 1.2317& \textBF{1.2317}& 1.2317& 1.2317&  1028 &  212& 82/772& 66/2121& 4.5& 0.9& \textBF{0.4}& 0.5 \\
D3/D7& \textBF{0.7245}& 0.7245& 0.7245& 0.7245&  1101 &  202& 88/831& 776/19184& 5.2& 0.9& \textBF{0.5}& 5.3 \\
D3/D8& 0.8798& 0.8798& \textBF{0.8798}& 0.8798&   916 &  200& 84/1113& 77/3675& 4.0& 0.9& \textBF{0.5}& 0.7 \\
D3/D9& \textBF{0.8302}& 0.8302& 0.8302& 0.8302&  1572 &  210& 72/955& 49/2193& 7.2& 1.0& \textBF{0.4}& 0.4 \\
D4/D5& 1.0063& 1.0063& \textBF{1.0063}& 1.0063&   936 &  245& 87/863& 71/2748& 3.9& 1.0& \textBF{0.4}& 0.5 \\
D4/D6& \textBF{0.8436}& 0.8436& 0.8436& 0.8436&  1312 &  245& 96/1102& 82/3463& 5.9& 1.1& \textBF{0.5}& 0.7 \\
D4/D7& 0.7899& \textBF{0.7899}& 0.7899& 0.7899&  1351 &  269& 96/912& 76/2509& 6.4& 1.2& \textBF{0.5}& 0.6 \\
D4/D8& 1.0950& 1.0982& 1.0982& \textBF{1.0982}&  5000 &  213& 52/550& 52/1930& 21.9& 0.9& \textBF{0.3}& 0.4 \\
D4/D9& 0.4892& 0.4892& \textBF{0.4892}& 0.4892&  1997 &  281& 70/1116& 54/2610& 9.1& 1.3& \textBF{0.4}& 0.5 \\
D5/D6& \textBF{0.7170}& 0.7170& 0.7170& 0.7170&  1335 &  220& 86/891& 69/2536& 5.4& 0.9& \textBF{0.4}& 0.5 \\
D5/D7& \textBF{0.9126}& 0.9126& 0.9126& 0.9126&   899 &  242& 72/652& 71/2132& 3.7& 1.0& \textBF{0.4}& 0.5 \\
D5/D8& 0.7152& 0.7152& \textBF{0.7152}& 0.7152&  1429 &  225& 71/1060& 50/2347& 5.9& 0.9& \textBF{0.4}& 0.4 \\
D5/D9& 0.7755& 0.7755& 0.7755& \textBF{0.7755}&  1413 &  252& 86/1216& 64/3451& 5.9& 1.0& \textBF{0.5}& 0.6 \\
D6/D7& 1.1098& \textBF{1.1098}& 1.1098& 1.1098&   826 &  195& 66/398& 62/1318& 3.8& 0.9& \textBF{0.3}& 0.4 \\
D6/D8& 0.9175& 0.9175& \textBF{0.9175}& 0.9175&   725 &  186& 57/721& 57/2305& 3.2& 0.8& \textBF{0.3}& 0.5 \\
D6/D9& 1.1055& \textBF{1.1055}& 1.1055& 1.1055&  1117 &  219& 95/1038& 79/3046& 5.0& 1.0& \textBF{0.5}& 0.6 \\
D7/D8& 1.0787& \textBF{1.0787}& 1.0787& 1.0787&   757 &  180& 68/796& 59/2455& 3.5& 0.8& \textBF{0.4}& 0.5 \\
D7/D9& \textBF{0.6079}& 0.6079& 0.6079& 0.6079&  1985 &  266& 79/1026& 72/3208& 9.3& 1.3& \textBF{0.5}& 0.7 \\
D8/D9& 0.8678& 0.8678& \textBF{0.8678}& 0.8678&   745 &  180& 57/713& 50/2107& 3.2& 0.8& \textBF{0.3}& 0.4 \\
 \cmidrule(l){2-13} 
AVG& 0.8847& 0.8852& \textBF{0.8854}& 0.8854&  1560 &  326& 95/1152& 95/3366& 7.28& 1.52& \textBF{0.57}& 0.78\\ 
\bottomrule
\end{tabular}
\end{table}

The comparison results for  \cref{data:shakespeare,data:mnist} are reported in \cref{table:shakespeare:movie:iRBBSs,table:mnist:iRBBSs}, respectively.
In the last ``AVG'' line,  we summarize the averaged   ``$\mathrm{PRW_p}$'', ``nGrad/nSk'' and  ``time''. 
  From the tables, we have the following observations.  
i). As for the solution quality,  \iRBBS-0 and \iRBBS-0.1 perform the best among all methods.
 ii). Among \iRBBS~methods, \iRBBS-0.1 and \iRBBS-0 take the least nGrad and \iRBBS-inf takes the most nGrad while \iRBBS-0 takes the most nSk and \iRBBS-inf takes the least nSk.
  Compared with \iRBBS-inf, \iRBBS-0.1 is about 2.8x faster for  \cref{data:shakespeare} and is about 2.7x faster for \cref{data:mnist}. Compared with \iRBBS-0, \iRBBS-0.1 is \rev{about} 1.4x faster for \cref{data:shakespeare,data:mnist}. iii). Compared with \RABCDv, \iRBBS-0.1 can always take significantly less nGrad and may take a bit more nSk. This makes it about 7x faster than \RABCD~and \rev{about} 10.4x faster than \RBCD~for  \cref{data:shakespeare}, and about 12.8x  faster than \RBCD~for \cref{data:mnist}. Besides, for instances D1/D8, D2/D4, D3/D4, D4/D8, \RBCD~meets the maximum iteration number.

We now further investigate the benefit of performing several Sinkhorn iterations over only performing one Sinkhorn iteration in each update of the variable $U$.  The results of \iRBBS-inf and \iRBBS-0.1 for different values of $\eta$, as well as the results of \RABCD~with $\eta = 0.1$ for \cref{data:shakespeare}  and \rev{\RBCD} with $\eta = 8$ for \cref{data:mnist},  are plotted as  \cref{fig:dynamic:eta:iRBBSs}.  From this figure, we can make the following observations. 
i) Our method \iRBBS~with different $\theta$ for solving \eqref{prob:sub:RALMexp} with a smaller $\eta$ is still much better  than \RABCDv~for solving \eqref{prob:sub:RALMexp} with a larger $\eta$ in terms of both speed and solution quality. Specifically, for $\eta = 0.05$, our \iRBBS-0.1 is still about  4.6x faster for \cref{data:shakespeare} and the returned $\mathrm{PRW}_p$ (0.0918) is 1.2x better than that of  \RABCD~(0.0754).  For $\eta = 3$, our \iRBBS-0.1 is still about 11.5x faster than \RBCD~for \cref{data:mnist},  and the returned $\mathrm{PRW}_p$ (886.1) is  better than that of  \RBCD~(884.7). 
ii) Generally, problem  \eqref{prob:sub:RALMexp} with a smaller $\eta$ can return better solutions but is  harder to solve.  The strategy of doing several Sinkhorn iterations  while updating $U$ once can  indeed  speed up the algorithm  
for different $\eta$ and the advantage of  \iRBBS-0.1 over  \iRBBS-inf  is larger for smaller $\eta$.  
\begin{figure}[!tbp]
\centering
\subfloat{ 
\includegraphics[width=0.242\linewidth,height=0.121\linewidth]{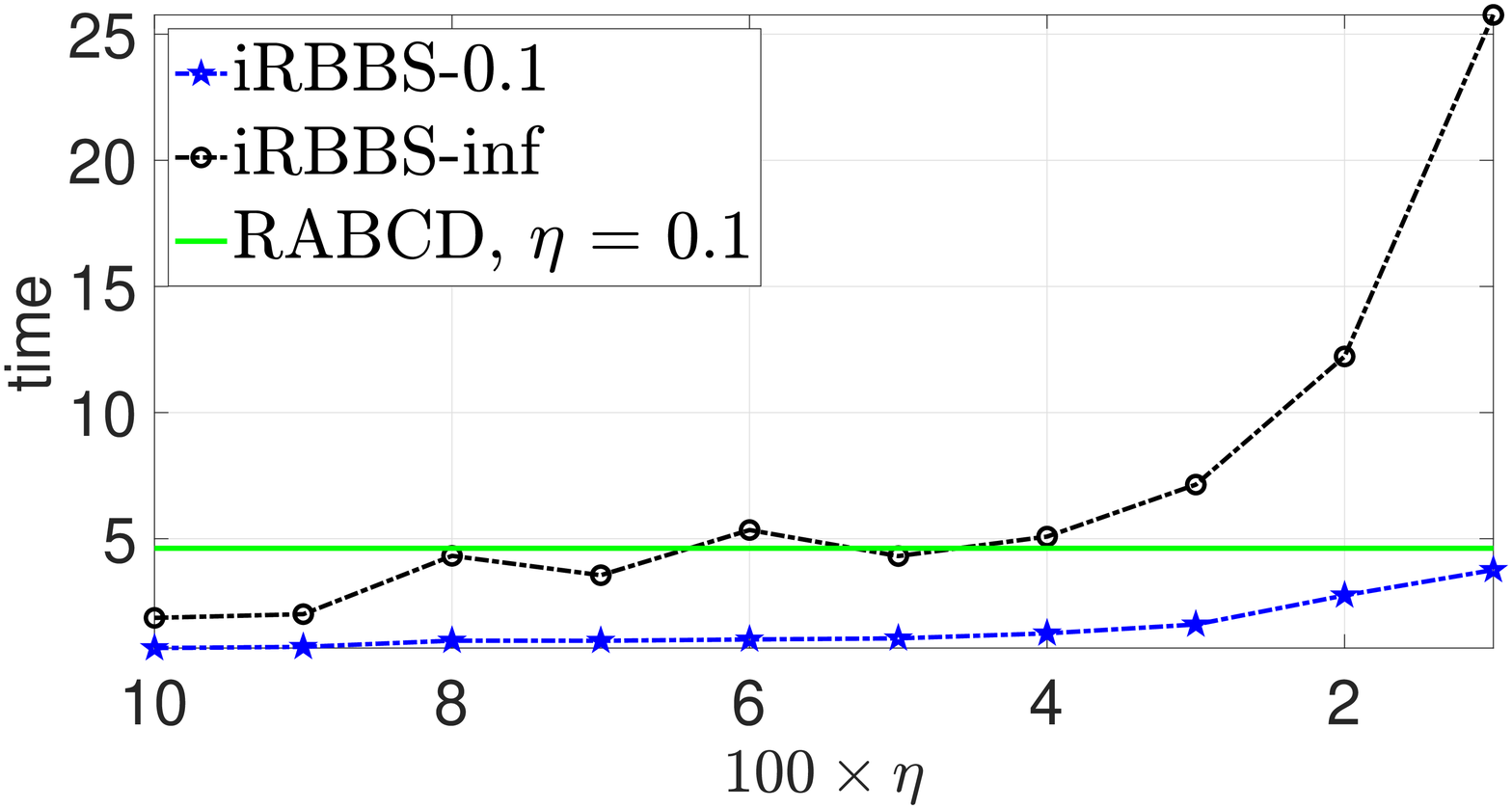}
\hspace{-2mm}
\includegraphics[width=0.242\linewidth,height=0.121\linewidth]{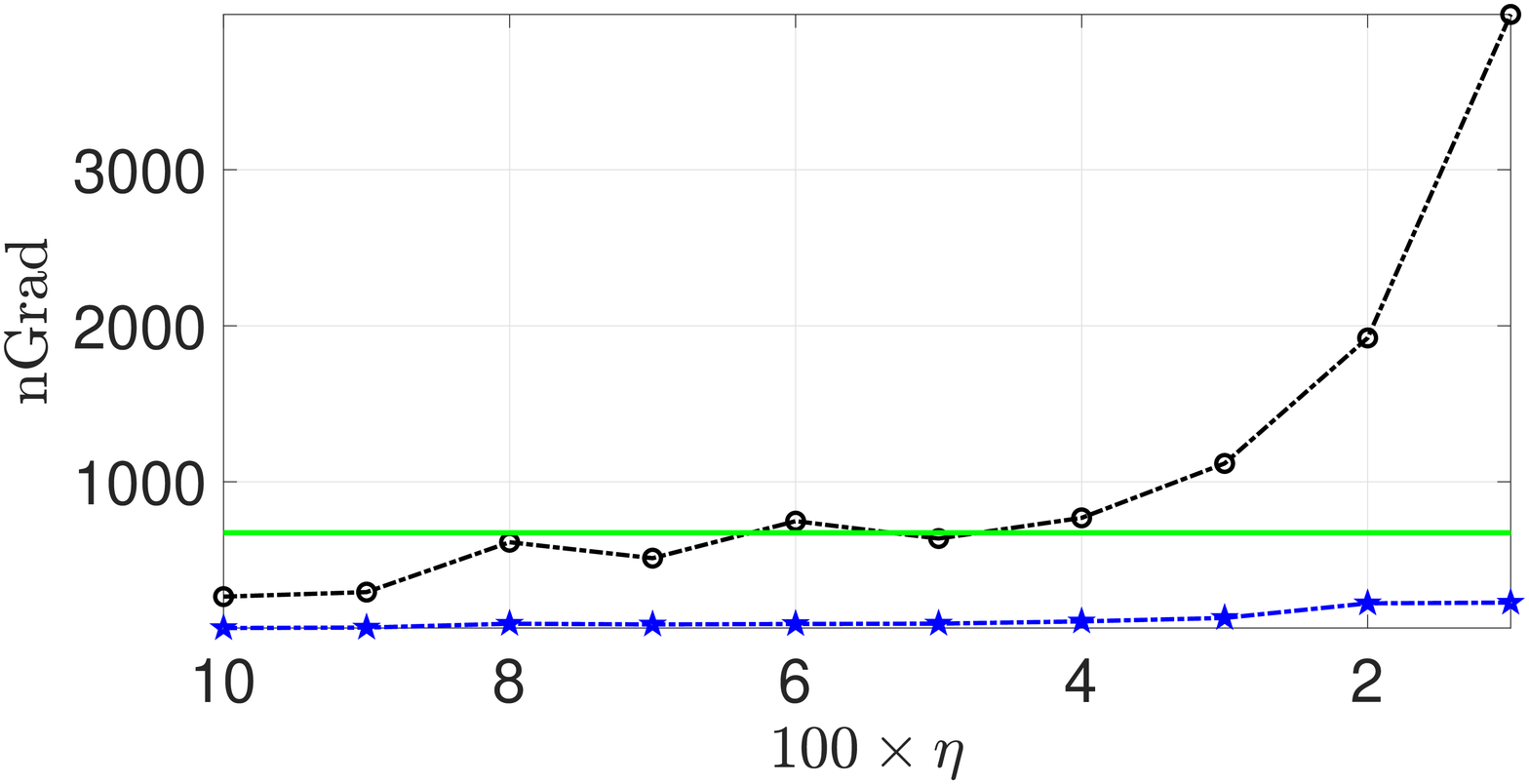}
\hspace{-2mm}
\includegraphics[width=0.242\linewidth,height=0.121\linewidth]{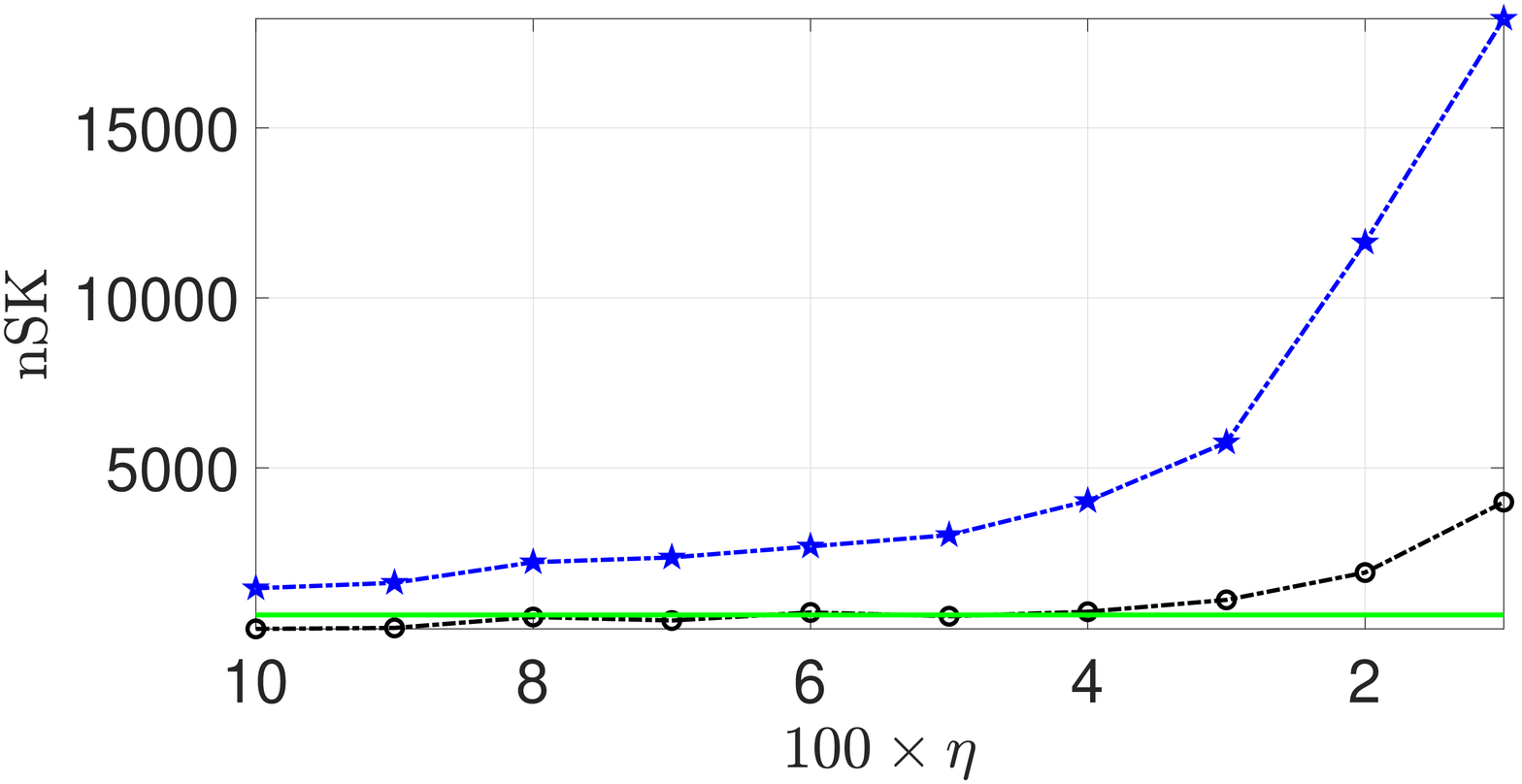}
\hspace{-2mm}
\includegraphics[width=0.242\linewidth,height=0.121\linewidth]{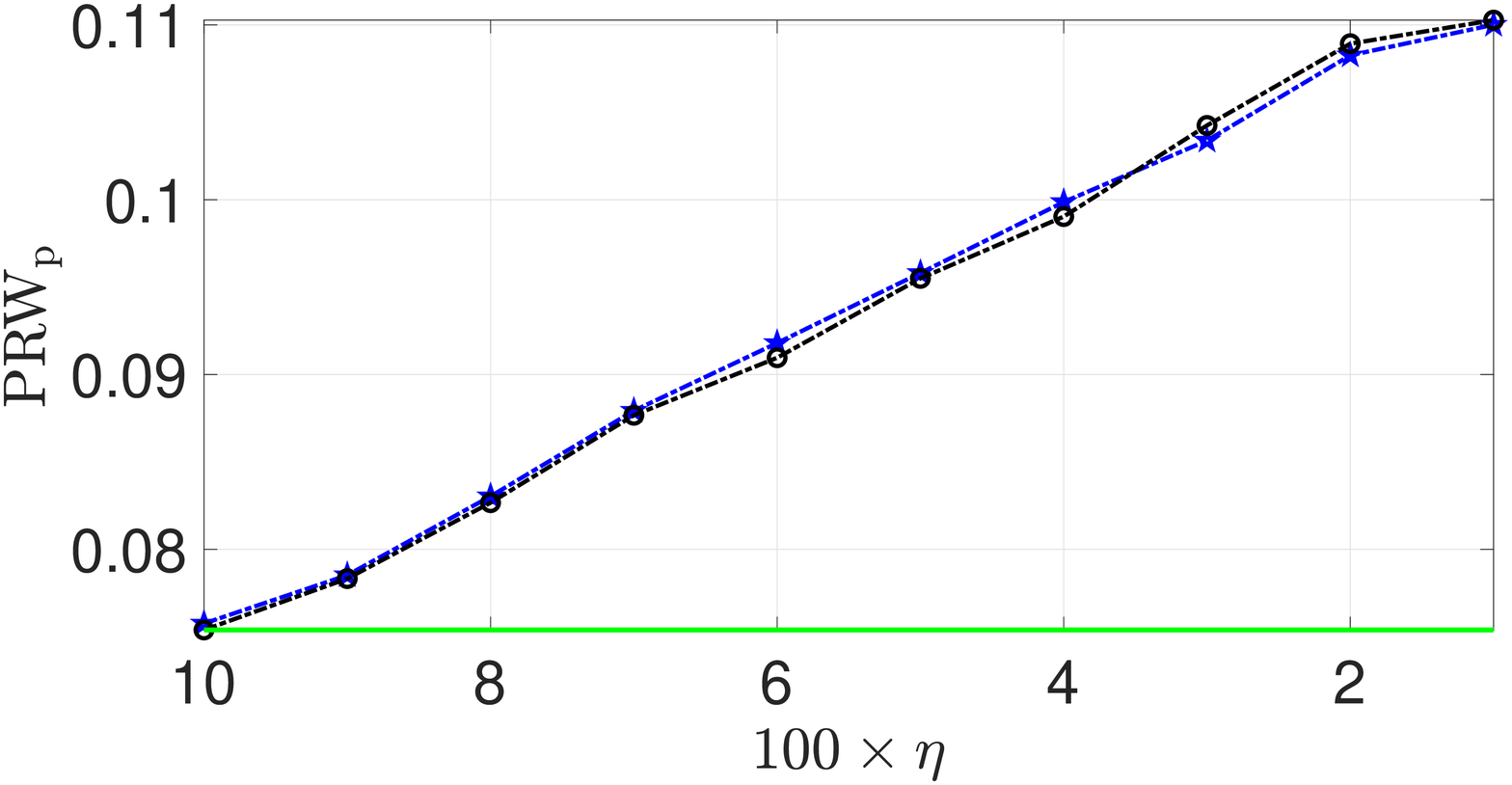}
}
\vspace{-3.5mm}
\hfill
\subfloat{ 
\includegraphics[width=0.242\linewidth,height=0.121\linewidth]{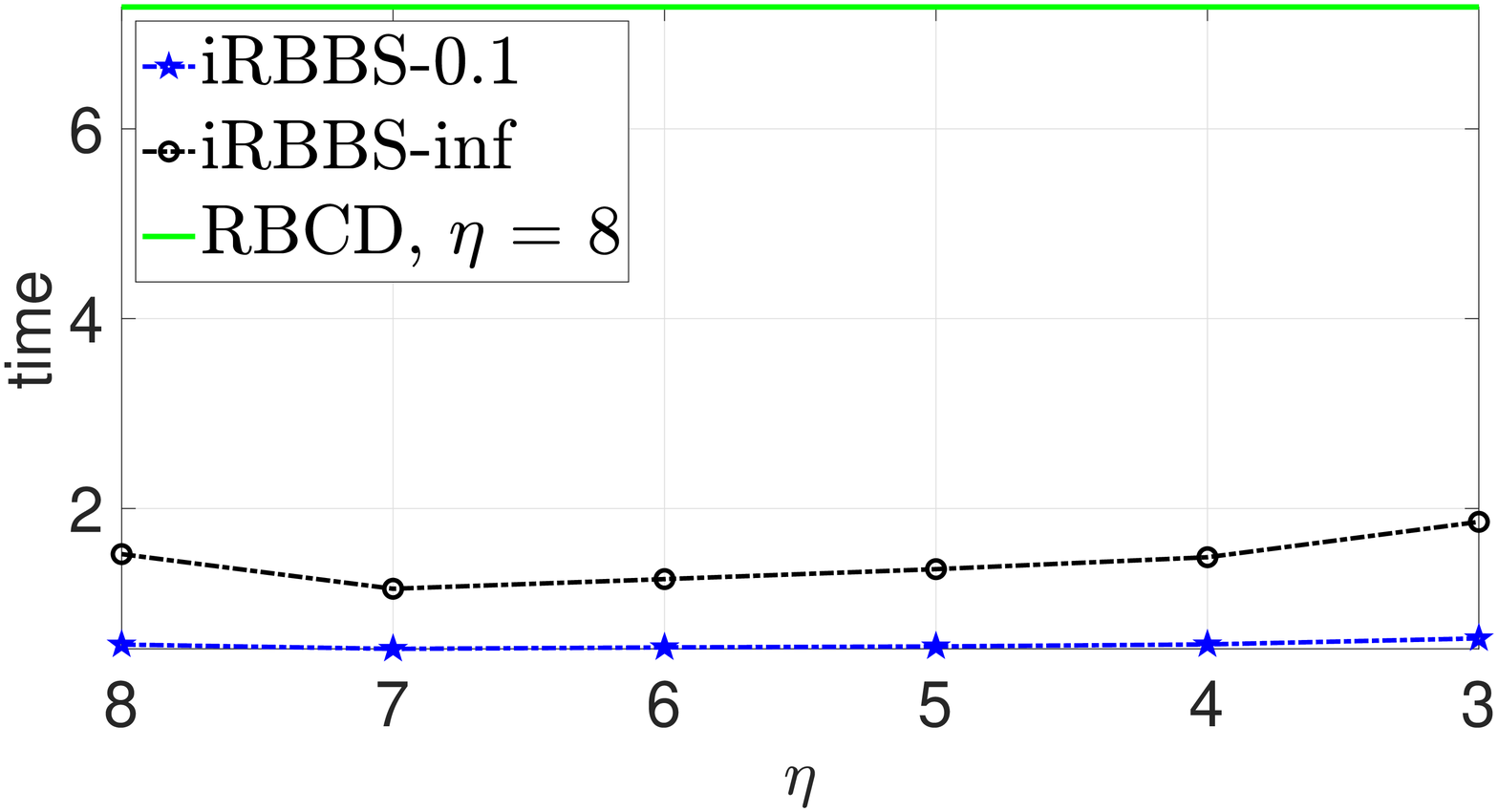}
\hspace{-2mm}
\includegraphics[width=0.242\linewidth,height=0.121\linewidth]{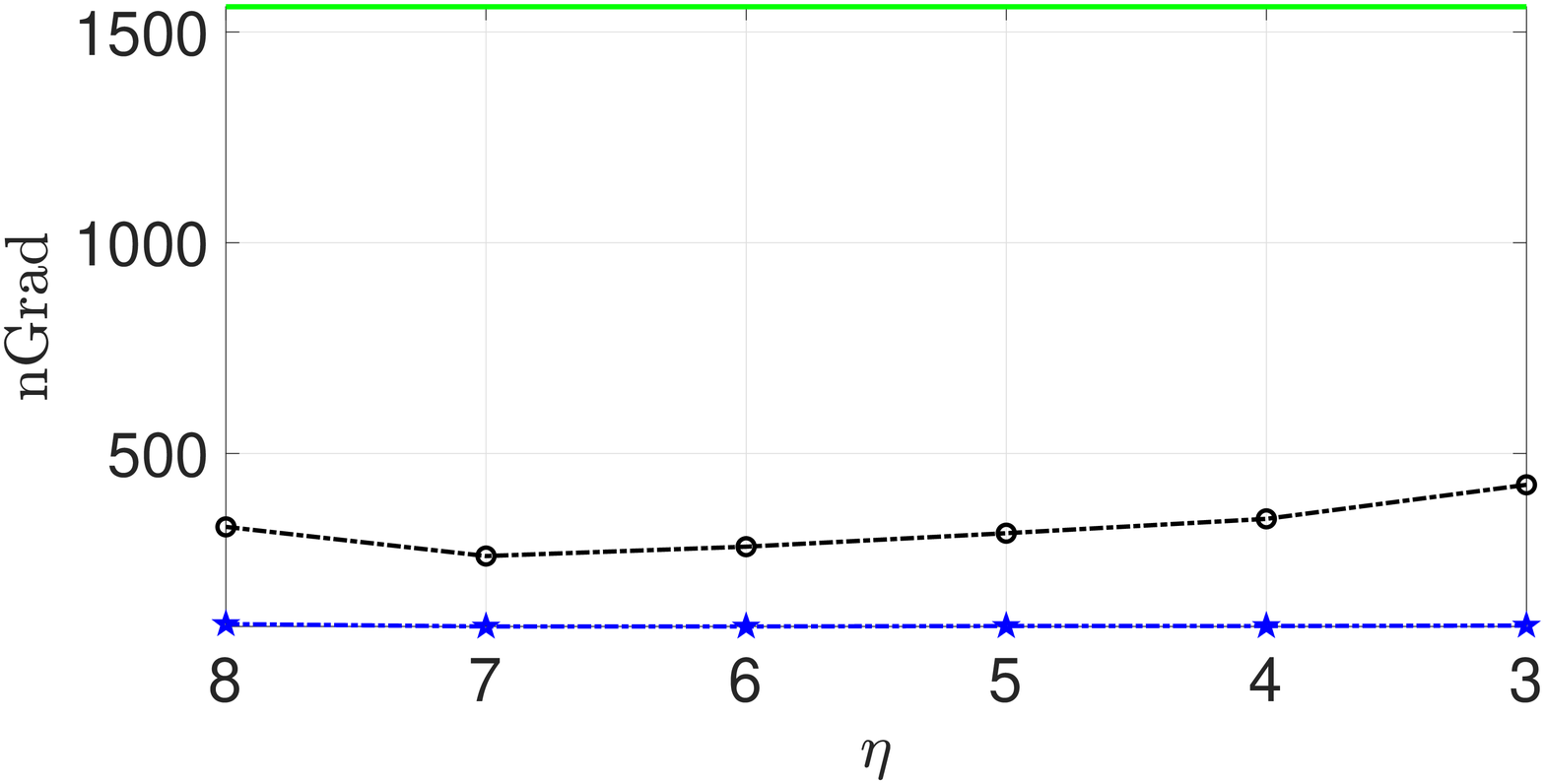}
\hspace{-2mm}
\includegraphics[width=0.242\linewidth,height=0.121\linewidth]{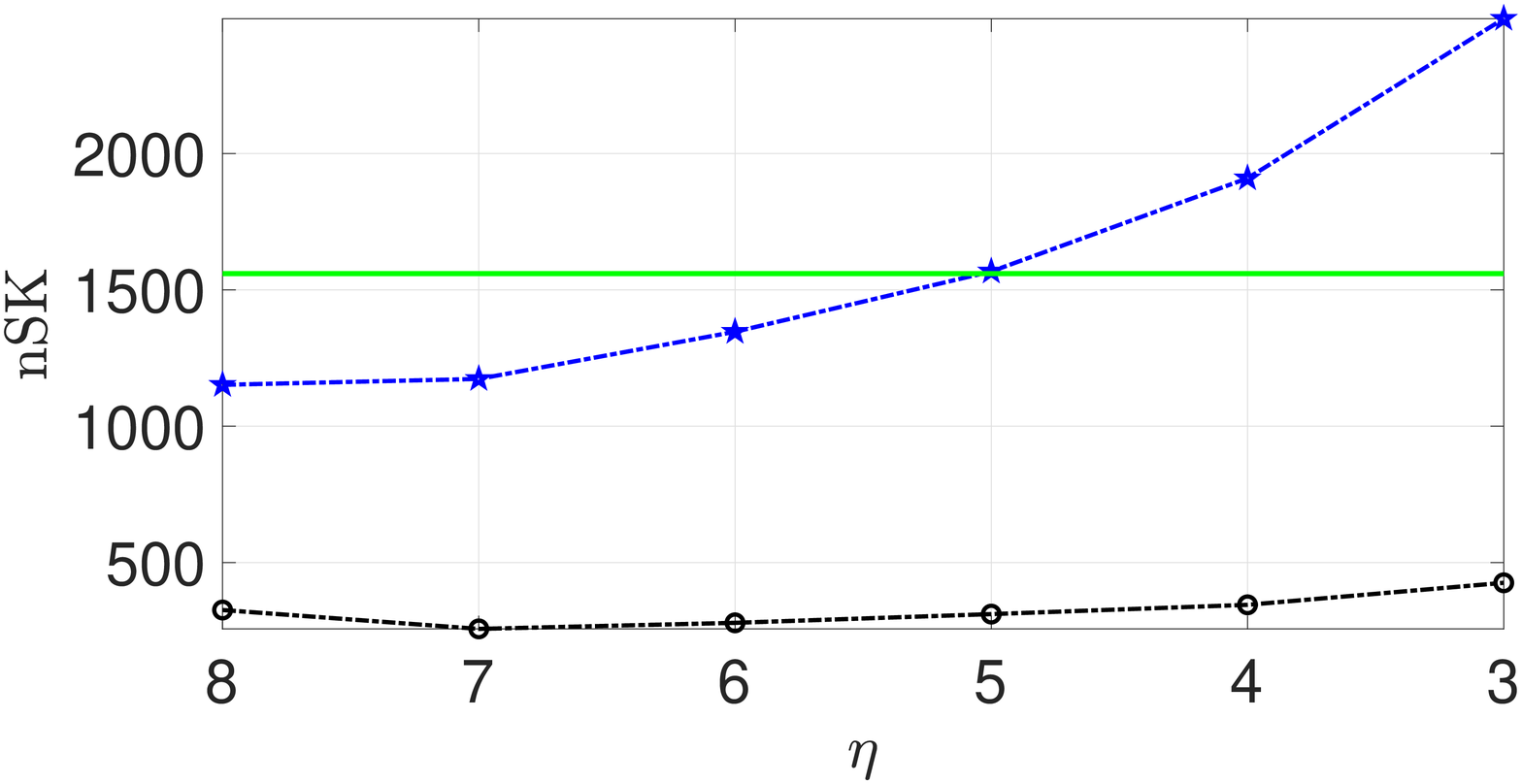}
\hspace{-2mm}
\includegraphics[width=0.242\linewidth,height=0.121\linewidth]{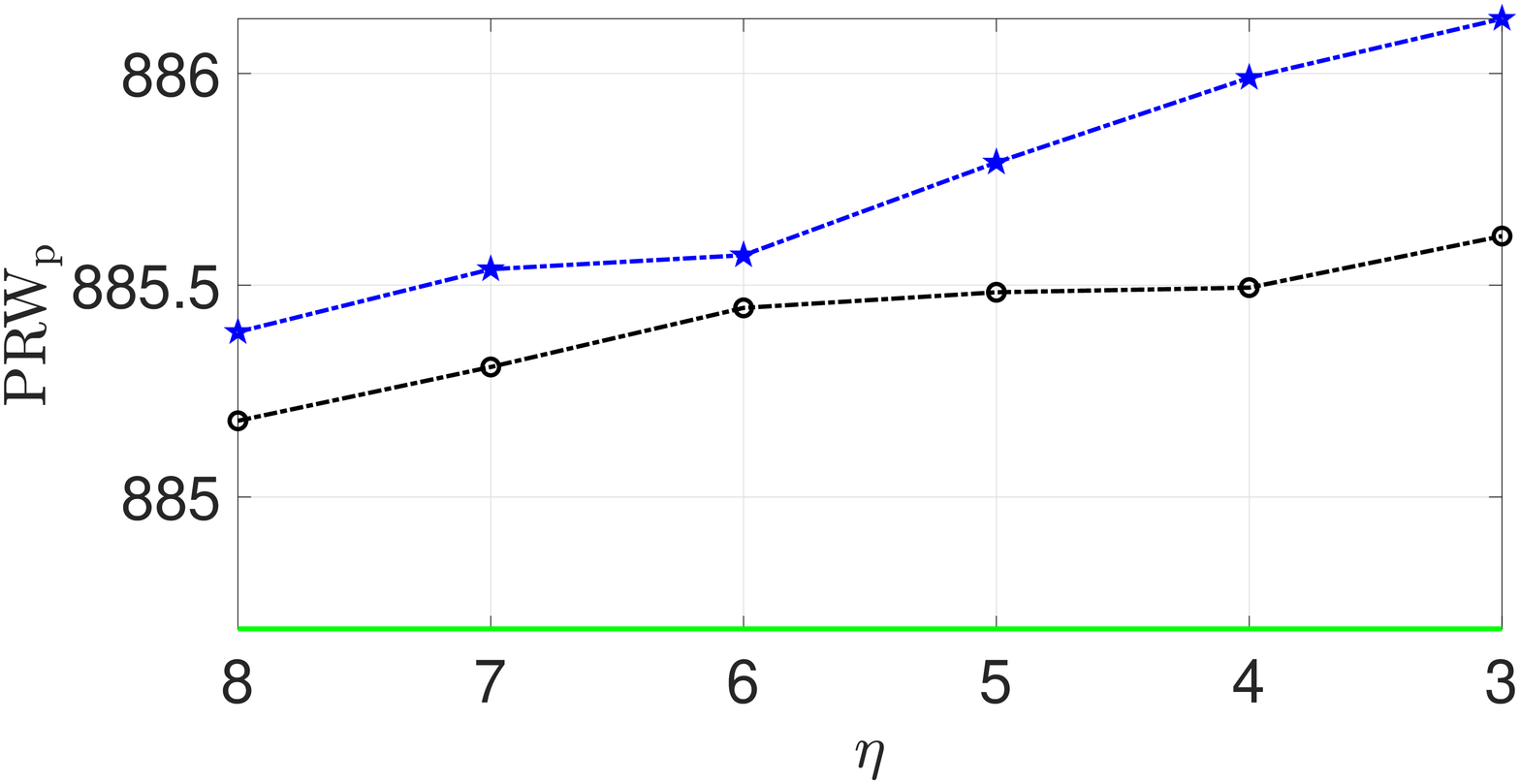}
}
\caption{Comparison of \iRBBS-inf and \iRBBS-0.1 for different $\eta$. The top four figures are for \cref{data:shakespeare} and the bottom four figures are for \cref{data:mnist}.}\label{fig:dynamic:eta:iRBBSs}
\end{figure}

From the above results,  we can conclude that \iRBBS~generally perform much better than \RABCDv~for the three datasets.  More importantly, our methods adopt the adaptive stepsize without needing to tune the best stepsize as done in \RABCDv.

\subsection{Comparison on computing the PRW distance \eqref{equ:PRW}}\label{subsec:num:PRW}
In this subsection, we present numerical results to illustrate the effectiveness and efficiency of our proposed \ReALM~algorithm, namely, \cref{alg:ReALM}.  The subproblem \eqref{prob:sub:RALMexp:0:1}  is solved by our developed \iRBBS~algorithm, namely,  \cref{alg:iRBBSs}, with $\theta_t$ chosen via \eqref{equ:theta:choice:numerical}.

{\it Parameters of \ReALM}.  In our numerical tests, we choose  $\epsilon_1  = 2 \|C\|_{\infty} \epsilon_2$,  $\epsilon_2 = 10^{-6} \max\{\|r\|_{\infty}, \|c\|_{\infty}\}$, and $\epsilon_c = 10^{-5}$. To prevent too small $\eta$, in our implementation, we update $\eta_{k+1} = \min\left\{\gamma_{\eta} \eta_{k}, \eta_{\min}\right\}$ other than using \eqref{equ:eta_k:update}, where $\eta_{\min}$ is a preset positive number. Once $\eta_k$ becomes $\eta_{\min}$, we set the corresponding $\epsilon_{k,1} = \epsilon_1$ and $\epsilon_{k,2} = \epsilon_2$ and stop the algorithm if 
$\err_1(\x^k,\pi^k)\leq \epsilon_1$ and $\err_2(\x^k,\pi^k)\leq \epsilon_2$.  To avoid possible numerical instability, we restrict the maximum number of updating $\pi^k$ as 8. We set   $\gamma_{\epsilon} = 0.25$ and choose   $\epsilon_{1,2} = 10^{-1} \max\{\|r\|_{\infty}, \|c\|_{\infty}\}$ and $\epsilon_{1,1}  = 2 \|C\|_{\infty} \epsilon_{1,2}$. The initial point is chosen according to the way in \cref{subsec:num:subprob}.   Moreover, we denote by \ReALM-$(\eta_{\min},\gamma_W)$  the \ReALM~algorithm with particular parameters $\eta_{\min}$ and $\gamma_W$. Note that choosing $\gamma_W = 0$ means that we adopt a continuation technique to solve the problem \eqref{prob:sub:RALMexp:0:1} with  $\pi^k \equiv \pi^1$ and $\eta_k = \eta_{\min}$, which is generally better than directly solving a single problem \eqref{prob:sub:RALMexp:0:1} with a small $\eta_k = \eta_{\min}$. 

{\it Parameters of \iRBBS~in  \ReALM}. 
If  $\max\{\|r\|_{\infty}, \|c\|_{\infty}\}/\eta_k \geq 500$ or $\|C(U^t) - \eta_k \log  \pi^k \|_{\var}/\eta_k$$ \geq 900$, we perform the Sinkhorn iteration \eqref{equ:SK}  and set $\theta = 10$ in \eqref{equ:theta:choice:numerical}; otherwise, we perform the Sinkhorn iteration  \eqref{equ:SK:2} and set $\theta =0.1$ in \eqref{equ:theta:choice:numerical}.

\begin{table}[!htbp]
\setlength{\tabcolsep}{3.5pt}
\centering
\scriptsize
 \caption{The average  results  for \cref{data:hypercube}.  In this table, ``a'' and ``b'' stand for \ReALM-$(0.02,0)$ and \ReALM-$(0.055,0.9)$, respectively.}\label{table:hypercube:ALM}
\vspace{-2mm}
\begin{tabular}{@{}cccrrrrrrrr@{}}
\toprule
&\multicolumn{2}{c}{$\mathrm{PRW_p}$}  &  \multicolumn{2}{c}{nGrad} &  \multicolumn{2}{c}{$\mbox{nSk}_{\exp}$/$\mbox{nSk}_{\log}$} & \multicolumn{2}{c}{time} & \multicolumn{2}{c}{iter} \\ 
 \cmidrule(l){2-3}   \cmidrule(l){4-5}    \cmidrule(l){6-7}   \cmidrule(l){8-9}  \cmidrule(l){10-11}  
$n$/$d$   & \multicolumn{1}{c}{a}&\multicolumn{1}{c}{b}&
  \multicolumn{1}{c}{a}&\multicolumn{1}{c}{b}& 
   \multicolumn{1}{c}{a}&\multicolumn{1}{c}{b}&
    \multicolumn{1}{c}{a}&\multicolumn{1}{c}{b}& 
     \multicolumn{1}{c}{a}&\multicolumn{1}{c}{b} \\
\midrule 
100/20 &8.3430 & 8.3430 &182 & 138 &3208/8342 & 13258/0 &1.1 & 0.1 &0.0/7.0 & 8.0/14.0  \\
100/100 &9.1732 & 9.1734 &287 & 241 &3837/3354 & 11369/0 &0.5 & 0.1 &0.0/7.0 & 8.0/14.0  \\
100/250 &10.8956 & 10.8969 &467 & 348 &3200/8858 & 10550/279 &1.3 & 0.2 &0.0/7.0 & 8.0/14.1  \\
100/500 &13.3111 & 13.3152 &680 & 469 &2192/3782 & 9844/213 &0.8 & 0.2 &0.0/7.0 & 8.0/14.0  \\
\midrule
100/50 &8.5999 & 8.6000 &321 & 189 &3945/13096 & 11700/0 &1.8 & 0.1 &0.0/7.0 & 8.0/14.0  \\
250/50 &8.2334 & 8.2334 &150 & 156 &2480/4168 & 10828/0 &2.2 & 0.3 &0.0/7.0 & 7.9/13.9  \\
500/50 &8.1299 & 8.1299 &122 & 141 &1944/4508 & 11578/0 &7.2 & 0.8 &0.0/7.0 & 8.0/14.0  \\
1000/50 &8.0710 & 8.0709 &111 & 130 &1716/5307 & 12344/0 &30.8 & 2.3 &0.0/7.0 & 8.0/14.0  \\
\midrule
20/20 &9.3098 & 9.3104 &591 & 221 &4016/23861 & 11906/0 &0.5 & 0.0 &0.0/7.1 & 8.0/14.0  \\
50/50 &9.3627 & 9.3629 &386 & 248 &4656/9416 & 11819/0 &0.5 & 0.1 &0.0/7.0 & 7.9/13.9  \\
250/250 &9.1744 & 9.1748 &282 & 257 &2942/4136 & 11371/248 &2.2 & 0.5 &0.0/7.0 & 8.0/14.0  \\
500/500 &9.1150 & 9.1154 &258 & 262 &2391/4884 & 9927/873 &7.5 & 2.2 &0.0/7.0 & 7.6/13.6  \\
\midrule
100/10 &8.1620 & 8.1619 &405 & 173 &4635/33570 & 65772/0 &4.6 & 0.3 &0.0/7.0 & 8.0/14.0  \\
200/20 &8.1331 & 8.1331 &180 & 128 &2861/10259 & 15007/0 &3.8 & 0.3 &0.0/7.0 & 8.0/14.0  \\
1000/100 &8.1187 & 8.1187 &124 & 145 &2096/5157 & 12153/141 &30.5 & 3.1 &0.0/7.0 & 8.0/14.0  \\
2500/250 &8.1164 & 8.1164 &117 & 146 &2416/5499 & 11482/624 &373.8 & 103.3 &0.0/7.0 & 7.9/13.9  \\
\midrule
\cmidrule(l){2-11}  
AVG &9.0156 & \textBF{9.0160} &291 & 212 &3033/9262 & 15057/149 &29.3 & 7.1 &0.0/7.0 & 8.0/14.0  \\
\bottomrule
\end{tabular}
\end{table}
The results for \cref{data:hypercube} are presented in \cref{table:hypercube:ALM}. In this table and the subsequent tables,  
the terms ``$\mbox{nSk}_{\log}$''  and  ``$\mbox{nSk}_{\exp}$''   mean the  total  numbers of Sinkhorn iterations \eqref{equ:SK} and  \eqref{equ:SK:2}, respectively, the pair ``$K_1$/$K$'' in the column ``iter''  means that the corresponding algorithm stops at the $K$-iteration and updates the multiplier matrix  $K_1$ times.  For each $(n,d)$ pair,  we randomly generate  10 instances, each equipped with 5 randomly generated initial points.  We conside \ReALM-$(0.055,0.9)$ and \ReALM-$(0.02,0)$ both with $\eta_1 = 1$ and $\gamma_{\eta} = 0.5$.  Note that the latter admits a smaller $\eta_{\min}$ and does not update the multiplier matrix. From \cref{table:hypercube:ALM}, we can observe that \ReALM-$(0.055,0.9)$ can not only return better solutions than \ReALM-$(0.02,0.9)$ but also is about 4x faster.  This shows that updating the multiplier matrix does help. In fact, on average \ReALM-$(0.055,0.9)$ updates the multiplier matrix 8 times in 14 total iterations.   The reasons why \ReALM~with updating the multiplier matrix outperforms \ReALM~without updating the multiplier matrix in terms of solution quality and speed are as follows. First, updating the multiplier matrix in \ReALM~can keep the solution quality even using a larger $\eta$.  Second,  solving the subproblem with a larger $\eta$ is always easier,  which enables that \ReALM-$(0.055,0.9)$ computes less   \rev{$\grad\Lcal(\x)$} and performs less Sinkhorn  iterations \eqref{equ:SK}  which involves computing the log-sum-exp function $\log \sum_{i} \exp(x_i/\eta) = x_{\max}/\eta + \log \sum_i \exp (x_i/\eta - x_{\max}/\eta)$ for small $\eta$.

\begin{table}[!t]
\setlength{\tabcolsep}{3.5pt}
\centering
\scriptsize
 \caption{The average  results  for \cref{data:shakespeare,data:mnist}. For \cref{data:shakespeare},  $\Gamma = 1$, ``a'' and ``b'' stand for \ReALM-$(0.0035,0)$ and \ReALM-$(0.07,0.9)$, respectively. For \cref{data:mnist}, $\Gamma = 10^{-3}$, ``a'' and ``b'' stand for \ReALM-$(1,0)$ and \ReALM-$(3,0.9)$, respectively.   }\label{table:shakespeare:mnist:ALM:short}
\vspace{-2mm}
\begin{tabular}{@{}cccrrrrrrrr@{}}
\toprule
&\multicolumn{2}{c}{$\Gamma \times \mathrm{PRW_p}$}  &  \multicolumn{2}{c}{nGrad} &  \multicolumn{2}{c}{$\mbox{nSk}_{\exp}$/$\mbox{nSk}_{\log}$} & \multicolumn{2}{c}{time} & \multicolumn{2}{c}{iter} \\ 
 \cmidrule(l){2-3}   \cmidrule(l){4-5}    \cmidrule(l){6-7}   \cmidrule(l){8-9}  \cmidrule(l){10-11}  
data   & \multicolumn{1}{c}{a}&\multicolumn{1}{c}{b}&
  \multicolumn{1}{c}{a}&\multicolumn{1}{c}{b}& 
   \multicolumn{1}{c}{a}&\multicolumn{1}{c}{b}&
    \multicolumn{1}{c}{a}&\multicolumn{1}{c}{b}& 
    \multicolumn{1}{c}{a}&\multicolumn{1}{c}{b}
     \\
     \midrule 
 H5/JC &0.0927 & \textBF{0.1099} &1081 & 675 &23809/8267 & 9786/0 &84.1 & 8.6 &0.0/8.0  & 8.0/15.0\\
H/MV &0.0638 & \textBF{0.0642} &1116 & 727 &3864/13382 & 30997/0 &175.7 & 14.9 &0.0/8.0  & 8.0/15.0\\
H/RJ &0.2171 & \textBF{0.2261} &547 & 788 &3091/4062 & 11747/0 &63.3 & 15.6 &0.0/8.0  & 8.0/15.0\\
JC/MV &\textBF{0.0627} & 0.0623 &1324 & 711 &12875/9727 & 17073/0 &55.3 & 5.9 &0.0/8.0  & 8.0/15.0\\
JC/O &0.0422 & \textBF{0.0428} &1650 & 1100 &4997/26612 & 35031/4500 &204.7 & 46.7 &0.0/8.0  & 8.0/15.0\\
MV/O &\textBF{0.0418} & 0.0366 &890 & 822 &6101/5289 & 31651/0 &57.0 & 13.3 &0.0/8.0  & 8.0/15.0\\
\cmidrule(l){2-11}  
AVG &0.1134 & \textBF{0.1146} &863 & 787 &7824/7364 & 20261/300 &78.0 & 15.0 &0.0/8.0 & 8.0/15.0  \\
     \midrule 
D0/D4 &1.2210 & \textBF{1.2316} &428 & 632 &2169/3599 & 6564/0 &20.8 & 4.8 &0.0/5.0  & 8.0/13.0\\
D2/D5 &1.0771 & \textBF{1.0861} &574 & 1083 &707/6040 & 7035/1642 &33.0 & 15.0 &0.0/5.0  & 8.0/13.0\\
D2/D7 &0.6955 & \textBF{0.7012} &272 & 538 &1890/2284 & 7972/0 &21.0 & 6.0 &0.0/5.0  & 7.0/12.0\\
D2/D9 &1.0570 & \textBF{1.0697} &502 & 721 &3027/4541 & 4172/0 &28.0 & 5.9 &0.0/5.0  & 8.0/13.0\\
\cmidrule(l){2-11}  
AVG &0.8861 & \textBF{0.8870} &386 & 783 &1504/2378 & 9551/36 &16.1 & 6.7 &0.0/5.0 & 6.7/11.7  \\
\bottomrule
\end{tabular}
\end{table}

The results over 20 runs on the real datasets \cref{data:shakespeare,data:mnist} are reported in 
\cref{table:shakespeare:mnist:ALM:short}.  We choose $\eta_1 = 20$ for \cref{data:shakespeare}  and \rev{$\eta_1 = 200$} for  \cref{data:mnist} and set $\gamma_{\eta} = 0.25$ for both datasets.   To save space, we only report instances where one method can return the 
value ``$\mathrm{PRW_p}$'' larger than 1.005 times of the smaller one of the two $\mathrm{PRW_p}$ values returned by the two methods. The better ``$\mathrm{PRW_p}$'' is marked in  \textBF{bold}.  Besides, the average performance over all instances  (15 instances in total for  \cref{data:shakespeare} and 45 instances in total for  \cref{data:mnist}) for each dataset is also kept in the ``AVG'' line.  From this table, we can see that updating the multiplier matrix also helps for the two real datasets. Compared with \ReALM~without updating the multiplier, \ReALM~with updating the multiplier can not only return better solutions but is about 5.2x faster for \cref{data:shakespeare} and is about
2.4x  faster for \cref{data:mnist}.

\section{Concluding remarks}\label{sec:conclusions}
In this paper, we considered the computation of the PRW distance arising from machine learning applications. By reformulating this problem as an optimization problem over the Cartesian product of the Stiefel manifold and the Euclidean space with additional nonlinear inequality constraints, we proposed a \ReALM~method.  The convergence of \ReALM~was also established.  To solve the subproblem in the \ReALM~method, we developed a framework of inexact Riemannian gradient descent methods. Also, we provided a practical \iRBBS~method with convergence and complexity guarantees, wherein the Riemannian BB stepsize and Sinkhorn iterations are employed. 
 Our numerical results showed that, compared with the state-of-the-art methods, our proposed \ReALM~and \iRBBS~methods both have advantages in solution quality and speed. Moreover, our proposed \ReALM~method can be also extended to solve more general Riemannian optimization with additional inequality constraints.

\bibliographystyle{siamplain}
\bibliography{bib-PRW}

\begin{thebibliography}{10}

\bibitem{absil2009optimization}
{\sc P.-A. Absil, R.~Mahony, and R.~Sepulchre}, {\em Optimization algorithms on
  matrix manifolds}, in Optimization Algorithms on Matrix Manifolds, Princeton
  University Press, 2009.

\bibitem{altschuler2017near}
{\sc J.~Altschuler, J.~Niles-Weed, and P.~Rigollet}, {\em Near-linear time
  approximation algorithms for optimal transport via {S}inkhorn iteration},
  Adv. Neural Inf. Process. Syst., 30 (2017).

\bibitem{andreani2008augmented}
{\sc R.~Andreani, E.~G. Birgin, J.~M. Mart{\'\i}nez, and M.~L. Schuverdt}, {\em
  On augmented {L}agrangian methods with general lower-level constraints}, SIAM
  J. Optim., 18 (2008), pp.~1286--1309.

\bibitem{barzilai1988two}
{\sc J.~Barzilai and J.~M. Borwein}, {\em Two-point step size gradient
  methods}, IMA J. Numer. Anal., 8 (1988), pp.~141--148.

\bibitem{berahas2021global}
{\sc A.~S. Berahas, L.~Cao, and K.~Scheinberg}, {\em Global convergence rate
  analysis of a generic line search algorithm with noise}, SIAM J. Optim., 31
  (2021), pp.~1489--1518.

\bibitem{bertsekas2014constrained}
{\sc D.~P. Bertsekas}, {\em Constrained optimization and {L}agrange multiplier
  methods}, Academic Press, 2014.

\bibitem{bonettini2011inexact}
{\sc S.~Bonettini}, {\em Inexact block coordinate descent methods with
  application to non-negative matrix factorization}, IMA J. Numer. Anal., 31
  (2011), pp.~1431--1452.

\bibitem{boumal2020introduction}
{\sc N.~Boumal}, {\em An introduction to optimization on smooth manifolds},
  Available online, May, 3 (2020).

\bibitem{boumal2019global}
{\sc N.~Boumal, P.-A. Absil, and C.~Cartis}, {\em Global rates of convergence
  for nonconvex optimization on manifolds}, IMA J. Numer. Anal., 39 (2019),
  pp.~1--33.

\bibitem{carter1991global}
{\sc R.~G. Carter}, {\em On the global convergence of trust region algorithms
  using inexact gradient information}, SIAM J. Numer. Anal., 28 (1991),
  pp.~251--265.

\bibitem{chambolle2022accelerated}
{\sc A.~Chambolle and J.~P. Contreras}, {\em Accelerated {Bregman} primal-dual
  methods applied to optimal transport and {W}asserstein barycenter problems},
  arXiv:2203.00802,  (2022).

\bibitem{cuturi2013sinkhorn}
{\sc M.~Cuturi}, {\em Sinkhorn distances: Lightspeed computation of optimal
  transport}, Adv. Neural Inf. Process. Syst., 26 (2013).

\bibitem{devolder2014first}
{\sc O.~Devolder, F.~Glineur, and Y.~Nesterov}, {\em First-order methods of
  smooth convex optimization with inexact oracle}, Math. Program., 146 (2014),
  pp.~37--75.

\bibitem{doljansky1999interior}
{\sc M.~Doljansky and M.~Teboulle}, {\em An interior proximal algorithm and the
  exponential multiplier method for semidefinite programming}, SIAM J. Optim.,
  9 (1998), pp.~1--13.

\bibitem{dvurechensky2018computational}
{\sc P.~Dvurechensky, A.~Gasnikov, and A.~Kroshnin}, {\em Computational optimal
  transport: Complexity by accelerated gradient descent is better than by
  {S}inkhorn’s algorithm}, in ICML, PMLR, 2018, pp.~1367--1376.

\bibitem{echebest2016convergence}
{\sc N.~Echebest, M.~D. S{\'a}nchez, and M.~L. Schuverdt}, {\em Convergence
  results of an augmented {L}agrangian method using the exponential penalty
  function}, J. Optim. Theory Appl., 168 (2016), pp.~92--108.

\bibitem{edelman1998geometry}
{\sc A.~Edelman, T.~A. Arias, and S.~T. Smith}, {\em The geometry of algorithms
  with orthogonality constraints}, SIAM J. Matrix Anal. Appl., 20 (1998),
  pp.~303--353.

\bibitem{fournier2015rate}
{\sc N.~Fournier and A.~Guillin}, {\em On the rate of convergence in
  {W}asserstein distance of the empirical measure}, Probab. Theory Relat.
  Fields, 162 (2015), pp.~707--738.

\bibitem{gao2018new}
{\sc B.~Gao, X.~Liu, X.~Chen, and Y.-x. Yuan}, {\em A new first-order
  algorithmic framework for optimization problems with orthogonality
  constraints}, SIAM J. Optim., 28 (2018), pp.~302--332.

\bibitem{gur2022convergent}
{\sc E.~Gur, S.~Sabach, and S.~Shtern}, {\em Convergent nested alternating
  minimization algorithms for nonconvex optimization problems}, Math. Oper.
  Res.,  (2022).

\bibitem{hu2007inexact}
{\sc Y.-Q. Hu and Y.-H. Dai}, {\em Inexact {B}arzilai-{B}orwein method for
  saddle point problems}, Numer. Linear Algebr., 14 (2007), pp.~299--317.

\bibitem{huang2021riemannian}
{\sc M.~Huang, S.~Ma, and L.~Lai}, {\em A {R}iemannian block coordinate descent
  method for computing the projection robust {W}asserstein distance}, in ICML,
  PMLR, 2021, pp.~4446--4455.

\bibitem{huang2021equipping}
{\sc Y.-K. Huang, Y.-H. Dai, and X.-W. Liu}, {\em Equipping the
  {Barzilai--Borwein} method with the two dimensional quadratic termination
  property}, SIAM J. Optim., 31 (2021), pp.~3068--3096.

\bibitem{iannazzo2018riemannian}
{\sc B.~Iannazzo and M.~Porcelli}, {\em The {Riemannian Barzilai--Borwein}
  method with nonmonotone line search and the matrix geometric mean
  computation}, IMA J. Numer. Anal., 38 (2018), pp.~495--517.

\bibitem{jiang2015framework}
{\sc B.~Jiang and Y.-H. Dai}, {\em A framework of constraint preserving update
  schemes for optimization on {S}tiefel manifold}, Math. Program., 153 (2015),
  pp.~535--575.

\bibitem{kullback1967lower}
{\sc S.~Kullback}, {\em A lower bound for discrimination information in terms
  of variation (corresp.)}, IEEE Trans. Inf. Theory, 13 (1967), pp.~126--127.

\bibitem{li1991aggregate}
{\sc X.~Li}, {\em An aggregate function method for nonlinear programming}, Sci.
  China Ser. A., 34 (1991), pp.~1467--1473.

\bibitem{li2021weakly}
{\sc X.~Li, S.~Chen, Z.~Deng, Q.~Qu, Z.~Zhu, and A.~M.-C. So}, {\em Weakly
  convex optimization over {S}tiefel manifold using {R}iemannian
  subgradient-type methods}, SIAM J. Optim., 31 (2021), pp.~1605--1634.

\bibitem{lin2020projection}
{\sc T.~Lin, C.~Fan, N.~Ho, M.~Cuturi, and M.~Jordan}, {\em Projection robust
  {W}asserstein distance and {R}iemannian optimization}, Adv. Neural Inf.
  Process. Syst., 33 (2020), pp.~9383--9397.

\bibitem{lin2022efficiency}
{\sc T.~Lin, N.~Ho, and M.~I. Jordan}, {\em On the efficiency of entropic
  regularized algorithms for optimal transport}, J. Mach. Learn. Res., 23
  (2022), pp.~1--42.

\bibitem{lin2021projection}
{\sc T.~Lin, Z.~Zheng, E.~Chen, M.~Cuturi, and M.~I. Jordan}, {\em On
  projection robust optimal transport: Sample complexity and model
  misspecification}, in AISTATS, PMLR, 2021, pp.~262--270.

\bibitem{liu2019quadratic}
{\sc H.~Liu, A.~M.-C. So, and W.~Wu}, {\em Quadratic optimization with
  orthogonality constraint: {E}xplicit {{\L}}ojasiewicz exponent and linear
  convergence of retraction-based line-search and stochastic variance-reduced
  gradient methods}, Math. Program., 178 (2019), pp.~215--262.

\bibitem{lu2012augmented}
{\sc Z.~Lu and Y.~Zhang}, {\em An augmented {L}agrangian approach for sparse
  principal component analysis}, Math. Program., 135 (2012), pp.~149--193.

\bibitem{niles2019estimation}
{\sc J.~Niles-Weed and P.~Rigollet}, {\em Estimation of {W}asserstein distances
  in the spiked transport model}, arXiv:1909.07513,  (2019).

\bibitem{paty2019subspace}
{\sc F.-P. Paty and M.~Cuturi}, {\em Subspace robust {W}asserstein distances},
  in ICML, PMLR, 2019, pp.~5072--5081.

\bibitem{peyre2019computational}
{\sc G.~Peyr{\'e} and M.~Cuturi}, {\em Computational optimal transport: With
  applications to data science}, Foundations and Trends{\textregistered} in
  Machine Learning, 11 (2019), pp.~355--607.

\bibitem{sachs2011nonmonotone}
{\sc E.~W. Sachs and S.~M. Sachs}, {\em Nonmonotone line searches for
  optimization algorithms}, Control Cybern., 40 (2011), pp.~1059--1075.

\bibitem{shi2022noise}
{\sc H.-J.~M. Shi, Y.~Xie, R.~Byrd, and J.~Nocedal}, {\em A noise-tolerant
  {quasi-Newton} algorithm for unconstrained optimization}, SIAM J. Optim., 32
  (2022), pp.~29--55.

\bibitem{torrealba2020exponential}
{\sc E.~Torrealba, L.~Matioli, M.~Nasri, and R.~Castillo}, {\em Exponential
  augmented {L}agrangian methods for equilibrium problems}, Optim. Method
  Softw.,  (2020), pp.~1--25.

\bibitem{tseng1993convergence}
{\sc P.~Tseng and D.~P. Bertsekas}, {\em On the convergence of the exponential
  multiplier method for convex programming}, Math. Program., 60 (1993),
  pp.~1--19.

\bibitem{vial1983strong}
{\sc J.-P. Vial}, {\em Strong and weak convexity of sets and functions}, Math.
  Oper. Res., 8 (1983), pp.~231--259.

\bibitem{weed2019sharp}
{\sc J.~Weed and F.~Bach}, {\em Sharp asymptotic and finite-sample rates of
  convergence of empirical measures in {W}asserstein distance}, Bernoulli, 25
  (2019), pp.~2620--2648.

\bibitem{wen2013feasible}
{\sc Z.~Wen and W.~Yin}, {\em A feasible method for optimization with
  orthogonality constraints}, Math. Program., 142 (2013), pp.~397--434.

\bibitem{xie2020fast}
{\sc Y.~Xie, X.~Wang, R.~Wang, and H.~Zha}, {\em A fast proximal point method
  for computing exact {W}asserstein distance}, in Uncertainty in Artificial
  Intelligence, PMLR, 2020, pp.~433--453.

\bibitem{yang2022bregman}
{\sc L.~Yang and K.-C. Toh}, {\em Bregman proximal point algorithm revisited: A
  new inexact version and its variant}, SIAM J. Optim., 32 (2022),
  pp.~1523--1554.

\bibitem{yang2014optimality}
{\sc W.~H. Yang, L.-H. Zhang, and R.~Song}, {\em Optimality conditions for the
  nonlinear programming problems on {R}iemannian manifolds}, Pac. J. Optim., 10
  (2014), pp.~415--434.

\bibitem{zhang2004nonmonotone}
{\sc H.~Zhang and W.~W. Hager}, {\em A nonmonotone line search technique and
  its application to unconstrained optimization}, SIAM J. Optim., 14 (2004),
  pp.~1043--1056.

\end{thebibliography}

\end{document}